%% file: FoonKhosh.tex
\documentclass[11pt]{memo-l}

\usepackage{FoonKhosh,setspace,amssymb,bm}
\usepackage{pdfsync}

\newtheorem{theorem}{Theorem}[chapter]
\newtheorem{lemma}[theorem]{Lemma}
\newtheorem{proposition}[theorem]{Proposition}
\newtheorem{corollary}[theorem]{Corollary}

\theoremstyle{definition}
\newtheorem{definition}[theorem]{Definition}
\newtheorem{example}[theorem]{Example}

\newtheorem{convention}[theorem]{Convention}

\theoremstyle{remark}
\newtheorem{remark}[theorem]{Remark}
\newtheorem{condition}[theorem]{Condition}

\numberwithin{section}{chapter}
\numberwithin{equation}{chapter}

\makeindex

\begin{document}
\onehalfspacing
\frontmatter

\title{On the Stochastic Heat Equation with Spatially-Colored Random  Forcing}


\author{Mohammud Foondun}
\address{School of Mathematics, Loughborough University, Leicestershire, LE11 3TU }
\email{m.i.foondun@lboro.ac.uk}
\thanks{Research supported in part by the National Science Foundation
	grant DMS-0706728.}

\author{Davar Khoshnevisan}
\address{University of Utah, Department of Mathematics, 155 South 1400 East JWB 233,
	Salt Lake City, UT 84112--0090}
\email{davar@math.utah.edu}

\date{}
\subjclass[2000]{Primary. 60H15; Secondary. 35R60.}
\keywords{The stochastic heat equation, spatially-colored homogeneous noise,
	L\'evy processes.}

\maketitle

\begin{abstract}
We consider the stochastic heat equation of the following form
\begin{equation*}
	\frac{\partial}{\partial t}u_t(x) = (\sL u_t)(x) +b(u_t(x)) +
	\sigma(u_t(x))\dot{F}_t(x)\quad \text{for }t>0,\ x\in \R^d,
\end{equation*}
where $\sL$ is the generator of a L\'evy process and $\dot{F}$ is a spatially-colored,
temporally white, gaussian noise.  We will be concerned mainly with the long-term 
behavior of the mild solution to this stochastic PDE.

For the most part, we work under the assumptions that the initial data
$u_0$ is a bounded and measurable function and $\sigma$ is 
nonconstant and Lipschitz continuous.
In this case, we find conditions under which the preceding stochastic PDE admits a unique 
solution which is also \emph{weakly intermittent}. In addition, we study the same
equation in the case that $\mathcal{L}u$ is replaced by its
massive/dispersive analogue $\mathcal{L}u-\lambda u$ where $\lambda\in\R$.
And we describe accurately the effect of the parameter $\lambda$ on
the intermittence of the solution in the case that $\sigma(u)$ is proportional
to $u$ [the ``parabolic Anderson model''].

Furthermore, we extend our analysis to the case that the initial data $u_0$ is a 
measure rather than a function.  As it turns out, the stochastic PDE in question
does not have a mild solution in this case.  
We circumvent this problem by introducing a new concept of a solution
that we call a \emph{temperate solution},
and proceed to investigate the existence and uniqueness of a temperate solution.
We are able to also give partial insight into the long-time behavior of the temperate
solution when it exists and is unique.

Finally, we look at the linearized version of our stochastic PDE, that is 
the case when $\sigma$ is identically equal to one [any other constant
works also].
In this case, we study not only the existence and uniqueness of a solution, but 
also the regularity of the solution when it exists and is unique.

\end{abstract}


\setcounter{page}{4}

\tableofcontents

\mainmatter

\include{FoonKhosh-Intro}			
\include{FoonKhosh-LP}				
\include{FoonKhosh-PosDefPot}		
\include{FoonKhosh-Lin}				
\include{FoonKhosh-NL1}				
\include{FoonKhosh-NL2}				
\include{FoonKhosh-Discrete}			
\appendix

\backmatter
\bibliographystyle{amsalpha}%
\bibliography{FoonKhosh}
\printindex

\end{document}

%% file: FoonKhosh-Intro.tex
\chapter{Introduction and Statements of Main Results}
The principle aim of this paper is to describe the asymptotic large-time 
behavior of the mild solution $\bm{u}:=\{u_t(x)\}_{t\ge 0,x\in\R^d}$ of
the stochastic heat equation, 
\begin{equation}\label{heat}
	\frac{\partial}{\partial t}u_t(x) = (\sL u_t)(x) +b(u_t(x)) +
	\sigma(u_t(x))\dot{F}_t(x),
\end{equation}
where $t>0$ and $x\in\R^d$, and the preceding 
stochastic PDE can be understood in the sense of Walsh \cite{Walsh}.


For the most part, we consider the case that
the initial data $u_0$ is a nonrandom, as
well as bounded and measurable, function. 
But we will also consider the physically-interesting 
case that $u_0$ is a nonrandom finite Borel
measure on $\R^d$.  The latter case will be the subject of Chapter 6.

Throughout we consider only functions $\sigma,b:\R\to\R$
that are nonrandom and Lipschitz continuous. 
Also, we let $\sL$ be the
$L^2$-generator of a $d$-dimensional L\'evy process $X:=\{X_t\}_{t\ge 0}$, 
and assume that $X$ has transition functions.

In the above discussion, we have used the standard 
notation of probability theory: Namely,
$g_t$ denotes the evaluation of a  [random or nonrandom] function
$g$ at time $t$, and \emph{never} the time derivative of $g$.  This notation will be used throughout the rest of the paper.

As regards the forcing term $\dot{F}$ in \eqref{heat}, we assume that
$\dot{F}$ is a generalized Gaussian random field \cite[Chapter 2, \S2.4]{GV} whose
covariance kernel is $\delta_0(s-t) f(x-y)$, where the ``correlation
function'' $f$ is a 
nonnegative definite, symmetric, and
nonnegative function that is not identically zero.\footnote{The symbol ``$f$'' is reserved
for this correlation function here and throughout. We \emph{never}
refer to any other function as $f$.}  Alternatively, one can use the following
\begin{equation}
	\dot{F}_t(x) := \frac{\partial^{d+1}}{\partial t\partial x_1
	\cdots\partial x_d}F(t\,,x),
\end{equation}
in the sense of generalized random fields,
where $F$ is a centered generalized Gaussian random
field with covariance kernel
\begin{equation}\label{eq:Cov}\begin{split}
	&\textrm{Cov}\left( \int_{\R_+\times\R^d}\phi\,\d F
		~,\, \int_{\R_+\times\R^d}\zeta\,\d F\right)\\
	&\hskip1.4in= \int_0^\infty
		\d s\int_{\R^d} \d x\int_{\R^d}\d y\ \phi_s(x)\zeta_s(y)
		f(x-y),
\end{split}\end{equation}
where $\int\phi\,\d F$ and $\int\zeta\,\d F$ are Wiener integrals
of $\R_+\times\R^d\ni(s\,,x)\mapsto \phi_s(x)$ and $
\R_+\times\R^d\ni(s\,,x)\mapsto\zeta_s(x)$ with respect to $F$,
and $\phi$ and $\zeta$ are nonnegative measurable functions
for which the right-most multiple integral in \eqref{eq:Cov} is 
absolutely convergent.

According to the Bochner--Schwartz theorem \cite[Theorem 3,
p.\ 157]{GV}:
\begin{enumerate}
 \item[(a)] The Fourier transform $\hat{f}$
of $f$ is  a [nonnegative Borel] tempered measure on $\R^d$; and
\item[(b)] Conversely, every tempered measure $\hat f$ on $\R^d$ is the Fourier transform
of one such correlation function $f$.
\end{enumerate}
The measure $\hat f$ is known as the ``spectral measure'' of the noise $F$.
Throughout, we assume without further mention that $F$
``has a spectral density.'' That is,
\begin{equation}
	\text{$\hat f$ is a measurable function}.
\end{equation}
This  implies that $\hat f$ is locally integrable on $\R^d$
as well. Strictly speaking,
these conditions are not always needed in our work, but we
assume them for the sake of simplicity.

By enlarging the underlying probability space, if need be, we introduce
an independent copy $X^*:=\{X_t^*\}_{t\ge 0}$ of the dual process $-X$. 
We can then use $X^*$ to define a symmetric L\'evy process $\bar X:=\{\bar X_t\}_{t\ge 0}$
on $\R^d$ via the assignment
\begin{equation}
	\bar X_t:= X_t+X_t^*\qquad\text{for all $t\ge 0$}.
\end{equation}
Motivated by the works of Kardar, Parisi, and Zhang \cite{KPZ}
and Kardar \cite{Kardar}, we may refer to $\bar X$ as the \emph{replica
L\'evy process} corresponding to $X$ and will therefore call the resolvent $\{\bar R_\alpha\}_{\alpha>0}$ of $\bar{X}$, the 
\emph{replica resolvent}.\footnote{These quantities are
defined in more detail in Chapter \ref{ch:Levy}.} We will consider the condition
that the correlation function $f$ has finite $\alpha$-potential at zero 
for all $\alpha>0$. That is, we consider the following:
\begin{condition}\label{cond:1}
	$(\bar{R}_\alpha f)(0)<\infty 	\quad\text{for all $\alpha>0$}$
\end{condition}
The above condition will imply an existence
and uniqueness result for the stochastic heat equation \eqref{heat}.
Moreover, our proof of existence and uniqueness is
closely linked to the large-time behavior of the solution
itself [via \emph{a priori} estimates]. We describe these results next. But first, let us define two important quantities:
The first denotes the
\emph{upper $L^p(\P)$-Liapounov exponent} of the solution
$\bm{u}:=\{u_t(x)\}_{t>0,x\in\R^d}$ to \eqref{heat} at the
spatial point $x\in\R^d$:
\index{gammabar000x@$\overline\gamma_x(p)$, the upper $L^p(\P)$-Liapounov exponent of the solution at $x$}%
\index{gammabar*@$\overline\gamma_*(p)$, the maximum upper $L^p(\P)$-Liapounov exponent of the solution.}%
\begin{equation}\label{def:gamma}
	\overline\gamma_x(p) :=\limsup_{t\to\infty} \frac 1t
	\ln\E\left( \left| u_t(x)\right|^p\right);
\end{equation}
and the second the \emph{upper maximum $L^p(\P)$-Liapounov exponent}:
\begin{equation}\label{def:gamma*}
	\overline\gamma_*(p):=\limsup_{t\to\infty} \frac 1t\sup_{x\in\R^d}
	\ln\E\left( \left| u_t(x)\right|^p\right).
\end{equation}
The above two quantities are variants of the well known \emph{Liapounov exponent}.

We are now ready to state the first main contribution of this paper.

\begin{theorem}\label{th:existence}
	Assume that Condition \ref{cond:1} holds,
	and suppose $u_0:\R^d\to\R$ is bounded and measurable.
	Then, \eqref{heat} has an a.s.-unique mild solution which satisfies
	the following: For all even integers $p\ge 2$,
	\begin{equation}\label{eq:exist:nonlinear}
		\overline\gamma_*(p)
		\le \inf\left\{\beta>0:\ Q(p\,,\beta)<1\right\},
	\end{equation}
	where
	\begin{equation}\label{def:Q}
		Q(p\,,\beta) := \frac{p\lip_b}{\beta}+
		z_p\lip_\sigma\sqrt{( \bar R_{2\beta/p} f)(0)},
	\end{equation}
	and $ z_p$ denotes the largest 
	positive zero of the Hermite polynomial
	$\text{\sl He}_p$.
\end{theorem}

Let us make two remarks before we continue
with our presentation of the main results of this paper. The first
one is consequence of the above result.

\begin{remark}
	It is possible to deduce from Condition \ref{cond:1} and the monotone
	convergence theorem that $\lim_{\alpha\to\infty}(\bar R_\alpha f)(0)=0$.
	This, in turn, implies that
	\begin{equation}
		\lim_{\beta\to\infty} Q(p\,,\beta)=0.
	\end{equation}
	Consequently, Theorem \ref{th:existence} implies among other things
	that $\overline\gamma_*(p)<\infty$ for all $p\in(0\,,\infty)$.
	\qed
\end{remark}

\begin{remark}[Borrowed from \protect{\cite[Remark 2.2]{FK}}]
	It might help to recall that
	\begin{equation}
		\text{\sl He}_k(x)=\frac{1}{2^{k/2}}
		{\sl H}_k\left( \frac{x}{\sqrt 2} \right)
		\quad\text{for all integers $k\ge 0$ and $x\in\R$,}
	\end{equation}
	where $\{H_k\}_{k=0}^\infty$ is defined uniquely via the following:
	\begin{equation}
		\e^{-2xt-t^2}=\sum_{k=0}^\infty \frac{t^k}{k!}\,{\sl H}_k(x)
		\quad\text{for all $t>0$ and $x\in\R$.}
	\end{equation}
	It is not hard to verify that
	\begin{equation}
		z_2 = 1
		\quad\text{and}\quad
		z_4=\sqrt{3+\sqrt{6}}\ \approx 2.334.
	\end{equation}
	This is valid simply because
	$\text{\sl He}_2(x) = x^2-1$ and
	$\text{\sl He}_4(x) = x^4-6x^2+3$. In addition,
	\begin{equation}
		z_p\sim 2\sqrt p
		\quad\text{as $p\to\infty$, and}\quad
		\sup_{p\ge 1} \left( \frac{z_p}{\sqrt p}\right)=2;
	\end{equation}
	see Carlen and Kree \cite[Appendix]{CarlenKree}.\qed
\end{remark}

Next we put Theorem \ref{th:existence} in the context of
the existing literature on the stochastic heat equation. With this
in mind, define for all $\beta\ge 0$,
\begin{equation}\label{eq:Upsilon}
	\Upsilon(\beta):=
	\frac{1}{(2\pi)^d}\int \frac{\hat f(\xi)}{\beta+2\Re\Psi(\xi)}\,\d\xi,
\end{equation}
where $\Psi$ is the characteristic exponent of the L\'evy process $X$.
Dalang \cite{Dalang} has established a very general 
result which guarantees the 
existence and uniqueness of solutions to large families of
SPDEs. If we apply Dalang's
result to the present parabolic problem \eqref{heat}, 
then we find the following: If $u_0$ is a constant, then the condition
\begin{equation}\label{cond:Dalang}
	\Upsilon(1)<\infty
\end{equation}
insures the existence and uniqueness of a [mild] solution to \eqref{heat}.
Moreover, Dalang's result shows that 
\eqref{cond:Dalang} is necessary and sufficient for existence
and uniqueness in the case
that \eqref{heat} is a linear SPDE; that is, when $\sigma(u)=1$ and $b(u)=0$.
For closely-related results [that also include hyperbolic equations] see
Carmona and Molchanov \cite{CarmonaMolchanov:94},
Conus and Dalang \cite{ConusDalang},
Dalang and Frangos \cite{DalangFrangos},
Dalang and Mueller \cite{DalangMueller},
Dalang, Mueller, and Tribe \cite{DalangMuellerTribe},
Dalang and Sanz-Sol\'e \cite{DalangSanzSole}, and
Peszat and Zabczyk \cite{PeszatZabczyk}.

Our next result implies among many other things that
Dalang's condition \eqref{cond:Dalang} is generically equivalent to 
the potential-theoretic Condition \ref{cond:1}.

\begin{theorem}[A Maximum Principle]\label{th:Dalang:1}
	For all $\beta>0$,
	\begin{equation}\label{eq:0sup}
		(\bar R_\beta f)(0)=
		\sup_{x\in\R^d}(\bar R_\beta f)(x)=\Upsilon(\beta).
	\end{equation}
	Thus, Condition \ref{cond:1} holds if and only if \eqref{cond:Dalang} holds.
	Furthermore, if Condition \ref{cond:1} $[$and/or \eqref{cond:Dalang}$]$ holds
	and $f$ is lower semicontinuous, then for all $\beta>0$
	there exists $\pi_\beta \in C_0(\R^d)$ such that $\bar R_\beta f=\pi_\beta$
	almost everywhere.
\end{theorem}

In light of Theorem \ref{th:Dalang:1} and Dalang's theorem \cite{Dalang}, 
the novel contributions of our Theorem \ref{th:existence}
are: 
\begin{enumerate}
	 \item[(a)] The condition that $u_0$ is a constant can be improved to one
		about the boundedness of $u_0$ [this can also be derived by adapting
		the method of Dalang \cite{Dalang} to the present setting]; and more significantly
	\item[(b)] We obtain a uniform upper bound for the maximum $L^p(\P)$-moment Liapounov
		exponent of the solution to \eqref{heat} as an \emph{a priori} 
		consequence of the existence of the solution. 
\end{enumerate}
This second contribution leads to the weak intermittence of
solutions, which is a notion that is rooted in the literature of
statistical mechanics.  With this in mind, let us recall the following \cite{FK}:

\begin{definition}
	Suppose that there exists an a.s.-unique solution $\bm{u}:=\{u_t(x)\}_{t>0,x\in\R^d}$
	to \eqref{heat}. We say that $\bm{u}$ is \emph{weakly intermittent}
	if
	\begin{equation}
		\overline\gamma_*(p)<\infty\quad\text{for all $p\in[2\,,\infty)$},
		\quad\text{and}\quad
		\inf_{x\in\R^d}\overline\gamma_x(2)>0.
	\end{equation}
\end{definition}

The same reasoning that was employed in \cite{FK} can be used to
deduce that if the solution to \eqref{heat} is nonnegative for all $t>0$,
then weak intermittence implies the much better-known property
of \emph{intermittency} \cite{CarmonaMolchanov:94,Molchanov,AlmightyChance};
that is, the property that
\begin{equation}
	p\mapsto \frac{\overline\gamma_x(p)}{p}
	\quad\text{is strictly increasing on $[2\,,\infty)$ for all $x\in\R^d$}.
\end{equation}
There is a large literature which shows that, under further mild hypotheses
on $\mathcal{L}$ and/or $f$,
if $u_0$ is nonnegative
then the solution to \eqref{heat} is nonnegative at all
times; see, for example the papers by
Assing and Manthey \cite{AssingManthey},
Carmona and Molchanov \cite{CarmonaMolchanov:94},
Donati-Martin and Pardoux \cite{DonatiMartinPardoux},
Hausmann and Pardoux \cite{HaussmannPardoux},
Kotelenez \cite{Kotelenez},
Manthey \cite{Manthey:86},
Manthey and Stiewe \cite{MantheyStiewe:92,MantheyStiewe:90},
Mueller \cite{Mueller},
Nualart and Pardoux \cite{NualartPardoux},
and  Shiga \cite{Shiga,Shiga:contrasting}.\footnote{In
connection to matters of positivity and regularity, we mention also
a closely-related and fundamental paper by Dawson, Iscoe, and Perkins
\cite{DawsonIscoePerkins}, where \eqref{heat} with
$\sigma(u)=\text{const}\cdot \sqrt{u}$ is considered. And
positivity of the solution is shown to follow from 
many-particle approximations to the underlying SPDE.}

Thus, we can draw the conclusion 
that, in all such cases, weak intermittence actually implies intermittency.

A quick calculation, using only H\"older's inequality, shows that
$p\mapsto \overline\gamma_x(p)/p$ is always nondecreasing on $[2\,,\infty)$.
However, the mentioned strict monotoncity does not 
always hold. When it does hold, then it has some physical significance;
see Zeldovitch, Ruzmaikin, and Sokoloff \cite{AlmightyChance} for a 
physical discussion of intermittency. And Molchanov \cite{Molchanov}
for a mathematical explanation of that physical phenomenon.

Our next main goal is to find nontrivial conditions
that guarantee the weak intermittence of the solution to \eqref{heat}.
In light of Theorem \ref{th:existence}, we aim to
derive a positive lower bound on $\inf_{x\in\R^d}\overline\gamma_x(2)$.
Unfortunately, it is quite hard to do this at the level of generality of
the conditions of Theorem \ref{th:existence}. In fact, informal arguments
suggest that the solution to \eqref{heat} might not always
be weakly intermittent. Thus, we seek to find reasonable restrictions of the 
various parameters of \eqref{heat} which guarantee that the solution
to \eqref{heat} is weakly intermittent.

Let $\hat g$ denotes the Fourier transform of a locally-integrable function
$g$, and consider the following:
\begin{condition}\label{conds:f}
	Suppose:
	\begin{enumerate}
		\item $\hat f(\xi)$ depends on $\xi\in\R^d$ only through
			$|\xi_1|,\ldots,|\xi_d|$;
		\item $|\xi_j|\mapsto \hat f(\xi)$ is nonincreasing for every $j=1,\ldots,d$; 
			and
		\item $\Re\Psi(\xi)$ depends
			on $\xi\in\R^d$ only through $|\xi_1|,\ldots,|\xi_d|$.
	\end{enumerate}
\end{condition}
These are relatively mild provisions on the spectral density $\hat f$
and the process $\bar X$. Our conditions on
the spectral density can be applied to all of the examples that 
we would like to cover. It is possible to show that 
they include the following choices for $f$:
\begin{itemize}
\item[(i)] \emph{Ornstein--Uhlenbeck-type kernels}.
	\[f(x)=c_1\e^{-c_2\|x\|^\alpha}
	\qquad\left[\hat f(\xi) = \frac{c_1}{(2\pi)^d}\int_{\R^d}
	\e^{-i\xi\cdot x -c_2\|x\|^\alpha}\,\d x\right],\]for
	constants $c_1,c_2\in(0\,,\infty)$ and $\alpha\in(0\,,2]$;
\item[(ii)] \emph{Poisson kernels.}
	\[f(x)=\frac{c_1}{ \left(\|x\|^2+ c_2\right)^{(d+1)/2}}
	\qquad\left[\hat f(\xi) = \text{\rm const}\cdot\e^{-\text{\rm const}
	\|\xi\|}\right],\]
	for $c_1,c_2\in(0\,,\infty)$.
\item[(iii)] \emph{Cauchy kernels}.
	\[f(x)=\frac{c_1}{\prod_{j=1}^d (c_2+x_j^2)}
	\qquad\left[\hat f(\xi) = \text{const}\cdot\e^{-\text{\rm const}
	\sum_{j=1}^d|\xi_j|}\right],\]
	for $c_1,c_2\in(0\,,\infty)$; and
\item[(iv)] \emph{Riesz kernels}.
	\[f(x)=\frac{c}{\|x\|^\alpha}
	\qquad\left[\hat f(\xi) = \frac{\rm const}{\|\xi\|^{d-\alpha}}\right],\]
	for $c\in(0\,,\infty)$ and $\alpha\in(0\,,d)$.
\end{itemize}
And one can construct a great number of other 
permissible examples as well.

Having introduced Condition \ref{conds:f}, we can now 
present the third main result of this paper.

\begin{theorem}\label{th:interm}
	Suppose $b\equiv 0$ and Conditions \ref{cond:1} and \ref{conds:f} 
	hold. Suppose, in addition,
	that $\eta:=\inf_{x\in\R^d}u_0(x)>0$
	and there exists ${\rm L}_\sigma\in(0\,,\infty)$ such that
	$\sigma(z)\ge{\rm L}_\sigma|z|$ for all $z\in\R$. Then,
	\begin{equation}\label{eq:cond:interm}
		\inf_{x\in\R^d}\overline\gamma_x(2)\ge \sup\left\{\beta>0:\
		(\bar R_\beta f)(0)\ge \frac{2^{d-1}}{{\rm L}_\sigma^2}
		\right\},
	\end{equation}
	where $\sup\varnothing:= 0$.
\end{theorem}

Before we pause to make a few remarks,
let us briefly study an example.
Consider the case that $(\bar R_0 f)(0)=\infty$.
In that case, $(\bar R_\beta f)(0)\ge 2^{d-1}/{\rm L}_\sigma^2$
for all $\beta>0$ sufficiently small. Hence, in this case,
the hypotheses of Theorem \ref{th:interm} guarantee  weak intermittence of the solution to \eqref{heat}
without further restrictions. 

\begin{remark}\begin{enumerate}
	\item In the case that $\dot{F}$ is space-time white noise,
		the condition ``$\sigma(z)\ge{\rm L}_\sigma|z|$''
		can be replaced with the slightly-better condition
		``$|\sigma(z)|\ge{\rm L}_\sigma|z|$'' \cite{FK}.
	\item We will see later on that,
		when $d=1$, the lower bound
		\eqref{eq:cond:interm} and the
		upper bound \eqref{eq:exist:nonlinear} can sometimes match.
		However, the two bounds can never agree when $d\ge 2$.
		This phenomenon is due to the fact that level sets
		of $\beta\mapsto(\bar R_\beta f)(0)$ 
		cannot describe the growth of $\bm{u}$
		exactly. The correct gauge appears to be 
		a much more complicated function,
		except in the cases that $\dot{F}$ is space-time
		white noise and when $d=1$; compare with \cite{FK}
		for results on the case that $\dot{F}$ denotes
		space-time white noise.
		\qed
\end{enumerate}\end{remark}

We are aware of a few variants of Theorem \ref{th:interm}, but the
next one  is perhaps the most striking since it assumes
only that the nonlinearity term $\sigma$ is asymptotically sublinear.
Thus, the local behavior of $\sigma$ is shown to not have an effect
on weak intermittence, provided that the initial data $u_0$
is sufficiently large. A significant drawback of this result is that its proof does 
not provide any information about how large ``sufficiently large'' 
should be. We introduce the following condition.

\begin{condition}\label{cond:0} $(\bar R_0 f)(0) = \infty$.
\end{condition}

The following result is the mentioned variant of the Theorem \ref{th:interm}.

\begin{theorem}\label{th:interm:asymp}
	Suppose $b\equiv 0$ and Conditions \ref{cond:1}, \ref{conds:f}, and
	\ref{cond:0} hold.
	Suppose, in addition, that $\sigma\ge 0$ pointwise, and 
	$q:=\liminf_{|z|\to\infty} \sigma(z)/|z|>0$. If $u_0(x)>0$
	and $\P\{u_t(x)>0\}=1$ for all $t>0$ and $x\in\R^d$, then
	\begin{equation}
		\overline\gamma_x(2)>0\qquad\text{for every $x\in\R^d$},
	\end{equation}
	provided that $\eta:=\inf_{x\in\R^d}u_0(x)$ is sufficiently large.
\end{theorem}

The preceding results describe our main contributions to the
analysis of the stochastic heat equation
\eqref{heat} in the case that $\sigma$ is not a constant and that $u_0$ is a bounded
and measurable function. But we also study the linearization of \eqref{heat};
this is the case when $\sigma$ is identically equal to one.
In addition to studying existence-and-uniqueness issues, 
we use the theory of Gaussian processes to study continuity properties of the solutions.  Moreover, we produce a class of interesting examples which we briefly describe next. 

Consider the linear stochastic heat equation
\begin{equation}\label{eq:heat33}
	\frac{\partial}{\partial t} u_t(x) = (\Delta u_t)(x) + \dot{F}_t(x),
\end{equation}
where $u_0\equiv 0$,
$x\in\R^3$, $t>0$, and the Laplacian acts on the $x$ variable only.
Then, we construct families of noises
$\dot{F}$ which ensure that \eqref{eq:heat33} has a solution
$\bm{u}:=\{u_t(x)\}_{t\ge 0,x\in\R^d}$ that is a square-integrable
random field. But that random field is discontinuous densely with probability one. In fact,
outside of a single null set [of realizations of the process $\bm{u}$],
\begin{equation}\label{bad}
	\sup_{(t,x)\in V}u_t(x)=-\inf_{(t,x)\in V}u_t(x)=\infty,
\end{equation}
for all open balls $V\subset\R_+\times\R^d$ with rational centers and radii!\footnote{%
It would be wonderful to construct versions of these examples
that apply to fully nonlinear problems; but we do not know how to proceed
in a fully nonlinear [or even semilinear] setting.}
We know of only a few examples of SPDEs with well-defined
random-field solutions that have  unbounded oscillations densely;
see Dalang and L{\'e}v{\^e}que 
\cite{DalangLeveque:06,DalangLeveque:04b,DalangLeveque:04a},
Mytnik and Perkins \cite{MytnikPerkins}, and Foondun, Khoshnevisan, and
Nualart \cite{FKN}. 
The preceding \eqref{bad} yields a quite-simple example of an
otherwise physically-natural stochastic PDE [the operator is the Laplacian in $\R^3$
and the noise is white in time] which
has a very badly-behaved solution.

This paper was influenced greatly by the theoretical physics literature on the
``parabolic Anderson model'' (see, for example, Kardar,
Parisi, and Zhang \cite{KPZ}, Krug and Spohn
\cite[\S5]{KrugSpohn},
Medina, Hwa, Kardar, and Zhang \cite{MHKZ}, and
the book by Zeldovitch, Ruzmaikin, and Sokoloff 
\cite{AlmightyChance}),
as well as the mathematical physics literature on the very same topic
(see, for example, Bertini and Cancrini
\cite{BertiniCancrini:98,BertiniCancrini},
Bertini, Cancrini, and Jona-Lasinio \cite{BertiniCancriniJona},
Bertini and Giacomin \cite{BertiniGiacomin:99,BertiniGiacomin:97},
Carmona, Koralov, and Molchanov \cite{CarmonaKoralovMolchanov},
Carmona and Molchanov 
	\cite{CarmonaMolchanov:95,CarmonaMolchanov:94},
Carmona and Viens \cite{CarmonaViens},
Cranston and Molchanov 
	\cite{CranstonMolchanov:07a,CranstonMolchanov:07b},
Cranston, Mountford, and Shiga 
	\cite{CranstonMountfordShiga:05,CranstonMountfordShiga:02},
Florescu and Viens \cite{FlorescuViens},
G\"artner and den Hollander \cite{GartnerDenHollander},
G\"artner and K\"onig \cite{GartnerKonig},
Hofsted, K\"onig, and M\"orters \cite{HofstedKonigMorters},
K\"onig, Lacoin, and M\"orters \cite{KonigLacoinMorters},
Lieb and Liniger \cite{LiebLiniger},
Molchanov \cite{Molchanov}, and
Woyczy\'nski  \cite{Woyczynski} for a partial listing).
Furthermore, there are interesting variations of the parabolic Anderson model
that correspond to continuous directed-polymer measures;
see Comets and Yoshida \cite{CometsYoshida}
and Comets, Shiga, and Yoshida \cite{CometsShigaYoshida1,CometsShigaYoshida2}.

In a nutshell, the parabolic Anderson model
is equation \eqref{heat} where
$\sigma(u)$ is proportional to $u$. There are many good
reasons why that equation has been studied intensively;
see for instance the Introduction of Carmona and Molchanov
\cite{CarmonaMolchanov:94}. Two such reasons are that the 
parabolic Anderson model it is exactly solvable in the 
two cases where $u_0\equiv \text{constant}$
and $u_0=\delta_0$; and it is related deeply to the stochastic
Burgers equation as well as the KPZ equation of statistical mechanics.

And perhaps not surprisingly,
the results of our Theorems \ref{th:existence}, \ref{th:interm},
and \ref{th:interm:asymp} are sharpest for the parabolic Anderson
model, particularly when $d=1$. However, an inspection of
Theorems \ref{th:existence} and \ref{th:interm} reveals an inconsistency: Our upper bound on the Liapounov exponent [Theorem \ref{th:existence}]
does not require the drift $b$ to be zero; whereas our lower bound
[Theorem \ref{th:interm}] does. 

David Nualart has asked us whether we know how the drift $b$ can affect
the weak intermittence of the solution to \eqref{heat}. 
This seems to be a hard question to answer rigorously 
when the drift $b$ is a general Lipschitz-continuous  function.
But it is intuitively clear that a sufficiently-strong drift
ought to destroy the natural tendency of the solution
to be weakly intermittent.

Although we are not aware of general theorems of this type, 
we are able to give a partial answer to Nualart's question;
and the striking nature of that partial 
answer confirms our initial suspicion that it might be rather difficult to
answer D. Nualart's question in good generality. 

Here is an instance where we can rigorously prove weak intermittency:
Consider the one-dimensional parabolic Anderson model
for the relativistic [or massive/dissipative] Laplacian; i.e.,
the stochastic PDE
\begin{equation}\label{heat:Lapla}
	\frac{\partial}{\partial t}u_t(x) = (\Delta u_t  )(x)
	+\frac{\lambda}{2}\, u_t(x)+\kappa u_t(x)\dot{F}_t(x),
\end{equation}
where $t>0$ and $x\in\R$, $\kappa\neq 0$,
$\lambda\in\R$, and $u_0:\R\to\R$ is a measurable function 
that is bounded uniformly away from zero and infinity. Let us
consider the special case that the correlation function of the noise  is
of Riesz type; that is,
\begin{equation}
	f(z):=\|z\|^{-1+b}
	\qquad\text{for all $z\in\R$},
\end{equation}
where $b\in(0\,,1)$.
Then, Example \ref{ex:massdiss} on page
\pageref{ex:massdiss} implies that weak intermittence holds if and only if
\begin{equation}\label{ouch}
	\lambda >-\left|\kappa\right|^{4/(1+b)} 
	8^{-(1-b)/(1+b)}\left[ 
	\frac{\Gamma(b/2)\Gamma((b+1)/2)}{\sqrt\pi}\right]^{2/(1+b)}.
\end{equation}

This completes our investigation of the solution to \eqref{heat}
when the initial data $u_0$ is a bounded and measurable function.

We conclude this paper by considering \eqref{heat} in the 
other physically-interesting cases where $u_0$ is a finite
Borel measure on $\R^d$. This condition on $u_0$ is
very natural, as in many physical applications $u_0$
denotes the initial distribution of particles in a particle
system in a disordered medium 
\cite{KPZ,MHKZ,KrugSpohn,Molchanov,CarmonaMolchanov:94,BertiniCancrini}. 
From a purely mathematical point of view, this problem
is interesting since it leads us to a different notion
of a solution, which we call \emph{temperate}.
Our proposed temperate solutions differ from 
the much better-known notion of ``mild solutions'' \cite[Chapter 3]{Walsh}. 
They also lead
us to an extension of the notion of a stochastic convolution,
which defines stochastic integrals at almost every time, rather
than pointwise. We believe that the notion of temperate solutions
has other uses in describing otherwise hard-to-define SPDEs.
In order to describe our results on temperate solutions---the final main contributions
of this paper---we would have to develop some machinery.
The details can be found in Chapter \ref{ch:temperate}.

A brief outline of the paper follows: In Chapter \ref{ch:Levy}
we review, very briefly, some analytical facts about L\'evy processes
and their generators, and also construct examples that will be used
in subsequent chapters. 

Chapter \ref{ch:PD} is concerned with positive-definite
functions and their connections to potential theory and harmonic
analysis. And Theorem \ref{th:Dalang:1} is shown to be
a consequence of these connections. Chapter \ref{ch:PD} also contains 
a probabilistic characterization of the analytic condition \eqref{cond:0}
and Condition \ref{cond:1} in terms of continuous additive functionals of
the replica process $\bar X$. Also, a family of useful correlation functions
is constructed in that chapter; that construction uses the results of
Chapter \ref{ch:Levy} on probabilistic potential theory.

In Chapter \ref{ch:linear} we study the linearization of \eqref{heat},
and derive necessary and sufficient conditions for the existence and
spatial continuity of the solution. We also consider various
examples that include \eqref{eq:heat33} above.

Chapter \ref{ch:NL} contains the proofs of Theorems \ref{th:existence}
and \ref{th:interm}. In that chapter we consider also the relativistic
version of \eqref{heat}, thereby constructing examples that include the
mentioned analysis of \eqref{heat:Lapla}.

Finally, we conclude with Chapter \ref{ch:temperate}, where
a new notion of ``temperate solution'' is introduced. And then
that notion is used to produce solutions to  \eqref{heat} in
the case that $b\equiv 0$ and $u_0$ belongs to a suitable family
of finite Borel measures on $\R^d$.

Let us conclude the present chapter
by introducing some notation that will be used throughout
the paper. For all integers $k\ge 1$,
\begin{equation}
	\|x\| := \left( x_1^2+\cdots+x_k^2\right)^{1/2}
	\qquad\text{for every $x\in\R^k$}.
\end{equation}
And if $g:\R^k\to\R$ is a function, then
\begin{equation}
	\lip_g := \sup_{\substack{x,y\in\R^k\\x\neq y}}\frac{
	|g(x)-g(y)|}{\|x-y\|}.
\end{equation}
This socalled Lipschitz constant of $g$ is well defined, but might
be infinity.

Throughout this paper, ``~$\,\widehat{\hskip0.2em}\,$~''
denotes the Fourier transform in the sense
of L. Schwartz; our Fourier transform is normalized so that
\begin{equation}
	\hat g(\xi) := \int_{\R^d} \e^{ix\cdot\xi} g(x)\,\d x
	\quad\text{for all $\xi\in\R^d$ and $g\in L^1(\R^d)$}.
\end{equation}

Finally, if $g:\R^d\to\R$ is a function, then we define
\begin{equation}
	\tilde g(x):= g(-x)\quad\text{for all $x\in\R^d$}.
\end{equation}
And similarly, if $\mu$ is a Borel measure on $\R^d$, then
for all Borel sets $A\subseteq\R^d$,
\begin{equation}
	\tilde\mu(A):= \mu(-A),
	\quad\text{where $-A:=\{-a:\, a\in A\}$}.
\end{equation}

\vskip.4in
\noindent\textbf{Acknowledgements.}
Many thanks are due to the following:
Michael Cranston, who among many other interesting things,
pointed us to the stochastic
heat equation with spatially-correlated noise;
David Nualart, who asked about the interplay between drift and
intermittence; Steven Zelditch, who told us about the lovely book
by Zeldovitch et al \cite{AlmightyChance};
and Daniel Conus for his insightful remarks and suggestions
on this topic.

%% file: FoonKhosh-LP.tex
\chapter{L\'evy Processes}\label{ch:Levy}

\section{Preliminaries}
\index{000X000@$X$, the underlying L\'evy process}%
We begin this chapter with the definition of a L\'evy process which throughout this paper, will be denoted by
 $X:=\{X_t\}_{t\ge 0}$.

\begin{definition}
	We say that $X:=\{X_t\}_{t\ge 0}$ is a \emph{L\'evy process} if:
	\begin{enumerate}
		\item $X_{t+s}-X_s$ is independent of the sigma-algebra
			generated by $\{X_r\}_{r\in[0,s]}$ for every $s,t\ge 0$;
		\item $X_{t+s}-X_s$ has the same distribution
			as $X_t$ for every $s,t\ge 0$;
		\item $t\mapsto X_t$ is continuous in probability; that is,
			$X_s$ converges to $X_t$ in probability as $s\to t$; and
		\item $X_0=0$.
	\end{enumerate}
\end{definition}
By adopting a suitable modification of the paths,
we can and will always assume, without loss of generality, that the
trajectories of $X$ are cadlag; i.e., $t\mapsto X_t$ is almost
surely right-continuous with left limits.
Comprehensive treatments can be found in the
books by Bertoin \cite{Bertoin:book},
Jacob \cite{Jacob}, Kyprianou \cite{Kyprianou}, and Sato \cite{Sato}.

Let $m_t$ denote the distribution of $X_t$ for every
$t\ge 0$; that is,
\begin{equation}
	m_t(A) := \P\{ X_t\in A\}
	\quad\text{for all $t\ge 0$ and Borel sets $A\subseteq\R^d$}.
\end{equation}
Let us recall that throughout this paper we are assuming that
the process $X$ has transition functions; that is,
\begin{equation}\label{cond:ac}
	m_t(\d x)\ll \d x
	\qquad\text{for all $t>0$}.
\end{equation}
\index{000p000t@$p_t(x)$, $\{p_t(x)\}_{t>0,x\in\R^d}$, etc., 
	the underlying transition functions}%
According to Theorem 2.2 of Hawkes \cite{Hawkes:PLP}, 
we can always select a version of
these transition functions that has the following \emph{regularity features}:
\begin{enumerate}
	\item $\int_A p_t(z)\,\d z=m_t(A)$ for all $t>0$ and 
		Borel sets $A\subseteq\R^d$;
	\item $(0\,,\infty)\times\R^d\ni(t\,,x)\mapsto p_t(x)\in\R_+$ 
		is Borel measurable;
	\item $x\mapsto p_t(x)$ is lower semicontinuous for all $t>0$;
	\item $p_{t+s}(x)=(p_t*p_s)(x)$ for all $s,t>0$ and $x\in\R^d$,
\end{enumerate}
where ``$*$'' denotes the convolution operator, defined in the
sense of L. Schwartz.  We work only with such a version of these transition functions.
Note that for all $t\ge 0$ and $x\in\R^d$, and for every Borel-measurable
function $\phi:\R^d\to\R_+$,
\begin{equation}\begin{split}
	\E \phi(x+X_t) &= \int_{\R^d} \phi(z)p_t(z-x)\,\d z\\
	&=(\phi*\tilde p_t)(x),
\end{split}\end{equation}
where we recall
$\tilde{p}_t(x):=p_t(-x)$.

\index{000P000t@$P_t$, $\{P_t\}_{t\ge 0}$, etc., the underlying semigroup}%
Alternatively, one can work with the semigroup
$\{P_t\}_{t\ge 0}$ of $X$, which is defined via
\begin{equation}
	(P_t \phi)(x) := \E \phi(x+X_t).
\end{equation}
It is easy to verify that 
$\{P_t\}_{t\ge 0}$ is a Feller semigroup; i.e., 
\begin{equation}
	P_t: C_0(\R^d)\to C_0(\R^d),
\end{equation}
\index{000C0@$C_0(\R^d)$, the space of all 
	real-valued continuous functions on $\R^d$
	that vanish at infinity}%
where $C_0(\R^d)$ denotes the collection of all continuous
functions $g:\R^d\to\R$ that vanish at infinity.
In fact, under the present conditions, $\{P_t\}_{t\ge 0}$
is strong Feller in the sense of Girsanov \cite{Girsanov:60}; 
see Hawkes \cite{Hawkes:PLP}.

Let us emphasize that
\begin{equation}
	P_t \phi=\phi*\tilde{p}_t,
\end{equation}
valid for all $t\ge 0$, and for instance for 
every nonnegative Borel-measurable 
functions $\phi:\R^d\to\R$.

\index{000Ra@$R_\alpha$, $\{R_\alpha\}_{\alpha>0}$, etc., the resolvent}%
Let $\{R_\alpha\}_{\alpha\ge 0}$ denote the resolvent
of $\{P_t\}_{t\ge 0}$; i.e.,
\begin{equation}
	R_\alpha :=\int_0^\infty \e^{-\alpha s} P_s\, \d s.
\end{equation}
It follows that if $\phi:\R^d\to\R_+$ is Borel measurable, then
\begin{equation}\begin{split}
	(R_\alpha \phi)(x) &= \int_0^\infty \phi(z) r_\alpha(x-z)\,\d z\\
	&=(\phi*\tilde r_\alpha)(x),
\end{split}\end{equation}
where
\index{000ralpha@$r_\alpha$, $\alpha$-potential density}%
\begin{equation}
	r_\alpha(x) := \int_0^\infty \e^{-\alpha  t} p_t(x)\, \d t
	\qquad\text{for $\alpha\ge 0$ and $x\in\R^d$}.
\end{equation}
Each ``$\alpha$-potential density''
$r_\alpha(x)$ is well defined, but could well be infinity
at some [in fact, even all, when $\alpha=0$] $x\in\R^d$.
Nevertheless, the regularity properties of the transition functions
imply that every $r_\alpha$ is lower semicontinuous. Furthermore,
\begin{equation}\label{eq:R:BG}
	R_\alpha:C_0(\R^d)\to C_0(\R^d)
	\qquad\text{for every $\alpha>0$}.
\end{equation}
In fact, $R_\alpha(C_0(\R^d))$ is uniformly dense in 
$C_0(\R^d)$ when $\alpha>0$; 
see Blumenthal and Getoor \cite[Exercise (9.13), p.\ 51]{BG}.

The \emph{characteristic exponent} of the process $X$
is a function $\Psi:\R^d\to\mathbf{C}$ that is
defined uniquely via 
\begin{equation}
	\E\e^{i\xi\cdot X_t} = \e^{-t\Psi(\xi)}
	\qquad\text{for all $\xi\in\R^d$ and $t\ge 0$}.
\end{equation}
The L\'evy--Khintchine formula 
\cite[Theorem 1.2, p.\ 13]{Bertoin:book}, and
a theorem of Schoenberg \cite{Schoenberg:a,Schoenberg:b} 
together imply that 
the family of all L\'evy processes is in one-to-one correspondence
with the family of all  ``negative-definite functions.''

\section{The Generator}\label{Generator}\index{Generator}%
We will be working with the $L^2$-theory of generators,
as developed, for instance, in the book by Fukushima, \={O}shima, and Takeda
\cite{FOT} for more general Markov processes.
We outline the details in the present special case; matters are
greatly simplified and in some cases generalized because of harmonic analysis.

\index{000Dom[L]@$\text{\rm Dom}[\mathcal{L}]$, the domain of the generator}%
Define
\begin{equation}\label{DomL}
	\text{Dom}[\sL] =\left\{
	\phi\in L^2(\R^d):\, \Psi\hat\phi\in L^2(\R^d)\right\}.
\end{equation}
Plancherel's theorem guarantees that
$\phi\in\text{Dom}[\sL]$ if and only if
$\phi:\R^d\to\R$ is Borel-measurable, locally integrable, and
\begin{equation}
	\int_{\R^d}\left( 1+ |\Psi(\xi)|^2\right)\,|\hat\phi(\xi)|^2\,\d\xi<\infty.
\end{equation}

It is well known that the following holds:
\begin{equation}\label{Bochner}
	\limsup_{\|\xi\|\to\infty}\frac{|\Psi(\xi)|}{\|\xi\|^2}<\infty.
\end{equation}
This can be derived directly from the L\'evy--Khintchine formula;
see the book by Bochner \cite[(3.4.14), p.\ 67]{Bochner}. In the ``Notes and
References'' section for Chapter 3 \cite[p.\ 169]{Bochner} of his influential
book, Bochner ascribes \eqref{Bochner} in part to Kolmogorov and L\'evy.

Recall that 
\begin{equation}\label{Wiener:Algebra}
	W^{1,2}(\R^d):=
	\left\{ \phi\in L^2(\R^d):\,  \nabla \phi\in L^2(\R^d)\right\}.
\end{equation}
Because of Plancherel's theorem,
\begin{equation}
	\|\phi\|_{L^2(\R^d)}^2+\|\nabla \phi\|_{L^2(\R^d)}^2
	=\int_{\R^d}\left( 1+\|\xi\|^2\right)|\hat \phi(\xi)|^2\,\d\xi.
\end{equation}
Therefore, we can see from \eqref{Bochner} that 
\begin{equation}\label{S:Dom:L}
	W^{1,2}(\R^d)\subseteq\text{Dom}[\sL]
	\subseteq L^2(\R^d).
\end{equation}
Of course, $\mathcal{S}$ is dense in $W^{1,2}(\R^d)$,
when the latter is endowed with the usual Sobolev
norm,
\begin{equation}
	\|\phi\|_{W^{1,2}(\R^d)}^2:=\|\phi\|_{L^2(\R^d)}^2
	+\|\nabla \phi\|_{L^2(\R^d)}^2.
\end{equation}

According to Plancherel's theorem, 
\begin{equation}\label{(f,Ptphi)}\begin{split}
	(\psi\,,P_t\phi)_{L^2(\R^d)} &= (\psi*m_t\,,\phi)_{L^2(\R^d)}\\
	&=\frac{1}{(2\pi)^d}\int_{\R^d} \overline{\hat\phi(\xi)}\,
		\hat \psi(\xi) \e^{-t\Psi(\xi)}\,\d\xi,
\end{split}\end{equation}
for all $t\ge 0$ and $\psi,\phi\in L^2(\R^d)$.
Moreover,
\begin{equation}
	\left( \psi\,, \frac{P_t\phi-\phi}{t}\right)_{L^2(\R^d)}
	= \frac{1}{(2\pi)^d}\int_{\R^d} \overline{\hat\phi(\xi)}\,
	\hat \psi(\xi) \left[\frac{\e^{-t\Psi(\xi)}-1}{t}\right]\,\d\xi.
\end{equation}
\index{000Lc@$\mathcal{L}$, the generator}%
It follows easily from this and \eqref{DomL} that
\begin{equation}
	\sL\phi:=\lim_{t\downarrow 0} \frac{P_t\phi-\phi}{t}
\end{equation}
exists in $L^2(\R^d)$
if and only if $\phi\in\text{Dom}[\sL]$. Indeed, the sufficiency follows
from the elementary bound
\begin{equation}
	\left|\e^{-t\Psi(\xi)}-1\right|\le t|\Psi(\xi)|,
\end{equation}
and  the Cauchy--Schwarz inequality. And the necessity follows from
Fatou's lemma, upon setting $\psi:=(P_t\phi-\phi)/t$.

Thus, we have the
socalled \emph{generator} [$L^2$-generator, in fact] $\sL$,
defined on its \emph{domain} $\text{Dom}[\sL]$. In addition,
$\sL$ can be thought of as a convolution [or pseudo-differential]
operator with multiplier [or symbol]
$\hat{\sL}=-\Psi$. More precisely,
\begin{equation}\label{FT:Dom:L}
	\widehat{\sL\phi}(\xi)=-\Psi(\xi)\hat\phi(\xi)
	\quad\text{for all $\phi\in\text{Dom}[\sL]$ and $\xi\in\R^d$}.
\end{equation}

Let us note that for all $t\ge 0$, $\xi\in\R^d$, and
$\phi\in L^1(\R^d)$,
\begin{equation}\label{eq:Pt:hat}
	\left| \widehat{P_t\phi}(\xi)\right|^2 =\e^{-2t\Re\Psi(\xi)}
	\cdot|\hat\phi(\xi)|^2.
\end{equation}
Therefore, the well-known nonnegativity of $\Re\Psi(\xi)$---which we prove at the 
beginning of the following section---implies the following.

\begin{lemma}\label{lem:contraction:A(Rd)}
	$P_t$ is a contraction on $W^{1,2}(\R^d)$ for all $t\ge 0$.
	Hence, $\alpha R_\alpha$ is also a contraction on
	$W^{1,2}(\R^d)$ for all $\alpha>0$.
\end{lemma}

\section{The Replica Semigroup and Associated Sobolev Spaces}
\index{000X*@$X^*:=-X$, the dual process}%
Let $X^*$ denote an independent copy of the L\'evy process
$-X$ and, following L\'evy \cite{Levy}, define
\index{000Xbar@$\bar X$, the replica process}%
\begin{equation}
	\bar{X}_t:=X_t+X_t^*\qquad\text{for all $t\ge 0$}.
\end{equation}
It is easy to see that $X^*:=\{X_t^*\}_{t\ge 0}$ is the dual process
to $X$, and $\bar{X}:=\{\bar{X}_t\}_{t\ge 0}$
is a symmetric L\'evy process on $\R^d$.
And if we denote the distribution of
$\bar{X}_t$ by $\bar{m}_t$, then
\begin{equation}
	\bar{m}_t(A)=(m_t*\tilde{m}_t)(A)
	\quad\text{for all Borel sets $A\subseteq\R^d$},
\end{equation}
where $\tilde{m}_t(A):=m_t(-A)$.
Note that the Fourier transform of $\bar{m}_t$ is
\begin{equation}\begin{split}
	\widehat{\bar{m}_t}(\xi) &= |\hat{m}_t(\xi)|^2\\
	&= \e^{-2t\Re\Psi(\xi)}.
\end{split}\end{equation}
Among other things, this implies the classical fact that
\begin{equation}\label{eq:RePsi:pos}
	\Re\Psi(\xi)\ge 0\text{ for all $\xi\in\R^d$}.
\end{equation}

The absolute-continuity
condition \eqref{cond:ac} implies that every $\bar m_t$ is
absolutely continuous with respect to the Lebesgue measure
on $\R^d$ [$t>0$]. We denote the resulting transition density
by $\bar p_t$.  Every $\bar p_t$ is a symmetric function 
on $\R^d$ $[t>0]$.

We can always choose a version of $\bar p$ that has good
regularity features [of the type mentioned earlier for $p$].
In fact, the following version works:
\index{000pbart@$\bar p_t(x)$, $\{\bar p_t(x)\}_{t>0,x\in\R^d}$, etc.,
	the replica transition functions}%
\begin{equation}\begin{split}
	\bar p_t(x) &:= (p_t*\tilde p_t)(x)\\
	&= \int_{\R^d} p_t(x+z)p_t(z)\,\d z
		\qquad\text{for $x\in\R^d$ and $t>0$}.
\end{split}\end{equation}
\index{000Pbart@$\bar P_t$, $\{\bar P_t\}_{t\ge 0}$, etc.,
	the replica semigroup}%
Equivalently, if $\bar{P}:=\{\bar{P}_t\}_{t\ge 0}$ denotes the
semigroup of $\bar{X}$, then 
\begin{equation}
	\bar{P}_t=P_tP_t^*\quad\text{for all $t\ge 0$},
\end{equation}
where $P_t^*$ denotes the adjoint of $P_t$ in $L^2(\R^d)$.
Every $\bar{P}_t$ is a self-adjoint contraction on $L^2(\R^d)$.

Motivated by the work of Kardar \cite{Kardar}, we refer to $\bar{X}$ and $\bar{P}$
respectively as the \emph{replica process} and the
\emph{replica semigroup}. The corresponding generator is denoted by
$\bar\sL$ and its domain by $\text{Dom}[\bar\sL]$.
\index{Replica!semigroup}\index{Replica!process}%

For all $\alpha\ge 0$, we can
define the \emph{replica $\alpha$-potential density} $\bar r_\alpha$ as
\index{000rbar@$\bar r_\alpha$, the replica $\alpha$-potential density}%
\begin{equation}
	\bar r_\alpha (x) := \int_0^\infty \e^{-\alpha s} \bar p_s(x)\,\d s
	\qquad\text{for all $x\in\R^d$}.
\end{equation}
Clearly,
$\bar r_\alpha(x)$ is well defined; but $\bar r_\alpha(x)$
can be infinite for some [and even
all, in the case that $\alpha=0$] $x\in\R^d$.
\index{000Rbar@$\bar R_\alpha$, $\{\bar R_\alpha\}_{\alpha>0}$,
	etc., the replica resolvent}%
The resolvent $\bar{R}:=\{\bar{R}_\alpha\}_{\alpha> 0}$ of the semigroup $\bar{P}$
can also be defined as follows
\begin{equation}\begin{split}
	(\bar{R}_\alpha \phi)(x) &:= \int_0^\infty \e^{-\alpha s}
		(\bar{P}_s\phi)(x)\,\d s\\
	&=\int_{\R^d} \phi(z)\bar r_\alpha(z-x)\,\d z,
\end{split}\end{equation}
for all $\alpha>0$ and $x\in\R^d$. Since $\bar r_\alpha$ is
a symmetric function on $\R^d$, it follows that $\bar R_\alpha\phi
=\phi*\bar r_\alpha$.

The preceding quantity which is called the \emph{$\alpha$-potential} of
$\phi$ makes sense,
for example, if $\phi:\R^d\to\R_+$ is Borel measurable,
or when $\phi\in L^p(\R^d)$ for some $p\in[1\,,\infty]$
because every $\bar{P}_s$ is a contraction on $L^p(\R^d)$.

\section{On the Heat Equation and Transition Functions}\label{sec:TF}

We begin by recalling some generally-known
facts about the fundamental [weak] solution to
the \emph{heat \textnormal{[or Kolmogorov]} equation} for $\sL$:
We seek to find a function $H$ such that
for all $t>0$ and $x\in\R^d$,
\begin{equation}\label{eq:heat:L}\left|\begin{split}
	&\frac{\partial}{\partial t} H_t(x) = (\sL H_t)(x),\\
	&H_0=\delta_z,
\end{split}\right.\end{equation}
where $z\in\R^d$ is fixed.
By rewriting the above in terms
of the [spatial] Fourier transform $\hat H_t$ and 
using the fact that the [weak] Fourier transform of 
$\mathcal{L}H_t$ is $-\Psi\cdot\hat H_t$ [see \eqref{FT:Dom:L}],
we obtain the ordinary differential equation,
\begin{equation}\left|\begin{split}
	&\frac{\partial}{\partial t} \hat{H}_t(\xi) = -\Psi(\xi)\hat{H}_t(\xi),\\
	&\hat{H}_0(\xi)=\e^{i\xi\cdot z}.
\end{split}\right.\end{equation}
The unique solution to this ODE is
\begin{equation}
	\hat H_t(\xi)=\e^{i\xi\cdot z -t\Psi(\xi)}.
\end{equation}
Direct inspection of the Fourier transform reveals that $H_t(x)=p_t(z-x)$.
Thus, we find that the fundamental solution to
\eqref{eq:heat:L} is the measurable
function $(0\,,\infty)\times\R^d\times\R^d
\ni (t\,;x\,,y)\mapsto p_t(y-x)$. In  particular,
we might observe that in order to have a function solution
to \eqref{eq:heat:L}, it is necessary as well as sufficient
that the underlying L\'evy process $X$ has transition densities.

We are thus led to the natural question: ``What are the necessary and sufficient conditions on the 
characteristic exponent $\Psi$ that ensure the existence of
transition densities of the corresponding L\'evy processes''?
Unfortunately, there is no satisfactory known answer
to this question at this time, though several
attempts have been made in the first half of the twentieth century;
see, for example Blum and Rosenblatt
\cite{BlumRosenblatt}, Fisz and Varadarjan \cite{FiszVaradara},
Hartman and Wintner \cite{HartmanWintner},
and Tucker \cite{Tucker:65,Tucker:64,Tucker:62}.

More recently, Bass and Cranston \cite{BassCranston} applied Malliavin calculus
to a family of stochastic differential equations driven by jump noises.
Their result can be used to supply good sufficient conditions that
ensure the existence of smooth transition functions. And in \cite{NourdinSimon}, Nourdin and Simon have proved, among other things,
that if a L\'evy process has transition functions, then so does the same
process plus a drift.

We will use the following unpublished result of Hawkes.  It 
typically provides a good-enough sufficient condition
for the existence of transition functions. We include a proof
in order to document this interesting fact.

\begin{proposition}[Hawkes \cite{Hawkes}]\label{pr:hawkes}
	The following conditions are equivalent:
	\begin{enumerate}
		\item\label{H1} Condition \eqref{cond:ac} holds and 
		$p_t\in L^2(\R^d)$
		for all $t>0$;
	\item\label{H2} Condition \eqref{cond:ac} holds and  
		$p_t\in L^\infty(\R^d)$
		for all $t>0$;
		\item\label{H11} Condition \eqref{cond:ac} holds and 
		$p_t\in L^2(\R^d)$
		for almost every $t>0$;
	\item\label{H21} Condition \eqref{cond:ac} holds and  
		$p_t\in L^\infty(\R^d)$
		for almost every $t>0$;
		\item\label{H3} $\exp(-\Re\Psi)\in L^t(\R^d)$ for all $t>0$.
		\item\label{H31} $\exp(-\Re\Psi)\in L^t(\R^d)$ for almost every $t>0$.
	\end{enumerate}
	Moreover, any one of these conditions implies that: (i)
	$(t\,,x)\mapsto p_t(x)$ has a continuous version which is
	uniformly continuous for all $(t\,,x)\in[\eta\,,\infty)\times\R^d$
	for every $\eta>0$; and (ii) $p_t$ vanishes at infinity for all $t>0$.
\end{proposition}

\begin{proof}[Proof (Hawkes \cite{Hawkes})]
	Recall that $\int_{\R^d} p_t(x)\,\d x=1$ and
	$p_t=p_{t/2}*p_{t/2}$. Therefore,
	two applications of Young's inequality yield
	\begin{equation}
		\|p_t\|_{L^\infty(\R^d)} \le \|p_{t/2}\|_{L^2(\R^d)}^2
		\le \|p_{t/2}\|_{L^\infty(\R^d)}
		\qquad\text{for all $t>0$}.
	\end{equation}
	Consequently, \eqref{H1}$\Leftrightarrow$\eqref{H2}
	and \eqref{H11}$\Leftrightarrow$\eqref{H21}.
	
	Next let us suppose that \eqref{H31} holds.
	Because
	\begin{equation}\begin{split}
		|\hat p_t(\xi)| &=\left|\e^{-t\Psi(\xi)}\right|\\
		&\le\e^{-t\Re\Psi(\xi)},
	\end{split}\end{equation}
	Plancherel's theorem ensures that
	\begin{equation}\begin{split}
		\| p_t\|_{L^2(\R^d)}^2 &=\frac{1}{(2\pi)^d}
			\|\hat p_t\|_{L^2(\R^d)}^2\\
		&\le \left\|\e^{-2t\Re\Psi} \right\|_{L^1(\R^d)}.
	\end{split}\end{equation}
	Since $\Re\Psi\ge 0$, it follows from \eqref{H31}
	that $p_t\in L^2(\R^d)$ for \emph{every} $t>0$; i.e.,
	\eqref{H31}$\Rightarrow$\eqref{H1}. Moreover,
	we have---in this case---the following inversion formula:
	For almost all $x\in\R^d$ and every $t>0$,
	\begin{equation}\label{eq:p_t}
		p_t(x) = \frac{1}{(2\pi)^d}\int_{\R^d}\e^{-i\xi\cdot x-t\Psi(\xi)}
		\,\d\xi.
	\end{equation}

	It remains to prove that \eqref{H1} and equivalently \eqref{H2} together 
	imply \eqref{H3}.
	Recall that $\tilde{p}_t(x):=p_t(-x)$ and observe that 
	\begin{equation}
		\widehat{p_{t/4}*\tilde{p}_{t/4}} = \e^{-(t/2)\Re\Psi}.
	\end{equation}
	Therefore, by Plancherel's theorem,
	\begin{equation}\begin{split}
		\left\|\e^{-\Re\Psi}\right\|_{L^t(\R^d)}
			&=\left\|\e^{-(t/2)\Re\Psi}\right\|_{L^2(\R^d)}\\
		&= (2\pi)^d \left\| p_{t/4}*\tilde{p}_{t/4}\right\|_{L^2(\R^d)}^2\\
		&\le (2\pi)^d \left\|  p_{t/4}*\tilde{p}_{t/4}\right\|_{%
			L^\infty(\R^d)}.
	\end{split}\end{equation}
	Consequently, Young's inequality, implies that
	\begin{equation}
		\left\|\e^{-\Re\Psi}\right\|_{L^t(\R^d)}
		\le(2\pi)^d\|p_{t/4}\|_{L^2(\R^d)}^2,
	\end{equation}
	which has the desired effect.
	
	Finally, if \eqref{H3} holds then the inversion
	theorem applies and tells us that
	we can always choose a version of $p$ that satisfies the properties of
	the final paragraph in the statement of theorem.
\end{proof}

There is also the following 1970 theorem of J. Zabczyk,
which characterizes \eqref{cond:ac} in special though important
cases.

\begin{proposition}[Zabczyk \protect{%
	\cite[Example (4.6)]{Zab}}]\label{pr:Zab}
	If $d\ge 2$ and $\Psi$ is a radial function, then \eqref{cond:ac}
	holds if and only if 
	\begin{equation}\label{eq:blow}
		\lim_{\|\xi\|\to\infty}
		\Psi(\xi)=\infty.
	\end{equation}
\end{proposition}

We now put this beautiful result in context. By the Riemann-Lebesgue lemma, 
if the process $X$ has transition functions then equality \eqref{eq:blow} holds. 
Proposition \ref{pr:Zab} states that the converse also holds,
provided that $d\ge 2$ and $\Psi$ is radial. 
In other words, the Riemann--Lebesgue lemma is a necessary and sufficient 
condition for the existence of transition functions whenever 
$d\ge 2$ and $\Psi$ is radial.  As was mentioned in Zabczyk \cite{Zab}, the preceding
is not in general true when $d=1$. This can be seen by considering
$\Psi$ to correspond to two independent
one-dimensional Poisson processes.

\section{On a Family of Isotropic L\'evy Processes}
The main result of this section will be needed to construct a counterexample in Chapter 4.  It is possible that it is known but we were not able to find an explicit reference. So we provide a complete proof.
We begin by recalling a few definitions used to study L\'evy processes. 

We say that a L\'evy process $X:=\{X_t\}_{t\ge 0}$ is \emph{isotropic}
if its characteristic exponent $\Psi$ is a radial function [and hence
also real-valued and nonnegative]. Such processes are also known
as \emph{radial processes}; see Millar \cite{Millar}.

A [standard] \emph{subordinator} $\tau:=\{\tau_t\}_{t\ge 0}$
is a one-dimensional L\'evy process that is nondecreasing
and $\tau_0:=0$.
According to the L\'evy--Khintchine formula
 \cite[Theorem 1.2, p.\ 13]{Bertoin:99},
 every subordinator $\tau$ is determined by the formula 
\begin{equation}
	\E \e^{-\lambda \tau_t} =\e^{-t\Phi(\lambda)}, 
\end{equation}
where $t,\lambda\ge 0$,
and
\begin{equation}\label{eq:Phi}
	\Phi(\lambda) = \int_0^\infty \left(
	1-\e^{-\lambda z}\right)\,\Pi(\d z),
\end{equation}
for a Borel measure $\Pi$ on $(0\,,\infty)$ that satisfies
\begin{equation}\label{eq:LM}
	\int_0^\infty (1\wedge x)\,\Pi(\d x)<\infty.
\end{equation}
The function $\Phi$ is the socalled
\emph{Laplace exponent} of the subordinator $\tau$.
	 
We have the following lemma.
\begin{lemma}\label{lem:subord}
	Choose and fix two numbers $p\in(0\,,1)$ and $q\in\R$. Then,
	there exists a subordinator $\tau$ on $\R_+$ whose
	Laplace exponent satisfies
	\begin{equation}
		0<\inf_{\lambda>\e} \frac{\Phi (\lambda)}{\lambda^p(\log\lambda)^{q/2}}
		\le\sup_{\lambda>\e}\frac{\Phi (\lambda)}{\lambda^p(\log\lambda)^{q/2}}<\infty.
	\end{equation}
\end{lemma}

\begin{proof}
	Define a measure $\Pi$ via
	 \begin{equation}	
		 \frac{\Pi(\d x)}{\d x} :=\begin{cases}
		 	x^{-1-p}\left(\log(1/x)\right)^{q/2}&\text{if $0<x<\frac12$},\\
			0&\text{otherwise}.
		 \end{cases}
	 \end{equation}
	 Since $p\in(0\,,1)$, it follows that \eqref{eq:LM} holds,
	 whence $\Pi$ is a L\'evy measure.
	 
	 We can also apply the definition \eqref{eq:Phi} of the Laplace exponent
	 and write $\Phi(\lambda)=\lambda^pQ(\lambda)$, where
	 \begin{equation}
	 	Q(\lambda) = \int_0^{\lambda/2} \frac{1-\e^{-x}}{x^{1+p}}
		\left(\log(\lambda /x)\right)^{q/2}\,\d x
		\qquad\text{for $\lambda> 0$}.
	 \end{equation}
	 In order to complete the proof, we will verify that 
	$Q(\lambda)\asymp(\log\lambda)^{q/2}$ for
	 $\lambda>\e$.\footnote{%
	As usual, $h(x)\asymp g(x)$ over a certain range of
	$x$'s is short-hand for the statement that, uniformly over that range of $x$'s,
	$h(x)/g(x)$ is bounded above and below by positive and finite constants.}
	 We do so in the special case that $q\ge 0$; similar arguments
	 can be used to estimate $Q(\lambda)$ in the case that $q<0$.
	 
	 Whenever $\lambda>\e$, we can write
	 \begin{equation}
	 	Q(\lambda) := I_1+I_2,
	 \end{equation}
	 where
	 \begin{equation}\begin{split}
	 	I_1 & := \int_1^{\lambda/2}\frac{1-\e^{-x}}{x^{1+p}}
			\left(\log(\lambda /x)\right)^{q/2}\,\d x,\\
		I_2 & := \int_0^1 \frac{1-\e^{-x}}{x^{1+p}}
			\left(\log(\lambda /x)\right)^{q/2}\,\d x.
	 \end{split}\end{equation}
	 Evidently,
	 \begin{equation}\begin{split}
	 	I_1 &\le (\log\lambda)^{q/2} \cdot \int_1^\infty \frac{\d x}{x^{1+p}}\\
		&= \frac1p (\log\lambda)^{q/2}.
	 \end{split}\end{equation}
	 Since $I_1\ge 0$, it remains to prove that
	 \begin{equation}
	 	I_2\asymp(\log\lambda)^{q/2}
		\qquad\text{for  $\lambda>1$}.
	\end{equation}
	We establish this by deriving first an upper, and then a lower,
	 bound for $I_2$. Because $1-\exp(-y)\le y$ for $y\ge 0$,
	 and since $\sup_{z\in(0,1)}z^\epsilon \log (1/z)<\infty$
	 for all $\epsilon\in(0\,,1)$, it follows that
	 \begin{equation}\begin{split}
	 	I_2&\le (\log\lambda)^{q/2}\cdot
			\int_0^1 \left(1+\frac{\log(1/x)}{\log\lambda}
			\right)^{q/2} \frac{\d x}{x^p}\\
		&\le \text{const}\cdot (\log\lambda)^{q/2}.
	 \end{split}\end{equation}
	 And a similar lower bound is obtained via the bounds: 
	 (i) $1-\exp(-x)\ge x/2$; and
	 (ii) $\log(\lambda/x)\ge\log\lambda$; both valid for all $x\in(0\,,1)$.
\end{proof}

The following is the main result of this section. It gives a special construction of
an isotropic L\'evy process $X:=\{X_t\}_{t\ge 0}$
whose characteristic exponent is regularly varying in a special manner.

\begin{theorem}\label{th:Levy:asymp}
	Choose and fix $r\in(0\,,2)$ and $q\in\R$. Then,
	there exists an isotropic L\'evy process $X:=\{X_t\}_{t\ge 0}$
	such that
	\begin{equation}
		0<\inf_{\substack{\xi\in\R^d:\\\|\xi\|>\e}}
		\frac{\Psi(\xi)}{\|\xi\|^r(\log\|\xi\|)^q} \le
		\sup_{\substack{\xi\in\R^d:\\\|\xi\|>\e}}\frac{\Psi(\xi)}{\|\xi\|^r(\log\|\xi\|)^q} 
		<\infty.
	\end{equation}
\end{theorem}

\begin{proof}
	Let $B:=\{B(t)\}_{t\ge 0}$ denote a $d$-dimensional
	Brownian motion, independent from the
	subordinator $\tau:=\{\tau_t\}_{t\ge 0}$ of Lemma \ref{lem:subord}.
	Define
	\begin{equation}
		X_t:=B(\tau_t)
		\qquad\text{for all $t\ge 0$}.
	\end{equation}
	It is well known---as well as easy to check---that the process
	$X:=\{X_t\}_{t\ge 0}$ is a L\'evy process with characteristic
	exponent
	\begin{equation}
		\E\left[ \e^{i\xi\cdot X_t}\right] = \exp\left(-t\Phi
		\left(\frac{\|\xi\|^2}{2}\right)\right)
		\qquad\text{for all $t\ge 0$ and $\xi\in\R^d$}.
	\end{equation}
	That is, $\Psi(\xi)=\Phi(\|\xi\|^2/2)$; Lemma \ref{lem:subord}
	[applied with $p:=r/2$]
	completes the remainder of the proof.
\end{proof}

%% file: FoonKhosh-PosDefPot.tex
\chapter[Positive-Definite Functions]{%
	Positive-Definite Functions, Fourier Analysis, and Probabilistic
	Potential Theory}
\label{ch:PD}

A large part of this paper relies heavily on our ensuing
analysis of positive-definite functions and 
their many connections to harmonic analysis. In this chapter we  develop 
the requisite theory and prove Theorem \ref{th:Dalang:1}. 
We also give intrinsically-probabilistic interpretations to the 
two central potential-theoretic
hypotheses of this paper; namely Conditions \ref{cond:1} and \ref{cond:0}.  
Even though, most of the results derived in this chapter will be use later in this paper, they might be of independent interest. 

\section{Fourier Analysis}
Let us begin by recalling some basic facts from harmonic analysis. 

First, let us recall 
our definition of Fourier transform.
\begin{equation}
	\hat g(\xi) := \int_{\R^d} \e^{ix\cdot\xi} g(x)\,\d x
	\quad\text{for all $\xi\in\R^d$ and $g\in L^1(\R^d)$}.
\end{equation}
This particular normalization is standard in probability theory
and leads to the following form
of the Parseval identity, which plays a big role in the ensuing
theory:
\begin{equation}\label{eq:Parseval}
	\int_{\R^d} g(x)\, h(x)\,\d x = \frac{1}{(2\pi)^d}
	\int_{\R^d}\hat g(\xi) \, \overline{\hat h(\xi)}\,\d\xi
	\quad\text{for all $g,h\in L^2(\R^d)$}.
\end{equation}
The preceding is another way to say that the Fourier transform is
an ``isometry'' from $L^2(\R^d)$ onto itself.\footnote{Strictly
speaking, this is not true, as is evidenced by the
multiplicative factor of $(2\pi)^{-d/2}$ in the
right-hand side of equation
\eqref{eq:Parseval}. And that is why the
word isometry appears in quotations: 
In order for the Fourier transform to be a proper isometry from $L^2(\R^d)$
onto itself, we need to use a different normalization than the
one used here. We have not done that as 
it would be nonstandard for probability theory.}

Recall that a function $g:\R^d\to\R$ is \emph{tempered}
if it is Borel-measurable and there exists $k\ge 0$ such that
\begin{equation}
	\sup_{x\in\R^d}\frac{|g(x)|}{(1+|x|)^k}<\infty.
\end{equation}

Also, recall that $g$ has \emph{at most polynomial growth} 
[\emph{at least polynomial growth}, resp.] if
the preceding holds for some $k\ge 0$ [$k<0$, resp.].
Finally, we say that $g\in\mathcal{S}$ if 
$g$ and all of its derivatives have at least polynomial growth.
\index{000Scal@$\mathcal{S}$, the space of all rapidly-decreasing 
test functions.}%
The collection $\mathcal{S}$ is the usual space of
\emph{rapidly-decreasing test functions} on $\R^d$.

The following is an
elementary, but important, variant of the Parseval identity.

\begin{lemma}\label{lem:Parseval}
	Parseval's identity
	\eqref{eq:Parseval} is valid when $g\in\mathcal{S}$ and $h:\R^d\to\R$ 
	is continuous and tempered.
\end{lemma}
We leave the proof to the interested reader.

\section{Positive-Definite Functions}
\index{Positive-definite function}%
Recall that a function $g:\R^d\to\R_+$ is \emph{positive definite}
if $g$ is tempered and $( \phi\,, g*\phi )_{L^2(\R^d)}
\ge 0$ for all rapidly-decreasing test functions $\phi$.
That is,
\begin{equation}
	\int_{\R^d\times\R^d} \phi(x)\phi(y) g(x-y)\,\d x\,\d y\ge 0
	\qquad\text{for all $\phi\in\mathcal{S}$}.
\end{equation}
The following central result of L. Schwartz characterizes
positive-definite functions.

\begin{proposition}[Schwartz \protect{\cite[Theorem 3,
	p.\ 157]{GV}}]\label{BochnerSchwartz}
	If $g:\R^d\to\bar\R$ is positive definite, then there exists a tempered
	measure $\Gamma$ on $\R^d$ such that $g=\hat\Gamma$.
\end{proposition}

The preceding has an elementary converse as well:
Recall that a Borel measure $\Gamma$ on $\R^d$ is of 
\emph{positive type} if $\hat\Gamma\ge 0$ in the sense of
distributions; that is, $\hat\Gamma (\phi)\ge 0$
for all nonnegative $\phi\in\mathcal{S}$. 
Then it is straightforward to check that
if $\Gamma$ is a tempered
measure of positive type on $\R^d$, then $\hat\Gamma$ is
positive definite.

Schwartz's theorem is a generalization of the following theorem
of Herglotz [$d=1$] and Bochner [$d\ge 2$]:

\begin{proposition}[Herglotz \cite{Herglotz}, Bochner \cite{Bochner:33}]\label{HerzogBochner}
	If $g:\R^d\to\R$ is continuous and positive definite,
	then there exists a finite Borel measure $\Gamma$ on $\R^d$
	such that $g=\hat\Gamma$. 
\end{proposition}

Recall also that a continuous function $g:\R^d\to\R$
is positive definite if and only if is is positive definite in 
the sense of Herglotz, Bochner, P\'olya, etc.; that is if
\begin{equation}
	\sum_{i,j=1}^N a_i\overline{a_j}\,
	g(x_i-x_j)\ge 0\text{ for all $a_1,\ldots,a_N\in\mathbf{C}$
	and $x_1,\ldots,x_N\in\R^d$}.
\end{equation}
The proof uses elementary integration theory, and we merely
recall the [easy] steps: Define the finite complex
measure
\begin{equation}
	\mu := \sum_{i=1}^N a_i \delta_{x_i},
\end{equation}
and choose 
a sequence $\phi_1,\phi_2,\ldots:\R^d\to\R$ of probability density
functions, each in $\mathcal{S}$, such that $\hat\phi_n\to 1$
pointwise as $n\to\infty$. Then, one can 
make precise the following approximation: As $n\to\infty$,
\begin{equation}\begin{split}
	\sum_{i,j=1}^N a_i\overline{a_j}\, g(x_i-x_j)&=
		\int(g*\mu)\,\d\mu\\
	&\approx \int(g*\phi_n*\mu)\,\d(\phi_n*\mu).
\end{split}\end{equation}
And the final quantity is nonnegative because $g$
is positive definite and $\Re(\phi_n*\mu),
\Im(\phi_n*\mu)\in\mathcal{S}$ for all $n\ge 1$.

\section{A Preliminary Maximum Principle}
Now that we have recalled the basic definitions and properties
of positive-definite functions, we can begin our proof of 
our maximum principle [Theorem \ref{th:Dalang:1}]. 
But first let us prove the following technical result.

\begin{lemma}\label{lem:C_0}
	If $\phi\in\mathcal{S}$, then there exists a version of
	$f*\phi$ that is in  $C_0(\R^d)$. Consequently,
	$\bar R_\beta(f*\phi)\in C_0(\R^d)$
	for every $\beta>0$.
\end{lemma}

\begin{proof}
	Because $\hat f$ is tempered, the following defines a uniformly continuous
	function on $\R^d:$
	\begin{equation}	
		h(x) = \frac{1}{(2\pi)^d}\int_{\R^d}\e^{-i x\cdot \xi}\hat f(\xi)
		\hat\phi(\xi)\,\d\xi
		\quad\text{for all $x\in\R^d$}.
	\end{equation}
	In fact, $h\in C_0(\R^d)$ because of the Riemann--Lebesgue lemma.
	Furthermore, if $\psi\in\mathcal{S}$, then
	\begin{equation}\begin{split}
		\int_{\R^d}\psi(x)h(x)\,\d x &=\frac{1}{(2\pi)^d}\int_{\R^d}
			\overline{\hat\psi(\xi)}\, \hat f(\xi)
			\hat\phi(\xi)\, \d \xi\\
		&=\int_{\R^d}\psi(x) (f*\phi)(x)\,\d x.
	\end{split}\end{equation}
	The first line is justified by the Fubini theorem, and the second
	by the Parseval identity. It follows from density and 
	the Lebesgue differentiation
	theorem that $h=f*\phi$ almost everywhere. This proves the
	first assertion of the lemma. In addition, 
	\begin{equation}\begin{split}
		(\bar R_\beta h)(x) &= \int_{\R^d} \bar r_\beta (y-x)h(y)\,\d y\\
		&= \int_{\R^d} \bar r_\beta (y-x)(f*\phi)(y)\,\d y\\
		&=\left( \bar R_\beta(f*\phi) \right)(x).
	\end{split}\end{equation}
	Since $h\in C_0(\R^d)$, it follows from \eqref{eq:R:BG}
	that $\bar R_\beta(f*\phi)= \bar R_\beta h \in C_0(\R^d)$.
\end{proof}

The following
contains a portion of the said maximum principle of Theorem \ref{th:Dalang:1}.  It also provides
some of the requisite technical estimates that are needed
for the remainder of the proof of Theorem \ref{th:Dalang:1}.

\begin{proposition}\label{pr:FA}
	For all $\beta\ge 0$,
	\begin{equation}\label{eq:Up:pot}
		\Upsilon(\beta)=\sup_{x\in\R^d}(\bar R_\beta f)(x)
		=\mathop{\text{\rm ess sup}}\limits_{x\in\R^d} (\bar R_\beta f)(x)
		=\limsup_{x\to 0}(\bar R_\beta f)(x),
	\end{equation}
	where $\Upsilon(0):=\lim_{\beta\downarrow 0}\Upsilon(\beta)$.
\end{proposition}

\begin{proof}
	First, we prove the proposition in the case that $\beta>0$.
	
	In accord with Lemma \ref{lem:C_0},
	if $\phi\in\mathcal{S}$, then 
	$\bar R_\beta(f*\phi)$ is continuous. Therefore,
	the Plancherel theorem applies pointwise: For all $x\in\R^d$,
	\begin{equation}\label{eq:x}
		(\bar R_\beta(f*\phi))(x) = \frac{1}{(2\pi)^d}\int_{\R^d}\frac{\hat{f}(\xi)
		\hat{\phi}(\xi)\e^{-ix\cdot\xi}}{\beta+2\Re\Psi(\xi)}\,\d\xi.
	\end{equation}
	[Without the asserted continuity, we could only deduce this for almost
	every $x\in\R^d$.]
	In particular, for all probability densities $\phi\in\mathcal{S}$,
	\begin{equation}\label{eq:sup}
		\sup_{x\in\R^d}\left( \bar R_\beta(f*\phi)\right)(x) \le 
		\Upsilon(\beta).
	\end{equation}
	[$\Upsilon$ was defined in \eqref{eq:Upsilon}.]
	If $\{\phi_n\}_{n=1}^\infty$ is an approximate identity
	consisting solely of probability densities in
	$\mathcal{S}$, then 
	\begin{equation}
		\bar R_\beta f\le\liminf_{n\to\infty}\bar{R}_\beta(f*\phi_n)
		\qquad\text{[pointwise],}
	\end{equation}
	by Fatou's lemma. Consequently, \eqref{eq:sup} implies that
	\begin{equation}\label{eq:UB1}
		\sup_{x\in\R^d}(\bar R_\beta f)(x)\le \Upsilon(\beta).
	\end{equation}
	
	In order to prove the reverse bound, define the
	Gaussian mollifiers $\{\phi_n\}_{n=1}^\infty$:
	\begin{equation}\label{eq:Gaussian:density}
		\phi_n(z) := \left(\frac{n}{2\pi} \right)^{d/2}\exp\left(
		-\frac{\|z\|^2n}{2}\right).
	\end{equation}
	And observe that
	\begin{equation}\label{eq:mollify:approx}\begin{split}
		\left( \bar R_\beta(f*\phi_n)\right)(0) &= \frac{1}{(2\pi)^d}\int_{\R^d}
			\frac{\hat{f}(\xi)\e^{-\|\xi\|^2/(2n)}}{\beta+2
			\Re\Psi(\xi)}\,\d\xi\\
		&=(1+o(1))\Upsilon(\beta)\qquad\text{as $n\to\infty$},
	\end{split}\end{equation}
	thanks to the monotone convergence theorem. On the other
	hand, 
	\begin{equation}\begin{split}
		\left( \bar R_\beta(f*\phi_n)\right)(0)
			&=\int_{\R^d} \bar r_\beta(y)(f*\phi_n)(y)\,\d y\\
		&=\int_{\R^d} (\bar R_\beta f)(y) \phi_n(y)\,\d y\\
		&\le\mathop{\text{\rm ess sup}}\limits_{x\in\R^d}(\bar R_\beta f)(x),
	\end{split}\end{equation}
	since $\int_{\R^d}\phi_n(y)\,\d y=1$. This and
	\eqref{eq:sup} together prove
	that $\Upsilon(\beta)$ is the maximum of
	the $\beta$-potential of $f$;
	and the maximum $\beta$-potential is finite if and only if $\Upsilon(\beta)$
	is. We can choose the $\phi_n$'s so that in addition to the preceding
	regularity criteria, every $\phi_n$ is supported in the ball 
	of radius $1/n$ about the origin. In that way we obtain
	\begin{equation}\begin{split}
		\left( \bar R_\beta(f*\phi_n)\right)(0) 
			&= ((\bar R_\beta f)*\phi_n)(0) \\
		&\le \sup_{\|x\|<1/n}(\bar R_\beta f)(x).
	\end{split}\end{equation}
	And this proves \eqref{eq:Up:pot} for every $\beta\in(0\,,\infty)$.
	
	The case that $\beta:=0$ has to be handled separately
	because $\bar R_0$ is not a finite measure. 
	Therefore, we next derive \eqref{eq:Up:pot} in the case that $\beta=0$.
	
	First of all, $\beta\mapsto\Upsilon(\beta)$ is nonincreasing. Therefore,
	$\Upsilon(0):=\lim_{\beta\downarrow 0}\Upsilon(\beta)$ exists
	as a nondecreasing limit. Because $\bar R_\beta f\le \bar R_0 f$ [pointwise]
	for all $\beta>0$, we can deduce that
	\begin{equation}
		\Upsilon(0)\le\mathop{\text{\rm ess sup}}_{x\in\R^d}\,(\bar R_0 f)(x),
	\end{equation}
	and 
	\begin{equation}
		\Upsilon(0)\le\limsup_{x\to 0}\,(\bar R_0f)(x).
	\end{equation}
	
	For the reverse bound, recall \eqref{eq:sup},
	and let $\beta\downarrow 0$. Since $\Upsilon(\beta)\to\Upsilon(0)$,
	we find that for every probability density $\phi\in\mathcal{S}$
	and $x\in\R^d$,
	\begin{equation}
		\lim_{\beta\downarrow 0}\left( \bar R_\beta(f*\phi_n)\right)(x) \le\Upsilon(0).
	\end{equation}
	But the left-hand side is 
	\begin{equation}
		\lim_{\beta\downarrow 0}\E\left[\int_0^\infty (f*\phi_n)(\bar X_s+x)
		\e^{-\beta s}\,\d s\right]=
		\left( \bar R_0(f*\phi_n)\right)(x),
	\end{equation}
	thanks to the monotone convergence theorem. Another application 
	of Fatou's lemma shows that $(\bar R_0 f)(x)\le \Upsilon(0)$
	for all $x\in\R^d$.
	This establishes \eqref{eq:Up:pot},
	and hence the proposition, in the case that $\beta=0$.
\end{proof}

\section{Proof of the Maximum Principle}

The main goal of this subsection is to establish Theorem \ref{th:Dalang:1}.
This subsection also contains a harmonic-analytic estimate that might be
of independent interest.
We will use that harmonic-analytic to demonstrate Theorem \ref{th:Dalang:1},
as well as the subsequent Theorems \ref{th:interm} and 
 \ref{th:interm:asymp}.
 
 In order to motivate our estimate, let us first consider
the important special case that
the correlation function $f$ is
of Riesz type. That is, 
\begin{equation}
	f(z) := \| z\|^{-(d-b)}\qquad\text{for $z\in\R^d$},
\end{equation}
where $0^{-1}:=\infty$. Clearly, $f$ is locally integrable
when $b\in(0\,,d)$, a condition which we assume,
and in fact its Fourier transform is
\begin{equation}\label{eq:Riesz:FT}
		\hat f(\xi) = \frac{\pi^{d/2}2^b \Gamma(b/2)}{\Gamma((d-b)/2)}
		\cdot\|\xi\|^{-b}.
\end{equation}
See Mattila \cite[eq.\ (12.10), p.\ 161]{Mattila}, for example.
Then, it is well known that
\begin{equation}\label{eq:Energies}\begin{split}
	\iint f(x-y)\,\mu(\d x)\,\mu(\d y)&=\iint \frac{\mu(\d x)\,\mu(\d y)}{\|x-y\|^{d-b}}\\
		&=\frac{\pi^{d/2}2^b \Gamma(b/2)}{\Gamma((d-b)/2)}\cdot
		\int_{\R^d}\frac{|\hat\mu(\xi)|^2}{\|\xi\|^b}\,\d\xi\\
	&=\frac{1}{(2\pi)^d}
		\int_{\R^d} |\hat\mu(\xi)|^2 \hat f(\xi)\,\d\xi.
\end{split}\end{equation}
It is easy to guess this famous identity from an informal
application of the Fubini theorem. However, a rigorous derivation
of \eqref{eq:Energies} requires a good deal of effort;
see  Mattila \cite[Lemma 12.12, p.\ 162]{Mattila}, for instance.
In the language of potential theory, the preceding asserts that 
the Riesz-type ``energy'' of the form $\iint \|x-y\|^{-d+b}\,\mu(\d x)\,
\mu(\d y)$ is equal to a constant multiple of the P\'olya--Szeg\H{o}-type
energy $\int_{\R^d}|\hat\mu(\xi)|^2 \|\xi\|^{-b}\,\d\xi$.
The following proposition shows that there is a very general lower bound
that is valid for every correlation function $f$.

\begin{proposition}\label{pr:KX}
	For all Borel probability 
	measures $\mu$ on $\R^d$,
	\begin{equation}\label{eq:KX}
		\iint f(x-y)\,\mu(\d x)\,\mu(\d y) \ge
		\frac{1}{(2\pi)^d}\int_{\R^d} |\hat\mu(\xi)|^2
		\hat{f}(\xi)\,\d\xi.
	\end{equation}
\end{proposition}

\begin{proof}
	We will assume, without incurring any loss in generality,
	that 
	\begin{equation}
		\iint f(x-y)\,\mu(\d x)\,\mu(\d y)<\infty;
	\end{equation}
	for there is nothing to prove otherwise. We also observe that
	\begin{equation}\label{eq:reciprocity}
		\iint f(x-y)\,\mu(\d x)\,\mu(\d y)=\int (f*\mu)\,\d\mu.
	\end{equation}
	This is essentially the famous ``reciprocity theorem'' of classical potential theory.
	
	Because $f*\mu\in L^1(\mu)$, Lusin's theorem implies
	that for all $\epsilon>0$
	there exists a compact set $A_\epsilon$ in $\R^d$ such that:
	\begin{itemize}
		\item[(i)] $\mu(A_\epsilon^c)<\epsilon$; and 
		\item[(ii)] $f*\mu$ is continuous on $A_\epsilon$.
	\end{itemize}
	Let 
	\begin{equation}
		\mu_\epsilon(\bullet) := \mu (\bullet\cap A_\epsilon)
	\end{equation}
	denote the restriction of $\mu$ to $A_\epsilon$,
	and recall the Gaussian densities $\{\phi\}_{n=1}^\infty$
	from \eqref{eq:Gaussian:density}. Because
	\begin{equation}
		\lim_{n\to\infty}\left(f*\phi_{n/2}*\mu\right)=
		f*\mu, \qquad\text{uniformly on $A_\epsilon$,}
	\end{equation}
	it follows from \eqref{eq:reciprocity} that
	\begin{equation}\label{eq:almostthere}\begin{split}
		\iint f(x-y)\,\mu(\d x)\,\mu(\d y) &\ge
			\int_{A_\epsilon} (f*\mu)\,\d\mu\\
		&=\int(f*\mu)\,\d\mu_\epsilon\\
		&=\lim_{n\to\infty}\int (f*\phi_{n/2}*\mu)\,\d\mu_\epsilon\\
		&\ge\limsup_{n\to\infty}\int(f*\phi_{n/2}*\mu_\epsilon)\,\d\mu_\epsilon.
	\end{split}\end{equation}
	Since $\phi_{n/2}=\phi_n*\phi_n$,	
	\begin{equation}\begin{split}
		\iint f(x-y)\,\mu(\d x)\,\mu(\d y) &\ge
			\limsup_{n\to\infty}\int_{\R^d} (f*\mu_{n,\epsilon})(x)
			\mu_{n,\epsilon}(x)\,\d x.
	\end{split}\end{equation}
	where $\mu_{n,\epsilon}(x):=(\phi_n*\mu_\epsilon)(x)$.
	Since $\mu_{n,\epsilon}\in\mathcal{S}$
	and $f*\mu_{n,\epsilon}$ is $C^\infty$ and
	tempered, Lemma \ref{lem:Parseval} implies that 
	\begin{equation}\begin{split}
		\int_{\R^d}(f*\mu_{n,\epsilon})(x)\mu_{n,\epsilon}(x)\,\d x
			&=\frac{1}{(2\pi)^d}\int_{\R^d}\hat f(\xi)
			|\hat\mu_{n,\epsilon}(\xi)|^2\,\d\xi\\
		&=\frac{1}{(2\pi)^d}\int_{\R^d}\e^{-\|\xi\|^2/n}
			|\hat\mu_\epsilon(\xi)|^2\hat f(\xi)\,\d\xi.
	\end{split}\end{equation}
	This, \eqref{eq:almostthere}, and the monotone convergence
	theorem together imply that
	\begin{equation}
		\iint f(x-y)\,\mu(\d x)\,\mu(\d y) \ge \frac{1}{(2\pi)^d}
		\int_{\R^d}|\hat\mu_\epsilon(\xi)|^2\hat f(\xi)\,\d\xi.
	\end{equation}
	Since $\mu_\epsilon$ converges weakly to $\mu$ as $\epsilon\to 0$,
	we know that $\hat\mu_\epsilon$ converges to $\hat\mu$ pointwise.
	The proposition follows from this and Fatou's lemma.
\end{proof}

An elementary computation \cite[Lemma 12.11, p.\ 161]{Mattila} 
and Proposition \ref{pr:KX} above together yield the following corollary. 

\begin{corollary}
	If $f$ is lower semicontinuous, then for all Borel probability
	measures $\mu$ on $\R^d$,
	\begin{equation}
		\iint f(x-y)\,\mu(\d x)\,\mu(\d y) =
		\frac{1}{(2\pi)^d}\int_{\R^d} |\hat\mu(\xi)|^2
		\hat{f}(\xi)\,\d\xi.
	\end{equation}
\end{corollary} 

Even though we do not use the above result in the sequel, we have chosen to include it as it is an interesting fact about energy forms that are based on lower semicontinuous positive-definite functions.   A weaker version of this result has been proved 
by Khoshnevisan and Xiao \cite[Theorem 5.2]{KX:HAAL}.

We are now ready to prove Theorem \ref{th:Dalang:1}.

\begin{proof}[Proof of Theorem \ref{th:Dalang:1}]
	Note that for all $t\ge 0$,
	\begin{equation}\begin{split}
		(\bar P_t f)(0) &= \left( P_t^*P_tf\right)(0)\\
		&=\left( \tilde{p}_t*p_t*f\right)(0)\\
		&=\iint f(x-y) p_t(x)p_t(y)\,\d x\,\d y.
	\end{split}\end{equation}
	This requires only the Tonelli theorem. Now Proposition
	\ref{pr:KX} can be applied to show us that
	\begin{equation}
		(\bar P_t f)(0) \ge \frac{1}{(2\pi)^d}\int_{\R^d}
		\e^{-2t\Re\Psi(\xi)}\hat f(\xi)\,\d\xi.
	\end{equation}
	Multiply both sides by $\exp(-\beta t)$ and integrate
	$[\d t]$ to find that
	\begin{equation}\begin{split}
		(\bar R_\beta f)(0) &\ge \frac{1}{(2\pi)^d}
			\int_{\R^d}\frac{\hat f(\xi)}{\beta+2\Re\Psi(\xi)}\,\d\xi\\
		&=\Upsilon(\beta).
	\end{split}\end{equation}
	This and Proposition \ref{pr:FA} together imply \eqref{eq:0sup}.
	Consequently, \eqref{cond:1} holds if and only if
	$\Upsilon(\beta)<\infty$ for all $\beta>0$. 
	
	Next we prove that $\bar R_\beta f\in C_0(\R^d)$
	whenever $\Upsilon(\beta)<\infty$ and $f$ is lower semicontinuous.
	
	When $f$ is lower semicontinuous we can find compactly-supported continuous
	functions $f_n$ that converge upward to $f$ as $n\to\infty$.
	Recall from \eqref{eq:R:BG} on page \pageref{eq:R:BG}
	that $\bar R_\beta$ maps $C_0(\R^d)$ to $C_0(\R^d)$. 
	Consequently, $\bar R_\beta f_n\in C_0(\R^d)$, and from
	this we may conclude that
	$\bar R_\beta f$ is lower semicontinuous. 
	
	Next, let us define
	\begin{equation}
		\pi_\beta (x) := \frac{1}{(2\pi)^d}\int_{\R^d}\frac{\e^{-i\xi\cdot x}
		\hat f(\xi)}{\beta+2\Re\Psi(\xi)}\,\d\xi.
	\end{equation}
	If $\Upsilon(\beta)<\infty$, then $\pi_\beta\in C_0(\R^d)$.
	Moreover, a few successive applications of Fubini's theorem tell us
	that for all $\phi\in\mathcal{S}$,
	\begin{equation}
		\int_{\R^d} \pi_\beta(x)\phi(x)\,\d x =
		\int_{\R^d} (\bar R_\beta f)(x)\phi(x)\,\d x.
	\end{equation}
	Thus, $\pi_\beta=\bar R_\beta f$ a.e.
	
	It remains to prove that if $\Upsilon(1)$ is finite, then 
	so is $\Upsilon(\beta)$ for every $\beta>0$. 
	
	We have shown that if $\Upsilon(1)$ is finite, 
	then $\bar R_1 f=\pi_1$ almost everywhere,
	$\pi_1\in C_0(\R^d)$; also, $\bar R_1 f$ is bounded [Proposition
	\ref{pr:FA}]. 
	If $h_1=h_2$ almost everwhere
	and $h_1,h_2:\R^d\to\R_+$ are measurable, then
	\begin{equation}\begin{split}
		(\bar R_\beta h_1)(x)  & = \int_{\R^d} \bar r_\beta (y)
			h_1(x-y)\,\d y\\
		&=\int_{\R^d} \bar r_\beta (y) h_2(x-y)\,\d y\\
		&=(\bar R_\beta h_2)(x),
	\end{split}\end{equation}
	for all $x\in\R^d$. Therefore, \eqref{eq:R:BG} on page \pageref{eq:R:BG}
	and the fact that $\pi_1\in C_0(\R^d)$ together imply that
	$\bar R_\beta \pi_1\in C_0(\R^d)$
	for all $\beta>0$. In particular,
	$\bar R_\beta\bar R_1 f\in C_0(\R^d)$---whence $\bar R_\beta\bar R_1 f$
	is bounded---for every
	$\beta>0$. And by the resolvent equation [see
	Blumenthal and Getoor \cite[(8.10), p.\ 41]{BG}],
	\begin{equation}
		(\bar R_\beta f)(x) = (\bar R_1 f)(x)+(1-\beta)(\bar R_\beta 
		\bar R_1 f)(x) \quad\text{for all\ $x\in\R^d$}.
	\end{equation}
	Thus, it follows from the boundedness of
	$\bar R_1 f$ and $\bar R_\beta \bar R_1 f$ [for all $\beta>0$]
	that $\bar R_\beta f$ is bounded for every
	$\beta>0$. Consequently, Proposition \ref{pr:FA} implies the
	finiteness of $\Upsilon(\beta)$ for every $\beta>0$.
\end{proof}

\section{Probabilistic Potential Theory}

Define,
for all measurable functions $\phi:\R^d\to\R_+$, a process
$t\mapsto L_t(\phi)$ as follows:
\begin{equation}
	L_t(\phi):=\int_0^t \phi(\bar X_s)\,\d s
	\qquad\text{for all $t\in[0\,,\infty]$}.
\end{equation}
The random field $(t\,,\phi)\mapsto L_t(\phi)$ defined above is often called the \emph{occupation
field} of the L\'evy process $X$; see, for example, Dynkin
\cite{Dynkin}.
It is well defined, but might well be infinite
even in simple cases. The following example highlights this,
and also paves the way to the ensuing discussion which yields
a probabilistic interpretation of Conditions \ref{cond:1}
and \ref{cond:0}.

\begin{example}[After Girsanov, 1962 \protect{\cite[\S3.10, pp.\ 78--81]{McKean}}]
	Consider $d=1$,
	and let $X$ denote one-dimensional Brownian motion,
	normalized so that
	\begin{equation}
		\E\e^{i\xi\cdot X_t} =\e^{-t\xi^2/4}
		\qquad\text{for all $t>0$ and $\xi\in\R$}.
	\end{equation}
	In this way, $\bar X$ is normalized to
	be standard linear Brownian motion; that is,
	\begin{equation}
		\E\e^{i\xi\cdot \bar{X}_t} =\e^{-t\xi^2/2}
		\qquad\text{for all $t>0$ and $\xi\in\R$}.
	\end{equation}
	Let $\phi(x) := |x|^{-\alpha}$ for all $x\in\R$, where
	$\alpha\in(0\,,1)$ is fixed. By the occupation density
	formula [see Corollary (1.6) of Revuz and Yor \cite[p.\ 209]{RevuzYor}],
	\begin{equation}
		L_t(\phi)=\int_{-\infty}^\infty \frac{\ell^x_t}{|x|^\alpha}\,\d x,
	\end{equation}
	where $\ell$ denotes the process of local times
	associated to $\bar X$. According to Trotter's theorem \cite[Theorem (1.7),
	p.\ 209]{RevuzYor}
	and the occupation density formula,
	\begin{equation}
		\P\left\{ \ell_t^\bullet\in C_0(\R)\text{ for all $t>0$}
		\right\}=1.
	\end{equation}
	It is a well-known consequence of 
	the Blumenthal zero-one law and Brownian scaling that 
	\begin{equation}
		\P\left\{\ell_t^0>0\text{ for all $t>0$}\right\}=1. 
	\end{equation}
	Consequently [if we ignore the null sets in the usual way],
	$L_t(\phi)<\infty$ for some---hence all---$t>0$
	if and only if $\alpha<1$. At the same time, we note that
	\begin{equation}\begin{split}
		\E L_t(\phi) &=\E\int_0^t \frac{\d s}{|\bar X_s|^\alpha}\\
		&=\int_0^t
		\frac{\d s}{s^{\alpha/2}}\cdot\frac{1}{\sqrt{2\pi}}
		\int_{-\infty}^\infty |z|^{-\alpha}\e^{-z^2/2}\,\d z.
	\end{split}\end{equation}
	This requires only Brownian scaling and the Tonelli theorem.
	Thus, $\E L_t(\phi)$ is finite for all $t>0$, if 
	and only if $\alpha<1$. In rough terms, we have shown that
	$L_t(\phi)<\infty$ if and only if $\E L_t(\phi)<\infty$.
	As we shall see, a suitable interpretation of this property
	can be generalized; see Theorem
	\ref{th:PPT} below. 
	\qed
\end{example}

It is convenient to use some notation from Markov-process
theory:
Recall from Markov-process theory that $\P_z$ denotes the law of
the underlying L\'evy process started at $z\in\R^d$
[so that $\P=\P_0$], and $\E_z$ denotes the corresponding expectation operator.
Since the underlying L\'evy process $X$ is L\'evy, $\P_z$ can
be interpretted as the law of $X_\bullet + z$.
Thus, 
\begin{equation}
	(\bar R_\alpha \phi)(z) = \E_z\int_0^\infty \e^{-\alpha s}
	\phi(\bar X_s)\,\d s.
\end{equation}

Before we state and prove the main result of this section,
let us first establish some technical facts which will be needed in the proof
of the main result. 

\begin{lemma}\label{lem:L1}
	For all $t,\alpha>0$ and measurable
	functions $\phi:\R^d\to\R_+$,
	\begin{equation}\label{eq:L1}
		\sup_{x\in\R^d}
		\E_x \left(\sup_{s>0}\left[\e^{-\alpha s}
		L_s(\phi)\right]\right) \le\sup_{z\in\R^d}
		(\bar R_\alpha \phi)(z)
		\le \chi_\alpha(t)\sup_{x\in\R^d}
		\E_x L_t(\phi),
	\end{equation}
	where
	\begin{equation}
		\chi_\alpha(t) := \frac{\e^{\alpha t}}{\e^{\alpha t}-1}.
	\end{equation}
\end{lemma}

\begin{remark}
	Let us point out the following elementary
	bound for the right-most term in the preceding display:
	For all $t,\alpha>0$,
	\begin{equation}
		\e^{-\alpha t}\sup_{x\in\R^d} \E_x L_t(\phi)
		\le \sup_{x\in\R^d}
		\E_x \left(\sup_{s>0}\left[\e^{-\alpha s}
		L_s(\phi)\right]\right).
	\end{equation}
	This shows that the quantities on the two extreme
	ends of \eqref{eq:L1} are one and the same, up to a 
	multiplicative constant that depends on $\alpha$
	and $t$ in an explicit way.
	\qed
\end{remark}

\begin{proof}[Proof of Lemma \ref{lem:L1}]
	Because $\phi$ is nonnegative, we have the sure
	inequality,
	\begin{equation}
		\e^{-\alpha s}L_s(\phi)\le\int_0^\infty \e^{-\alpha r}
		\phi(\bar X_r)\,\d r,
		\qquad\text{valid for all $s,\alpha>0$}.
	\end{equation}
	We take suprema over $s>0$
	and then apply expectations $[\d\P_x]$
	to deduce the first inequality in \eqref{eq:L1}.
	For the second bound, let us note that for all $\alpha,t>0$
	and $z\in\R^d$,
	\begin{equation}\begin{split}
		(\bar R_\alpha \phi)(z) &= \sum_{n=0}^\infty \E_z
			\int_{nt}^{(n+1)t}\e^{-\alpha s} \phi(\bar X_s)\,\d s\\
		&\le \sum_{n=0}^\infty \e^{-\alpha nt}\E_z\int_0^t
			\phi(\bar X_{s+nt})\,\d s.
	\end{split}\end{equation}
	This implies the second inequality, because
	\begin{equation}\begin{split}
		\E_z\int_0^t\phi(\bar X_{s+nt})\,\d s&=
			\E_z\E_{X_{nt}}\int_0^t\phi(\bar X_s)\,\d s\\
		&\le\sup_{x\in\R^d} \E_x\left[ L_t(\phi)\right],
	\end{split}\end{equation}
	in accord with the Markov property.
\end{proof}

\begin{lemma}\label{lem:L2}
	For all $t>0$ and measurable $\phi:\R^d\to\R_+$,
	\begin{equation}
		\sup_{x\in\R^d}
		\E_x\left(\left| L_t(\phi)\right|^2\right) 
		\le 2\left(\sup_{z\in\R^d}\E_z\left[ L_t(\phi)\right]\right)^2.
	\end{equation}
\end{lemma}

\begin{proof}
	We can write
	\begin{equation}
		\E_x\left(\left| L_t(\phi)\right|^2\right) =2
		\int_0^t\d u\int_u^t\d v\ \E_x\left(
		\phi(\bar X_u)\phi(\bar X_v)
		\right).
	\end{equation}
	Since
	\begin{equation}\begin{split}
		\E_x\left(\left.\int_u^t \phi(\bar X_v)\,\d v\ \right|\,
			\bar X_s;\,s\le u \right)
			&=\E_{X_u}\int_0^{t-u}\phi(\bar X_s)\,\d s\\
		&=\E_{X_u} \left[ L_{t-u}(\phi)\right]
			\qquad\text{$\P_x$-a.s.,}
	\end{split}\end{equation}
	it follows that
	\begin{equation}\begin{split}
		\E_x\left(\left| L_t(\phi)\right|^2\right)
			&=2\E_x\int_0^t\phi(\bar X_u)\E_{X_u}\left[ L_{t-u}(\phi)
			\right]\,\d u\\
		&\le2\E_x\int_0^t\phi(\bar X_u)\cdot
			\sup_{z\in\R^d}\E_z\left[ L_{t-u}(\phi)\right]\,\d u.
	\end{split}\end{equation}
	The lemma follows since $L_{t-u}(\phi)\le L_t(\phi)$.
\end{proof}
The following constitutes the third, and final, technical 
lemma of this section.
\begin{lemma}\label{lem:L3}
	For all $\alpha,t>0$,
	\begin{equation}
		\sup_{x\in\R^d}\P_x\left\{
		L_t(f) \ge \frac{(\bar R_\alpha f)(0)}{2\chi_\alpha(t)}
		\right\} \ge \frac18,
	\end{equation}
	where $\chi_\alpha(t)$ is defined in Lemma \ref{lem:L1}.
\end{lemma}

\begin{proof}
	Recall the \emph{Paley--Zygmund inequality} \cite{PaleyZygmund}:
	If $Z$ is a nonnegative random variable in $L^2(\P)$
	with $\E Z>0$, then
	\begin{equation}\label{eq:PZ}
		\P\left\{ Z\ge \frac12\E Z\right\} \ge
		\frac{(\E Z)^2}{4\E(Z^2)}.
	\end{equation}
	We can apply this with $Z:=L_t(\phi)$---where
	$\phi:\R^d\to\R_+$ is bounded away from zero and infinity---to see that
	for all $z\in\R^d$,
	\begin{equation}
		\sup_{x\in\R^d}\P_x\left\{ L_t(\phi) \ge \frac12 \E_z L_t(\phi)
		\right\}\ge\frac{\left( \E_z L_t(\phi) \right)^2}{%
		4\E_z\left( |L_t(\phi)|^2\right)}.
	\end{equation}
	By selecting $z$ appropriately, we can ensure that
	\begin{equation}
		\E_z L_t(\phi)\ge
		(1-\epsilon)\cdot\sup_{x\in\R^d}\E_x L_t(\phi),
	\end{equation}
	where $\epsilon\in(0\,,1)$ is arbitrary but fixed.
	Let $\epsilon\downarrow 0$ and appeal to the continuity
	properties of probability measures to deduce that
	\begin{equation}\begin{split}
		\sup_{x\in\R^d}\P_x\left\{ L_t(\phi) \ge \frac12 
			\sup_{z\in\R^d}\E_z L_t(\phi)
			\right\}
			&\ge\frac{\left( \sup_{z\in\R^d}\E_z L_t(\phi) \right)^2}{%
			4\sup_{x\in\R^d}\E_x\left( |L_t(\phi)|^2\right)}\\
		&\ge\frac18;
	\end{split}\end{equation}
	see Lemma \ref{lem:L2} for the last inequality.
	A monotone-class argument shows that 
	the preceding holds true for all bounded and measurable functions
	$\phi\not\equiv 0$. We apply it with $\phi_N:=\min(f\,,N)$ in place
	of $\phi$, where $N\ge 1$
	is fixed. In this way we obtain
	\begin{equation}
		\sup_{x\in\R^d}\P_x\left\{ L_t(f) \ge \frac12 
		\sup_{z\in\R^d}\E_z L_t(\phi_N)
		\right\}\ge\frac18.
	\end{equation}
	This and Lemma \ref{lem:L1} together tell us that
	\begin{equation}\begin{split}
		&\sup_{x\in\R^d}\P_x\left\{ L_t(f) \ge
			\frac{(\bar R_\alpha \phi_N)(0)}{2\chi_\alpha(t)}
			\right\}\\
		&\hskip1.2in\ge \sup_{x\in\R^d}\P_x\left\{ L_t(f) \ge
			\sup_{z\in\R^d}\frac{(\bar R_\alpha \phi_N)(z)}{2\chi_\alpha(t)}
			\right\}\\
		&\hskip1.2in\ge\frac18.
	\end{split}\end{equation}
	As $N\uparrow\infty$, $(\bar R_\alpha\phi_N)(0)\uparrow
	(\bar R_\alpha f)(0)$, and the lemma follows.
\end{proof}

The next result yields a probabilistic characterization of
Condition \ref{cond:1} which we recall for the reader's convenience.

\begin{equation*}
	(\bar{R}_\alpha f)(0)<\infty
	\quad\text{for all $\alpha>0$}.
\end{equation*}

\begin{theorem}\label{th:PPT}
	Under Condition \ref{cond:1}, 
	\begin{equation}\label{PPT1}
		\P_z\left\{L_t(f)<\infty\text{ for all $t>0$}\right\}=1
		\qquad\text{for all $z\in\R^d$}.
	\end{equation}
	Moreover, in this case, $t\mapsto L_t(f)$ grows subexponentially.
	That is,
	\begin{equation}\label{PPT2}
		\P_z\left\{\limsup_{t\to\infty} \frac{\log L_t(f)}{t}\le0
		\right\}=1
		\qquad\text{for all $z\in\R^d$}.
	\end{equation}
	On the other hand, if Condition \ref{cond:1} fails to hold, then
	\begin{equation}
		\P_z\left\{ L_t(f)<\infty\text{ for some $t>0$}\right\}=0
		\qquad\text{for some $z\in\R^d$}.
	\end{equation}
\end{theorem}

\begin{remark}
	Consider the stochastic heat equation where
	$\dot F$ is space-time white noise. Formally
	speaking, this means that $f:=\delta_0$
	is our correlation ``function.'' In this case, one
	can [again formally] interpret 
	\begin{equation}
		L_t(f) = L_t(\delta_0) = \int_0^t \delta_0(\bar X_s)\,\d s
	\end{equation}
	as the local time of the replica process $\bar X$ at zero.
	And if we interpret Theorem \ref{th:PPT} loosely as well,
	then Theorem \ref{th:existence} suggests that
	\eqref{heat} has a mild solution if and only if $\bar X$ has
	local times. This interpretation is correct, as well
	as easy to check, and leads to deeper connections between SPDEs
	driven by space-time white noise on one hand and local-time
	theory on the other hand \cite{EFK,FKN}. 
	In the case of the parabolic Anderson model [that is, \eqref{heat}
	with $\sigma(u)=\text{const}\cdot u$ and $b\equiv 0$], Bertini and Cancrini
	\cite{BertiniCancrini} and Hu and Nualart \cite{HuNualart} discuss other
	closely-related connections to local times. In the case that $x$ is a discrete
	variable [for example, because $\mathcal{L}$ is the generator of a L\'evy process
	on $\mathbf{Z}^d$; i.e., the generator of a continuous-time random walk],
	similar connections were found earlier; see Carmona and Molchanov
	\cite{CarmonaMolchanov:94}, for instance.
	\qed
\end{remark}


\begin{proof}[Proof of Theorem \ref{th:PPT}]
	If $(\bar R_\alpha f)(0)<\infty$ for some $\alpha>0$,
	then $(\bar R_\beta f)(0)<\infty$ for all $\beta>0$
	by Theorem \ref{th:Dalang:1}. It follows from
	the first inequality of Lemma \ref{lem:L1},
	and Theorem \ref{th:Dalang:1}, that for all $\beta>0$
	and $z\in\R^d$,
	\begin{equation}
		\E_z \left[ \sup_{t>0}\left(\e^{-\beta t}
		L_t(f)\right) \right] \le (\bar R_\beta f)(0)<\infty.
	\end{equation}
	This implies \eqref{PPT1}; it also implies that
	\begin{equation}
		\limsup_{t\to\infty} \left[ \e^{-t\beta}L_t(f)\right]<\infty
		\qquad\text{almost surely [$\P_z$]}.
	\end{equation}
	This implies \eqref{PPT2}
	because $\beta>0$ and $z\in\R^d$ are arbitrary. 
	
	In order to finish the proof, let us consider the remaining case
	that $(\bar R_\alpha f)(0)=\infty$ for all $\alpha>0$. 
	
	According to Lemma \ref{lem:L3},
	\begin{equation}
		\sup_{x\in\R^d}\P_x\left\{ L_t(f) =\infty\right\}\ge \frac18
		\qquad\text{for all $t>0$}.
	\end{equation}
	In particular, there exists $z\in\R^d$ such that
	\begin{equation}
		\P_z\left\{ L_t(f) =\infty\text{ for some $t>0$}\right\}\ge \frac19.
	\end{equation}
	Because $t\mapsto L_t(f)$ is nondecreasing, the Blumenthal zero-one
	law applies and implies that
	\begin{equation}
		\P_z\left\{ L_t(f)=\infty\text{ for some $t>0$}\right\}
		=1;
	\end{equation}
	this implies the remaining portion of the theorem.
\end{proof}

We now have the following  consequence of
Theorem \ref{th:PPT}.  It is particularly useful because its hypothesis are verified
by all the examples that we have mentioned in the Introduction.

\begin{corollary}
	Suppose $f$ is bounded uniformly on the complement
	of every open neighborhood of the origin. Then,
	Condition \ref{cond:1} is equivalent to the following:
	$\P\{L_t(f)<\infty\text{ for some $t>0$}\}=1$.
\end{corollary}

\begin{proof}
	According to Theorem \ref{th:PPT}, if Condition \ref{cond:1}
	holds then
	\begin{equation}
		\P_z\left\{ L_t(f)<\infty\text{ for some $t>0$}
		\right\} =1\qquad\text{for all $z\in\R^d$.}
	\end{equation}
	Set $z:=0$ to obtain half of the corollary.
	
	Conversely, suppose Condition \ref{cond:1} fails. According to 
	Theorem \ref{th:PPT}, there exists a point $z\in\R^d$
	such that
	\begin{equation}\label{eq:bad:z}
		\P_z\left\{ L_t(f)<\infty\text{ for some $t>0$}
		\right\}=0.
	\end{equation}
	We need to prove that $z=0$. This holds because if $z$
	were not equal to the origin, 
	then
	\begin{equation}
		\P_z\left\{ L_t(f)<\infty\text{ for all $t\in[0\,,\tau)$}
		\right\}=1,
	\end{equation}
	where $\tau$ denotes the first hitting time of the open ball
	of radius $\|z\|/2$ around $0$. Indeed, 
	\begin{equation}
		\sup_{0\le t<\tau} L_t(f)\le \tau 
		\cdot\sup_{\|u\|\ge\|z\|/2}
		f(u)<\infty,
	\end{equation}
	$\P_z$-almost surely. Since the paths of $X$ are
	right-continuous, $\P_z\{\tau>0\}=1$, and hence \eqref{eq:bad:z}
	is contradicted.
\end{proof}

Condition \ref{cond:0} [p.\ \pageref{cond:0}] 
also has a probabilistic interpretation
that is given by  following proposition.

\begin{proposition}\label{pr:transient}
	If $(\bar R_\alpha f)(0)<\infty$ for
	some, hence all, $\alpha>0$, then:
	\begin{equation}
		(\bar R_0f)(0)<\infty
		\qquad\Longrightarrow\qquad
		L_\infty(f)<\infty\quad\text{a.s.};
	\end{equation}
	and 
	\begin{equation}
		(\bar R_0 f)(0)=\infty\qquad\Longrightarrow\qquad
		L_\infty(f)=\infty\quad\text{a.s.}
	\end{equation}
\end{proposition}

\begin{proof}
	If $(\bar R_0 f)(0)<\infty$, then because
	$(\bar R_0f)(0)=\E L_\infty(f)<\infty$, it follows that $L_\infty(f)$
	is finite a.s.
	
	If, on the other hand,
	$(\bar R_0f)(0)=\infty$, then because
	\begin{equation}
		\E\int_0^\infty\e^{-\alpha s}f(\bar X_s)\,\d s=
		(\bar R_\alpha f)(0), 
	\end{equation}
	the Paley--Zygmund inequality 
	\eqref{eq:PZ} implies that
	\begin{equation}
		\P\left\{\int_0^\infty \e^{-\alpha s}f(\bar X_s)\,\d s
		\ge \frac12(\bar R_\alpha f)(0)\right\}
		\ge \frac{\left| (\bar R_\alpha f)(0)\right|^2}{
		4\E\left(\left|\int_0^\infty \e^{-\alpha s}f(\bar X_s)\,\d s\right|^2\right)}.
	\end{equation}
	It follows from this and Lemma \ref{lem:L2}---see also the
	proof of Lemma \ref{lem:L3}---that
	\begin{equation}\begin{split}
		\P\left\{\int_0^\infty f(\bar X_s)\,\d s
			\ge \frac12(\bar R_\alpha f)(0)\right\}&
			\ge \frac{(\bar R_\alpha f)(0)}{8
			\sup_{x\in\R^d} (\bar R_\alpha f)(x)}\\
		&=\frac18,
	\end{split}\end{equation}
	owing to Theorem \ref{th:Dalang:1}
	[p.\ \pageref{th:Dalang:1}]. Let $\alpha\downarrow 0$ to find
	that 
	\begin{equation}\begin{split}
		\mathcal{P} &:= \P\left\{\int_0^\infty f(\bar X_s)\,\d s
			=\infty\right\}\\
		&\ge \frac18.
	\end{split}\end{equation}
	But according to Theorem \ref{th:PPT}, 
	$\int_0^Tf(\bar X_s)\,\d s<\infty$, for all $T>0$
	a.s., because $(\bar R_\alpha f)(0)<\infty$ for some
	[hence all] $\alpha>0$. This implies that
	\begin{equation}
		\mathcal{P} = \P\left\{\lim_{T\to\infty}\int_T^\infty f(\bar X_s)\,\d s
		=\infty\right\}.
	\end{equation}
	That is: (i) $\mathcal{P}$ is the probability of a tail event;
	and (ii) $\mathcal{P}$ is strictly positive, in fact $\mathcal{P}\ge1/8$.
	By the Hewitt--Savage zero-one law \cite{HewittSavage}, $\mathcal{P}=1$.
\end{proof}

\section{A Final Observation}
Let us conclude this chapter with an observation that
will be used later on in Theorem \ref{theorem:counter}
[p.\ \pageref{theorem:counter}] in order 
to produce a stochastic PDE whose
random-field solution exists but is discontinuous densely.

Let $\bm{X}:=\{X_t\}_{t\ge 0}$ denote a L\'evy process
on $\R^d$ with characteristic exponent $\Psi$. Recall that
$\bm{X}$ has a \emph{one-potential density} $v$
if $v$ is a probability density on $\R^d$ that satisfies the following
for all Borel-measurable functions $\phi:\R^d\to\R_+$:
\begin{equation}\label{eq:pot:density}
	\E\left[\int_0^\infty \e^{-s} \phi(X_s)\,\d s\right] =\int_{\R^d}
	\phi(x) v(x)\,\d x.
\end{equation}
Because $\phi\ge 0$,
the preceding expectation commutes with the $\d s$-integral.
Recall that $m_s$ denotes the law of $X_s$, and restrict attention to
only nonnegative $\phi\in\mathcal{S}$. In that case,
\begin{equation}\begin{split}
	\int_0^\infty \e^{-s} \E\phi(X_s)\, \d s &=\int_0^\infty
		\e^{-s}\left(\int\phi\,\d m_s\right)\,\d s\\
	&= \frac{1}{(2\pi)^d}\int_0^\infty \e^{-s} \d s \int_{\R^d}
		\d\xi \ \overline{\hat\phi(\xi)}\
		\e^{-s\Psi(\xi)}\\
	&= \frac{1}{(2\pi)^d}\int_{\R^d} \frac{\overline{\hat\phi(\xi)}}{
		1+\Psi(\xi)}\,\d\xi.
\end{split}\end{equation}
We compare this to the right-hand side of \eqref{eq:pot:density},
and then apply Plancherel's theorem to the latter, to deduce the 
following well-known formula:
\begin{equation}
	\hat{v}(\xi) = \frac{1}{1+\Psi(\xi)}\qquad
	\text{for all $\xi\in\R^d$}.
\end{equation}

If we consider only the case that $\bm{X}$ is symmetric,
then $\hat v$ is rendered nonnegative,
since $\Psi$ is nonnegative in this case. This observation and the 
Bochner--Schwartz theorem [Theorem \ref{BochnerSchwartz}]
together imply that $v$ is a correlation function. Because 
products---and hence integer powers---of correlation functions are themselves
correlation functions, Theorem \ref{th:Levy:asymp} yields the
following byproduct.

\begin{theorem}\label{th:nice:cor}
	Choose and fix $a>0$ and $b\in\R$. Then,
	there exists a correlation function $v$ on $\R^d$ such that
	\begin{equation}
		\hat{v}(\xi) \asymp \frac{1}{\|\xi\|^a(\log\|\xi\|)^b}
		\qquad\text{for $\xi\in\R^d$ with $\|\xi\|>\e$}.
	\end{equation}
\end{theorem}

%% file: FoonKhosh-Lin.tex
\chapter{The Linear Equation}
\label{ch:linear}

Before we study the fully nonlinear equation \eqref{heat}, we
analyse the far simpler linearized form of the same equation
$[\sigma\equiv 1,\ b\equiv 0]$, and show that
it has many interesting features of its own. Because the solutions,
if any, to the said linear equations can only be Gaussian random
fields, we are able to use the theory of Gaussian processes in order
to produce some definitive existence and regularity results. 
Our results should be
compared with our earlier joint effort with Eulalia Nualart
\cite{FKN}, in which $\dot{F}$ was space-time white noise.
Our earlier effort was, in turn, motivated strongly by the earlier works
of Dalang and Frangos \cite{DalangFrangos},
Dalang \cite{Dalang}, and Peszat and Zabcyzk \cite{PeszatZabczyk}.

\section{Existence and uniqueness}

The linearized form of \eqref{heat} is the stochastic PDE
\begin{equation}\label{heat:lin}
	\frac{\partial}{\partial t}
	u_t(x) = (\sL u_t)(x) + \dot{F}_t(x),
\end{equation}
subject to $u_0$ being the initial function, which as mentioned in
the Introduction is assumed to be a nonrandom bounded
and measurable function
$u_0:\R^d\to\R$.
One can follow through the theory of Walsh, and define
the \emph{weak solution} to \eqref{heat:lin} as the Gaussian
random field $\bm{u}:=\{u_t(\phi)\}_{t> 0,\phi\in\mathcal{S}}$,
where [we recall] $\mathcal{S}$ denotes the collection of all rapidly-decreasing
test functions on $\R^d$, and
\begin{equation}
	u_t(\phi) = \int_{\R^d} u_0(x)(P_t^*\phi)(x)\,\d x +
	\int_0^t\int_{\R^d} \left( p_{t-s}*\phi\right)(y)\, F(\d s\,\d y).
\end{equation}
The double integral is a Wiener integral and
$\{P_t^*\}_{t\ge 0}$ is the semigroup associated
to the dual process $X^*_t:=-X_t$ $[t\ge 0]$.

There are two main questions that one needs to answer before one proceeds
further: 
\begin{enumerate}
\item[(a)] Is $\bm{u}$ well defined?
\item[(b)] What is the largest family of $\phi$'s for which $u_t(\phi)$ is well defined?
\end{enumerate}
An affirmative answer to the first question would imply existence of 
solutions in the general sense of Walsh \cite{Walsh}.
Since the analysis of the nonrandom quantity $\int_{\R^d} u_0(x)(P_t^*\phi)(x)\,\d x$
is standard, we can reduce our problem to the special case that $u_0\equiv 0$.
In that case, these question are addressed by the following estimate.
Here and throughout, we define
\begin{equation}\label{eq:mathcalE}
	\mathcal{E}_\lambda(v) := \frac{1}{2(2\pi)^d}\int_{\R^d}
	\frac{|\hat v(\xi) |^2\hat{f}(\xi)}{\lambda^{-1}+\Re\Psi(\xi)}\,\d\xi
\end{equation}
for all Schwartz distributions $v$ whose Fourier transform is a function.

\begin{lemma}\label{lem:linear:L2:est}
	The weak solution $\bm{u}$ to \eqref{heat:lin} with $u_0\equiv 0$
	exists as a well-defined Gaussian random field parametrized by
	$t>0$ and $\phi\in\mathcal{S}$. Moreover, for all $t,\lambda>0$
	and $\phi\in\mathcal{S}$,
	\begin{equation}
		a(t)\mathcal{E}_\lambda(\phi)
		\le\E\left(\left| u_t(\phi) \right|^2\right)
		\le b(t)\mathcal{E}_\lambda(\phi),
	\end{equation}
	where $a(t):=(1-\e^{-2t/\lambda})$ and $b(t):=\e^{2t/\lambda}$.
\end{lemma}

In fact, Lemma \ref{lem:linear:L2:est} holds under far greater generality than
the one presented here. For instance, it holds even when transition functions
do not necessarily exist, and when the correlation function is a general
correlation measure. The less general formulation above suffices for our needs.

\begin{proof}
	If $\phi\in\mathcal{S}$ then $\hat{p_t}\hat{\phi}\in \mathcal{S}$
	for all $t\ge 0$.
	Since the Fourier transform is an isometry on $\mathcal{S}$,
	this proves that
	\begin{equation}
		P_t^*\phi=p_t\ast \phi\in \mathcal{S}\quad
		\text{for every $t\ge 0$}. 
	\end{equation}
	Therefore, in accord with \eqref{eq:Cov}, the second moment
	$\E(|u_t(\phi)|^2)$ is equal to
	\begin{equation}\begin{split}
		&\E\left(\left|\int_0^t\int_{\R^d} \left( P_{t-s}^*\phi\right)(y)\, F(\d s\,\d y)
			\right|^2\right)\\
		&\hskip1.3in= \frac{1}{(2\pi)^d}\int_0^t\d s\int_{\R^d}\d\xi\ \left|
			\e^{-(t-s)\Psi(\xi)}\right|^2\cdot  |\hat\phi(\xi) |^2\hat{f}(\xi)\\
		&\hskip1.3in= \frac{1}{(2\pi)^d}\int_0^t\d s\int \d\xi\
			\e^{-2s\Re\Psi(\xi)}\cdot  |\hat\phi(\xi) |^2\hat{f}(\xi).
	\end{split}\end{equation}
	And the lemma follows from the preceding and Lemma 3.5 of \cite{FKN}.
\end{proof}

There are standard ways to extend the domain of Gaussian random fields.
In our case, we proceed as follows:
Consider the pseudo-distances $\{\rho_t\}_{t>0}$ defined by
\begin{equation}
	\rho_t(\phi\,,\psi) :=
	\left\{ \E\left( \left| u_t(\phi)-u_t(\psi)
	\right|^2\right) \right\}^{1/2}
	\qquad\text{for $\phi,\psi\in\mathcal{S}$}.
\end{equation}
Because $L^2(\P)$-limits
of Gaussian random fields are themselves Gaussian random fields,
we deduce the following: Suppose $v$ is a Schwartz distribution
such that $\lim_{n\to\infty}\rho_t(v\,,v*\phi_n)=0$ for all $t>0$,
where $\{\phi_n\}_{n=1}^\infty$ is the sequence of
Gaussian densities, as defined in \eqref{eq:Gaussian:density}
[p.\ \pageref{eq:Gaussian:density}].
Then $u_t(v)$ is well defined in $L^2(\P)$, 
and the totality $\{u_t(v)\}$ of all such random variables 
forms a Gaussian random field. 

We follow \cite{FKN} and say that
\eqref{heat:lin} has a \emph{random-field solution} if we can obtain
$u_t(\delta_x)$ in this way for all $x\in\R^d$. 

Consider the space $\mathcal{H}_0$ of all $\phi\in\mathcal{S}$ such
that $\mathcal{E}_1(\phi)<\infty$ for all $t>0$. Evidently,
$\mathcal{H}_0$ can be metrized, using the distance
\begin{equation}
	\delta(\phi\,,\psi) := \sqrt{\mathcal{E}_1(\phi-\psi)}
	\qquad\text{for $\psi,\phi\in\mathcal{S}$}.
\end{equation}
Define $\mathcal{H}_1$ to be the completion of $\mathcal{H}_0$ in
the distance $\delta$. 

\begin{lemma}
	Condition \ref{cond:1} holds iff $\delta_x\in\mathcal{H}_1$ for
	some, hence all, $x\in\R^d$.
\end{lemma}

\begin{proof}
	First of all, we recall \eqref{eq:Upsilon} and check that
	\begin{equation}\begin{split}
		\mathcal{E}_\lambda(\delta_x)&=
			\frac{1}{2(2\pi)^d}\int_{\R^d}
			\frac{\d\xi}{\lambda^{-1}+2\Re\Psi(\xi)}\\
		&=\frac12\Upsilon(1/\lambda).
	\end{split}\end{equation}
	In particular, the value of $\mathcal{E}_\lambda(\delta_x)$ does
	not depend on $x\in\R^d$. And
	Theorem \ref{th:Dalang:1} implies that $\mathcal{E}_\lambda(\delta_x)$
	is finite for some $\lambda>0$ if and only if it is finite for all $\lambda>0$.
	
	Let us first suppose that $\mathcal{E}_\lambda(\delta_x)$
	is finite. We can note that $\delta_x*\phi_n=\phi_n(\bullet-x)\in\mathcal{S}$,
	where $\{\phi_n\}_{n=1}^\infty$ was defined in \eqref{eq:Gaussian:density}.
	Therefore, for all $n,m\ge 1$,
	\begin{equation}\label{eq:E:phi}\begin{split}
		&\mathcal{E}_\lambda(\delta_x*\phi_n - \delta_x*\phi_m)\\
		&\hskip.8in=\frac{1}{2(2\pi)^d} \int_{\R^d}
			\frac{1}{\lambda^{-1}+2\Re\Psi(\xi)}
			\left| 1-\e^{-\|\xi\|^2\left|\frac{1}{2n}-\frac{1}{2m}\right|}
			\right|^2\,\d\xi.
	\end{split}\end{equation}
	Since $\mathcal{E}_\lambda(\delta_x)$
	is finite, the dominated convergence theorem
	tells us that the sequence $\{\delta_x*\phi_n\}_{n=1}^\infty$ is Cauchy
	in $\mathcal{H}_0$. A calculation similar to the preceding
	shows that the quantity 
	$\mathcal{E}_\lambda(\delta_x*\phi_n-\delta_x)$
	converges to zero as $n\to\infty$. And therefore, $\delta_x\in\mathcal{H}_1$.
	
	Conversely, if $\delta_x\in\mathcal{H}_1$, then
	$\mathcal{E}_\lambda(\delta_x*\phi_n-\delta_x*\phi_m)\to 0$
	as $n,m\to\infty$. We can extract an unbounded subsequence
	$n_1\le n_2\le \cdots$ of positive integers such that
	\begin{equation}\label{eq:binary}
		\mathcal{E}_\lambda(\delta_x*\phi_{n_j}-\delta_x*\phi_{n_{j+1}})\le
		2^{-j}\quad\text{for all $j\ge 1$}.
	\end{equation}
	It follows from \eqref{eq:E:phi} that
	if $k^{-1}\le |n_j^{-1}-n_{j+1}^{-1}|$, then
	\begin{equation}\begin{split}
		\mathcal{E}_\lambda(\delta_x-\delta_x*\phi_k)
			&=\frac{1}{2(2\pi)^d}\int_{\R^d}\frac{1}{\lambda^{-1}+
			2\Re\Psi(\xi)}\left| 1 - \e^{-\|\xi\|^2/(2k)}\right|^2\\
		&\le\mathcal{E}_\lambda(\delta_x*\phi_{n_j}-\delta_x*\phi_{n_{j+1}})\\
		&\le 2^{-j}.
	\end{split}\end{equation}
	Let $k\to\infty$ and then $j\to\infty$, in this order,
	to deduce from the preceding discussion
	that
	\begin{equation}
		\lim_{k\to\infty}\mathcal{E}_\lambda
		(\delta_x-\delta_x*\phi_k) =0.
	\end{equation}
	Because $v\mapsto
	\sqrt{\mathcal{E}_\lambda(v)}$ satisfies the triangle inequality,  it
	follows from \eqref{eq:binary} that for all $k\ge 2$,
	\begin{equation}\begin{split}
		\sqrt{\mathcal{E}_\lambda(\delta_x)} &\le
			\sqrt{\mathcal{E}_\lambda(\delta_x-\delta_x*\phi_k)} 
			+\sqrt{\mathcal{E}_\lambda(\delta_x*\phi_k)}\\
		&\le\sqrt{\mathcal{E}_\lambda(\delta_x-\delta_x*\phi_k)} 
			+\sqrt{\mathcal{E}_\lambda(\delta_x*\phi_k-\delta_x*\phi_{k-1})}\\
		&\hskip2.6in + \sqrt{\mathcal{E}_\lambda(\delta_x*\phi_{k-1})}\\
		&\le \sqrt{\mathcal{E}_\lambda(\delta_x-\delta_x*\phi_k)} 
			+ 2^{-(k-1)/2}+\sqrt{\mathcal{E}_\lambda(\delta_x*\phi_{k-1})}\\
		&\ \vdots\\
		&\le \sqrt{\mathcal{E}_\lambda(\delta_x-\delta_x*\phi_k)} +
			\sum_{j=1}^{k-1} 2^{-j/2} + \sqrt{\mathcal{E}_\lambda(\delta_x*\phi_1)}.
	\end{split}\end{equation}
	We have shown that the first quantity on the right-hand side converges to zero as $k\to\infty$;
	and the second term remains bounded. Finally, the third quantity on
	the right-hand side of the preceding is finite since $\phi_1\in\mathcal{S}$.
	Therefore, it follows that if $\delta_x\in\mathcal{H}_1$
	then $\mathcal{E}_\lambda(\delta_x)<\infty$. This concludes our proof.
\end{proof}


For more general
initial functions $u_0 \ge 0$, \eqref{heat:lin} has a random-field
solution if and only if Condition
\ref{cond:1} holds and $(P_tu_0)(x)<\infty$
for all $t>0$ and $x\in\R^d$. Let us conclude this section with a
lemma that provides simple conditions that ensure that $(P_tu_0)(x)$
is finite for all $t>0$ and $x\in\R^d$.

\begin{lemma}\label{lem:p_t:lin}
	Suppose $\exp(-\Re\Psi)\in L^t(\R^d)$ for all $t>0$, and
	$u_0\in L^\beta(\R^d)$ for some $\beta\in[1\,,\infty]$. Then,
	$(P_tu_0)(x)<\infty$ for all $t>0$ and $x\in\R^d$.
	Moreover, $P_tu_0$ is uniformly bounded and continuous for every
	fixed $t>0$.
\end{lemma}

\begin{proof}
	Since $P_t$ is a contraction on $L^\infty(\R^d)$, it suffices
	to consider only the case that $1\le\beta<\infty$.
	
	Choose and fix some $t>0$. According to Young's inequality,
	\begin{equation}\begin{split}
		\|P_tu_0\|_{L^\infty(\R^d)}
			&=\|\tilde{p}_t*u_0\|_{L^\infty(\R^d)}\\
		&\le \| p_t\|_{L^p(\R^d)}\cdot \|u_0\|_{L^q(\R^d)},
	\end{split}\end{equation}
	where $p^{-1}+q^{-1}=1$. On the other hand, for all $p\in(1\,,\infty)$,
	\begin{equation}
		\|p_t\|_{L^p(\R^d)} \le \|p_t\|_{L^\infty(\R^d)},
	\end{equation}
	and this is finite, thanks to Proposition \ref{pr:hawkes} on page
	\pageref{pr:hawkes}.
	Therefore, it remains to prove continuity.
	
	First consider the case that $\beta<\infty$.
	In that case, we can bound the quantity
	\begin{equation}
		\left| (P_tu_0)(x)-(P_tu_0)(x')\right|
		=\left| (\tilde{p}_t*u_0)(x)-(\tilde{p}_t*u_0)(x')\right|,
	\end{equation}
	from above, by
	\begin{equation}\begin{split}
		&\int_{\R^d} p_t(y) \left|
			u_0(y-x)-u_0(y-x')\right|\,\d y\\
		&\hskip1.3in\le \left(\int_{\R^d} p_t(y) \left|
			u_0(y-x)-u_0(y-x')\right|^\beta\,\d y\right)^{1/\beta}\\
		&\hskip1.3in\le\|p_t\|_{L^\infty(\R^d)}^{1/\beta}\cdot
			\left\| u_0(\bullet-x)-u_0(\bullet-x')\right\|_{L^\beta(\R^d)}.
	\end{split}\end{equation}
	It is a
	classical fact that $u_0\in L^\beta(\R^d)$ implies that
	$u_0$ is continuous in $L^\beta(\R^d)$. Therefore,
	$P_tu_0=\tilde{p}_t*u_0$ is continuous.
	
	Next let us consider the case that $\beta=\infty$.
	In that case, we write
	\begin{equation}\begin{split}
		&\left| (\tilde{p}_t*u_0)(x) - (\tilde{p}_t*u_0)(x')\right|\\
		&\hskip1.2in\le \|u_0\|_{L^\infty(\R^d)}\cdot
			\left\| p_t(\bullet-x)-p_t(\bullet-x')\right\|_{L^1(\R^d)},
	\end{split}\end{equation}
	which goes to zero because, once again, $p_t\in L^1(\R^d)$ implies that
	$p_t$ is continuous in $L^1(\R^d)$.
\end{proof}

\section{Spatial regularity: Examples}\label{Spatial regularity: Examples}
Define for all $x,y\in\R^d$,
\begin{equation}
	d(x\,,y) := \left(\frac{1}{(2\pi)^d}\int_{\R^d}
	\frac{1-\cos(\xi\cdot(x-y))}{1+2\Re\Psi(\xi)}\hat{f}(\xi)\,
	\d\xi\right)^{1/2}.
\end{equation}
Then, $d$ defines a pseudo-distance on $\R^d$.
Let $N_d$ denote the \emph{metric entropy} of $[0\,,1]^d$.
That is, for all $\epsilon>0$, $N_d(\epsilon)$ denotes the
minimum number of radius-$\epsilon$ $d$-balls required to
cover $[0\,,1]^d$. We can combine Lemma \ref{lem:linear:L2:est} with
theorems of Dudley (see, for example, Marcus and 
Rosen \cite[Theorem 6.1.2, p.\ 245]{MR}) and Fernique
\cite[Theorem 6.2.2, p.\ 251]{MR}, together with Belyaev's dichotomy
\cite[Theorem 5.3.10, p.\ 213]{MR}, and deduce the following:

\begin{proposition}\label{pr:cont:x}
	Suppose $\exp(-\Re\Psi)\in L^t(\R^d)$ for all $t>0$,
	and $\Upsilon(1)<\infty$ so that \eqref{heat:lin}
	has a random-field solution $\bm{u}$ with $u_0\equiv 0$.
	Then the following are equivalent:
	\begin{enumerate}
		\item $x\mapsto u_t(x)$ has a continuous modification for some $t>0$;
		\item $x\mapsto u_t(x)$ has a continuous modification for all $t>0$;
		\item The following metric-entropy condition holds:
			\begin{equation}\label{cond:ME}
				\int_{0^+} \left(\log N_d(\epsilon)\right)^{1/2}\,\d\epsilon<\infty.
			\end{equation}
	\end{enumerate}
\end{proposition}

Next we describe a large family of examples for \eqref{heat:lin} 
that have continuous random-field solutions. Throughout,
we write ``$h\asymp g$'' in place of ``$(h/g)$ is bounded, above and
below uniformly, by finite positive constants.''

\begin{theorem}\label{th:example:riesz}
	Suppose $f(x)=\text{\rm const}/ \|x\|^{d-\beta}$ and $\Re\Psi(\xi)\asymp
	\|\xi\|^\alpha$ for some $\alpha\in[0\,,2]$
	and $\beta\in(0\,,d)$. Then \eqref{heat:lin} has a
	random-field solution if and only if  	$\alpha+\beta>d$. In this case, 
	\eqref{cond:ME} holds. In fact, we have
	\begin{equation}
		d(x\,,y) \asymp g(\|x-y\|) \qquad\text{uniformly when $\|x-y\|<1/\e$},
	\end{equation}
	where for all $r\in(0\,,1/\e)$,
	\begin{equation}
		g(r) := \begin{cases}
				r^{(\alpha+\beta-d)/2}&\text{if $\alpha+\beta\in(d+1\,,d+2)$},\\
				r\sqrt{\log(1/r)}&\text{if $\alpha+\beta=d+2$},\\
			r&\text{if $\alpha+\beta>d+2$}.
		\end{cases}
	\end{equation}
\end{theorem}

\begin{remark}
	The condition that $\Re\Psi(\xi)\asymp\|\xi\|^\alpha$ implies
	that the upper and lower Blumenthal--Getoor
	indices of $\Psi$ match and are both equal to $\alpha$; see 
	Blumenthal and Getoor \cite[Theorem 3.2]{BG:61}
	and Khoshnevisan and Xiao \cite{KX:HAAL} for definitions
	and further details, including the various
	connections that exist between those indices and the fractal
	properties of the underlying L\'evy process $X$.
	\qed
\end{remark}

\begin{proof}[Proof of Theorem \ref{th:example:riesz}]
	Recall \eqref{eq:Upsilon}.
	In order to prove the existence of random-field solutions, 
	it suffices to show that $\Upsilon(1)<\infty$. We begin by writing
	\begin{equation}\begin{split}
		\Upsilon(1)&=\frac{1}{(2\pi)^d}\int_{\R^d}
			\frac{\hat{f}(\xi)}{1+2\Re\Psi(\xi)}\,\d\xi\\
		&:=I_1+I_2,
	\end{split}\end{equation}
	where
	\begin{equation}\begin{split}
		I_1 &:=\frac{1}{(2\pi)^d}\int_{\|\xi\|\leq 1}
			\frac{\hat{f}(\xi)}{1+2\Re\Psi(\xi)}\,\d\xi,
			\qquad\text{and}\\
		I_2 &:=\frac{1}{(2\pi)^d}\int_{\|\xi\|>1}
			\frac{\hat{f}(\xi)}{1+2\Re\Psi(\xi)}\,\d\xi. \\
	\end{split}\end{equation}
	It is clear from the hypothesis of the Lemma that $I_1$ is always finite, because
	$\beta<d$. We now turn our attention to $I_2$, and note that
	\begin{equation}
		I_2\asymp \int_{\|\xi\|>1}\frac{\d \xi}{\|\xi\|^{\alpha+\beta}}.
	\end{equation}
	Therefore, $I_2<\infty$ if and only if $\alpha+\beta>d$. 
	This concludes the first part of the result. 
	
	For the second part we assume that
	\begin{equation}
		\|x-y\|\le 1/\e.
	\end{equation}
	
	We write
	\begin{equation}
			\left| d(x\,,y)\right|^2 := \frac{1}{(2\pi)^d}\left( J_1 + J_2 + J_3\right),
	\end{equation}
	where
	\begin{equation}\begin{split}
			J_1 &:= \int_{\|\xi\|\le 1}\frac{1-\cos(\xi\cdot(x-y))}{1+2\Re\Psi(\xi)}\hat{f}(\xi)\,\d\xi,\\
			J_2 &:= \int_{\|\xi\|>1/\|x-y\|}\frac{1-\cos(\xi\cdot(x-y))}{1+2\Re\Psi(\xi)}\hat{f}(\xi)\,\d\xi,\\
			J_3 &:= \int_{1<\|\xi\|\le 1/\|x-y\|}\frac{1-\cos(\xi\cdot(x-y))}{%
				1+2\Re\Psi(\xi)}\hat{f}(\xi)\,\d\xi.
		\end{split}\end{equation}
		We can estimate each $J_j$ separately.
		
		Because $1-\cos\theta\asymp\theta^2$ for $\theta\in(-1\,,1)$,
		\begin{equation}\label{eq:J1}\begin{split}
		J_1 & \asymp \|x-y\|^2 \cdot\int_{\|\xi\|\le 1} 
			\|\xi\|^2\hat{f}(\xi)\,\d\xi\\
		&\asymp \|x-y\|^2\cdot\int_{\|\xi\|\le 1} 
			\|\xi\|^{2-\beta}\,\d\xi\\
		&\asymp \|x-y\|^2\cdot \int_0^1 \frac{\d r}{r^{\beta-d-1}}.
	\end{split}\end{equation}
	Because $\beta<d$, the
	integral term in the above display is finite, 
	and hence
	\begin{equation}
		J_1\asymp \|x-y\|^2.
	\end{equation}
	
	We estimate $J_2$ similarly:
	\begin{equation}\label{eq:J2}\begin{split}
		J_2 &\le \text{const}\cdot
			\int_{\|\xi\|>1/\|x-y\|} \frac{\hat{f}(\xi)\,\d\xi}{\|\xi\|^\alpha}\\
		&\asymp  \int_{1/\|x-y\|}^\infty \frac{\d r}{r^{\alpha+\beta-d+1}}.
	\end{split}\end{equation}
	The final integral is finite if and only if $\alpha+\beta>d$; 
	and in this case,
	we have the estimate 
	$0\le J_2\le\text{const}\cdot \|x-y\|^{\alpha+\beta-d}$.
			
	Finally,
	\begin{equation}\label{eq:J3}\begin{split}
		J_3 &\asymp \int_{1<\|\xi\|\le 1/\|x-y\|}
			\|\xi\|^{2-\alpha-\beta}\cdot\|x-y\|^2\,\d\xi\\
		&\asymp \|x-y\|^2\cdot\int_1^{1/\|x-y\|} \frac{\d r}{r^{\alpha+\beta-d-1}}.
	\end{split}\end{equation}
	We evaluate the integrals \eqref{eq:J2} and \eqref{eq:J3}
	for the different cases of $\alpha+\beta$ to obtain the result.
\end{proof}
The following is an immediate consequence of Proposition \ref{pr:cont:x} and Theorem \ref{th:example:riesz}.

\begin{corollary}
	Every random-field solution $\bm{u}$ given by Theorem \ref{th:example:riesz} has
	a continuous modification for all $t>0$.
\end{corollary}

We devote the remainder of this section to a special case of \eqref{heat:lin} namely
\begin{equation}\label{heat:R3}
	\frac{\partial}{\partial t} u_t(x) = (\Delta u_t)(x) + \dot{F}_t(x),
\end{equation}
where $u_0\equiv 0$,
$x\in\R^3$, $t>0$, and the Laplacian acts on the $x$ variable only.
The noise $F$ is a centered Gaussian noise, as before, that is
white in time and homogeneous in space with a correlation
function $f$ that satisfies the following for a fixed $q\in\R$:
\begin{equation}\label{eq:nice:cor}
	\hat{f}(\xi) \asymp \frac{1}{\|\xi\|(\log\|\xi\|)^q}
	\qquad\text{for $\xi\in\R^3$ with $\|\xi\|>\e$}.
\end{equation}
According to Theorem \ref{th:nice:cor} on page \pageref{th:nice:cor},
such correlation functions exist.

The following lemma will be useful for the proof of the main result of this section.

\begin{lemma}\label{lem:spherical}
	If $g:\R^3\mapsto\R_+$ is a Borel-measurable radial function, then
	\begin{equation}
		\int_{\|x\|>1/\|y\|}(1-\cos(x\cdot y))g(x)\,\d x\geq \text{\rm const}\cdot
		\int_{\|x\|>1/\|y\|}g(x)\,\d x,
	\end{equation}
	uniformly for all $y\in\R^3\setminus\{0\}$.

\end{lemma}

\begin{proof}
	Clearly,
	\begin{equation}\begin{split}
		&\int_{\|x\|>1/\|y\|}
			(1-\cos(x\cdot y))g(x)\,\d x\\
		&\hskip1.4in=\int_{1/\|y\|}^\infty  r^2
			R(r)\,\d r
			\int_{\textbf{S}^2}\d\theta\ (1-\cos(y\cdot r\theta)),
	\end{split}\end{equation}
	where $R$ is the function on $\R_+$ defined
	by $R(\|x\|):=g(x)$ for all $x\in\R^3$.
	But for all $r>0$, the $\d\theta$-integral can be computed as
	\begin{equation}
		\int_{\mathbf{S}^2}(1-\cos(y\cdot r\theta))\, \d \theta
		=\text{const}\cdot\left(1-\frac{\sin(r\|y\|)}{r\|y\|}\right),
	\end{equation}
	 and this is bounded below uniformly, as long as $r>1/\|y\|$. We
	combine the preceding two displays to obtain the result.
\end{proof}
The following is the main result concerning \eqref{heat:R3}. 

\begin{theorem}\label{theorem:counter}
	Consider the stochastic heat equation \eqref{heat:R3} in $\R^3$,
	where the correlation function $f$ of the noise satisfies \eqref{eq:nice:cor}
	for a given fixed value $q\in\R$. Then:
	\begin{enumerate}
	\item[(a)] \eqref{heat:R3} has a random-field solution $\bm{u}$ iff $q>1$;
	\item[(b)] $x\mapsto u_t(x)$ has a continuous modification for all $t>0$ iff $q>2$.
	\end{enumerate}
\end{theorem}

\begin{remark}\label{rem:discontinuous:RF}
	Theorem \ref{theorem:counter}, and general facts about 
	stationary Gaussian processes [see Belyaev's
	dichotomy \cite[Theorem 5.3.10, p.\ 213]{MR}, for instance], 
	together prove that when $q\in(1\,,2]$,
	the stochastic heat equation \eqref{eq:nice:cor} has a random-field solution
	$\bm{u}$ that almost surely has infinite oscillations in every open
	space-time set. This example was mentioned at the end of Introduction.
	\qed
\end{remark}

\begin{proof}
	In order to show existence of a random-field solution, it suffices to show that
	$\Upsilon(1)<\infty$ if and only if $q>1$. Because
	$\Psi(\xi)=\|\xi\|^2$, we may write $\Upsilon(1)$ as follows:
	\begin{equation}\begin{split}
		\Upsilon(1) &=\frac{1}{8\pi^3}\int_{\R^3}\frac{\hat{f}(\xi)}{1+2
			\|\xi\|^2}\,\d \xi\\
		&:=\frac{I_1+I_2}{8\pi^3},
	\end{split}\end{equation}
	where
	\begin{equation}
		I_1 := \int_{\|\xi\|<\e}\frac{\hat{f}(\xi)}{1+2 \|\xi\|^2}\,\d \xi
		\quad\text{and}\quad
		I_2 := \int_{\|\xi\|>\e}\frac{\hat{f}(\xi)}{1+2 \|\xi\|^2}\,\d \xi.
	\end{equation}
	Direct inspection reveals that
	\begin{equation}
		I_1\asymp \int_0^\e \frac{r}{(\log r)^q}\,\d r
		\quad\text{and}\quad
		I_2\asymp \int_\e^\infty \frac{1}{r(\log r)^q}\,\d r.
	\end{equation}
	It follows readily from this that $\Upsilon(1)<1$ if and only if $q>1$.
	
	We now turn our attention to the second part of the proof. 
	Throughout, we assume that $\|x-y\|\leq 1/\e$, and consider the following integral:
	\begin{equation}\begin{split}
		|d(x\,,y)|^2&=\frac{1}{8\pi^3}
			\int_{\R^3}\frac{1-\cos(\xi\cdot(x-y))}{1+2\|\xi\|^2}\hat{f}(\xi)\,\d \xi\\
		&:=\frac{J_1+J_2+J_3}{8\pi^3},
	\end{split}\end{equation}
	where
	\begin{equation}\begin{split}
			J_1 &:= \int_{\|\xi\|\le\e}\frac{1-\cos(\xi\cdot(x-y))}{
				1+2\|\xi\|^2}\hat{f}(\xi)\,\d\xi,\\
			J_2 &:= \int_{\|\xi\|>1/\|x-y\|}\frac{1-\cos(\xi\cdot(x-y))}{
				1+2\|\xi\|^2}\hat{f}(\xi)\,\d\xi,\\
			J_3 &:= \int_{\e<\|\xi\|\le 1/\|x-y\|}\frac{1-\cos(\xi\cdot(x-y))}{
				1+2\|\xi\|^2}\hat{f}(\xi)\,\d\xi.
	\end{split}\end{equation}
	We estimate each of the integral separately. The first term can be dealt with easily, 
	and we obtain the following, using similar computations to those in the
	proof of Theorem \ref{th:example:riesz}:
	\begin{equation}
		J_1\asymp \|x-y\|^2.
	\end{equation}
	The estimation of the second term requires a little bit more work, viz.,
	\begin{equation}\begin{split}
		J_2&\le \frac12 \int_{\|\xi\|>1/\|x-y\|}\frac{1-\cos(\xi\cdot(x-y))}{
			\|\xi\|^2}\hat{f}(\xi)\,\d \xi\\
		&\leq\text{const}\cdot
			\int_{\|\xi\|>1/\|x-y\|}\frac{\hat{f}(\xi)}{\|\xi\|^2}\,\d \xi\\
		&\leq\text{const}\cdot
			\int_{1/\|x-y\|}^\infty\frac{1}{r(\log r)^q}\,\d r\\
		&=\text{const}\cdot
			\left(\log \frac{1}{\|x-y\|}\right)^{-q+1}.
	\end{split}\end{equation}
	Finally, we consider the final term $J_3$:
	\begin{equation}\begin{split}
		J_3&\asymp  \|x-y\|^2\int_{\e<\|\xi\|\leq 1/\|x-y\|} \hat{f}(\xi)\,\d \xi\\
		& \asymp  \|x-y\|^2\int_\e^{1/\|x-y\|}\frac{r}{(\log r)^q}\,\d r\\
		&\leq  \text{const}\cdot \left(\log \frac{1}{\|x-y\|}\right)^{-q}.
	\end{split}\end{equation}
	Upon combining the above estimates we obtain the
	bound
	\begin{equation}
		|d(x\,,y)|^2\leq \text{const}\cdot \left(\log \frac{1}{\|x-y\|} \right)^{-q+1}.
	\end{equation}
	
	Next, we compute a similar lower bound for $|d(x\,,y)|^2$. 
	Since the integrands are nonnegative throughout, we may consider 
	only $J_2$. In that case, Lemma \ref{lem:spherical} yields the following:
	\begin{equation}\begin{split}
		J_2&\geq\text{const}\cdot
			\int_{\|\xi\|\geq 1/\|x-y\|}\frac{\hat{f}(\xi)}{\|\xi\|^2}\,\d \xi\\
		&\ge \text{const}\cdot \left(\log \frac{1}{\|x-y\|}\right)^{-q+1}.
	\end{split}\end{equation}

	Thus far, we have proved that
	\begin{equation}
		d(x\,,y)\asymp |\log (\|x-y\|)|^{(1-q)/2},
	\end{equation}
	uniformly, as long as $\|x-y\|<1/\e$. 
	From this, we obtain 
	\begin{equation}
		\log N_d(\epsilon)\asymp \epsilon^{2/(1-q)},
	\end{equation}
	valid for  $0<\epsilon<1/\e$.
	In particular, the metric-entropy
	condition \eqref{cond:ME} applies if and only if $q>2$. 
	Since the other conditions of
	Proposition \ref{pr:cont:x} hold [for elementary reasons],
	the second part of the theorem follows from Proposition \ref{pr:cont:x}.
\end{proof}

%% file: FoonKhosh-NL1.tex
\chapter{The Nonlinear Equation}
\label{ch:NL}

The primary goal of this chapter is to study the fully-nonlinear
stochastic  heat equation \eqref{heat}  as described in the introduction. 

In the first part,
we derive a series of \emph{a priori} estimates that 
ultimately lead to the proof of Theorem \ref{th:existence}. The latter
theorem shows that the finite-potential Condition \ref{cond:1} is
sufficient for the existence of a mild solution to the stochastic
heat equation. As a byproduct, that theorem
also yields a temporal growth rate for the solution.
This means that under some natural conditions on the multiplicative nonlinearity $\sigma$,
the mild solution will not be intermittent.

The second part is devoted to the proofs of 
Theorems \ref{th:interm} and
\ref{th:interm:asymp}, and thereby establishing
the fact that, in contrast to the preceding discussion, if
``there is enough symmetry and nonlinearity,''
then the mild solution to the
stochastic heat equation is weakly intermittent.

In the third and final part, we give a partial answer to a deep question
of David Nualart who asked about the ``effect of drift'' on the intermittence
of the solution. In particular, we show that if the
drift is exactly linear---which corresponds to a massive and/or dissipative
version of \eqref{heat}---then there is frequently an explicit
phase transition which describes
the amount of drift  needed in order to offset the intermittent 
multiplicative effect of the underlying noise.

\section{Existence and Uniqueness}
The main goal of this section is to prove Theorem \ref{th:existence}.
With that aim in mind, we can formulate \eqref{heat}, in mild form, as follows:
\begin{equation}\label{heat:mild:form}\begin{split}
	u_t(x) &= (P_tu_0)(x) + \int_0^t\d s\int_{\R^d}\d y\
		p_{t-s}(y-x) b(u_s(y))\\
	&\hskip1.4in+\int_0^t\int_{\R^d} p_{t-s}(y-x)\sigma(u_s(y))\, F(\d s\,\d y).
\end{split}\end{equation}
As is customary, we seek to find a mild solution that satisfies
the following integrability condition:
\begin{equation}\label{eq:mild}
	\sup_{t\in[0,T]}\sup_{x\in\R^d}
	\E\left(|u_t(x)|^2\right)<\infty
	\qquad\text{for all $T>0$}.
\end{equation}

In order to prove Theorem \ref{th:existence}
we apply a familiar fixed-point argument,
though the details of this argument are not entirely standard.

Let $\bm{\mathcal{F}}:=\{\mathcal{F}_t\}_{t\ge 0}$ denote the right-continuous
complete filtration generated by the noise $F$. Specifically, for every positive $t$, we define
$\mathcal{F}_t^0$ to be the $\sigma$-algebra generated by
random variables of the form of
the Wiener integral $\int_{[0,t]\times\R^d} \phi_s(x)\, F(\d s\,\d x)$,
as $\phi$ ranges over $L^2([0\,,t]\times\R^d)$. Define $\mathcal{F}^1_t$
to be the $\P$-completion of $\mathcal{F}^0_t$, and finally define
\begin{equation}
	\mathcal{F}_t := \bigcap_{s>t}\mathcal{F}_s^1
\end{equation}
as the right-continuous extension. 

We recall from Walsh \cite{Walsh} that a random field $\{v_t(x)\}_{t\ge 0,x\in\R^d}$ is
\emph{predictable} if it can be realized as an $L^2(\P)$-limit of finite linear combination
of random fields of the type
\begin{equation}
	z_t(x)(\omega) := X(\omega)\1_{(a,b]\times A}(t\,,x)
	\qquad\text{for $t>0$, $x\in\R^d$, and $\omega\in\Omega$,}
\end{equation}
where: $0<a<b<\infty$, $A\subseteq\R^d$ is compact, and
$X$ is an $\mathcal{F}_t$-measurable and bounded random variable.\footnote{%
We are using the standard ``$(\Omega\,,\mathcal{F},\P)$'' notation of
probability for the underlying probability space, of course.}
Define for all predictable random fields $v$,
\index{000Av@$\mathcal{A}v$, a stochastic integral}%
\begin{equation}\label{def:A}
	(\mathcal{A}v)_t(x) := \int_{[0,t]\times\R^d}
	p_{t-s}(y-x)\sigma(v_s(y))\, F(\d s\,\d y),
\end{equation}
and
\begin{equation}
	(\mathcal{B}v)_t(x) := \int_0^t\d s\int_{\R^d}\d y\
	p_{t-s}(y-x)b(v_s(y)),
\end{equation}
provided that the integrals exist: The first integral must exist
in the sense of Walsh \cite{Walsh}; and the second in the sense of Lebesgue.

Define for all $\beta,p>0$, and all predictable random fields $v$,
\begin{equation}
	\|v\|_{\beta,p} := \sup_{t>0} \sup_{x\in\R^d}
	\left[ \e^{-\beta t} \E\left(|v_t(x)|^p\right) \right]^{1/p}.
\end{equation}
It is easy to see that the preceding defines a [pseudo-] norm on
random fields, for every fixed choice of $\beta,p>0$. In
fact, these are one among many possible infinite-dimensional $L^p$-norms.
And the corresponding $L^p$-type space is denoted by $\bm{B}_{\beta,p}$.
We make the following definition which will be in force throughout the rest of the paper.

\index{000Bbetap@$\bm{B}_{\beta,p}$, a family of stochastic
	Banach spaces}%
\begin{definition}\label{def:B_beta:p}
	Let $\bm{B}_{\beta,p}$ denote the
	collection of all [equivalence classes of modifications of] predictable 
	random fields $X:=\{X_t(x)\}_{t\ge0,x\in\R^d}$
	such that $\|X\|_{\beta,p}<\infty$.
\end{definition}
One can easily checks easily that $\|\cdot\|_{\beta,p}$ defines a pseudo-norm on
$\bm{B}_{\beta,p}$. Moreover, if we identify $X\in\bm{B}_{\beta,p}$
with $Y\in\bm{B}_{\beta,p}$ when $\|X-Y\|_{\beta,p}=0$, then
[the resulting collection of equivalence classes in] 
$\bm{B}_{\beta,p}$ becomes a Banach space. Because $\|X-Y\|_{\beta,p}=0$
if and only if $X$ and $Y$ are modifications of one another, it follows
that---after the usual identification of a process with its 
modifications---$\bm{B}_{\beta,p}$ is a Banach space of [equivalence classes of]
functions with finite $\|\cdot\|_{\beta,p}$ norm. 

Our next two lemmas contain \emph{a priori} estimates on Walsh-type
stochastic integrals, as well as certain  Lebesgue integrals.
Among other things, these lemmas show that $\mathcal{B}$ and $\mathcal{A}$
are bounded linear maps from predictable processes to predictable processes.
These lemmas
are motivated strongly by the theory of optimal regularity
for parabolic equations, as is our entire approach to the proof of 
Theorem \ref{th:existence}; see Lunardi \cite{Lunardi}. 
We follow the main idea of optimal regularity,
and aim to find a good function space such that if $u_0$ resides
in that function space, then $u_t$ has to live in the same
function space for all $t$. As we shall soon see,
the previously-defined Banach spaces
$\{\bm{B}_{\beta,p}\}_{\beta,p>0}$ form excellent candidates
for those function spaces. In a rather different context,
this general idea appears also in Dalang and Mueller \cite{DalangMueller}.
Those authors show that $L^2(\R^d)$ is also a good candidate for
such a function space provided that $\sigma(0)=0$.

Here and throughout,
we will use the following notation on Lipschitz functions.

\begin{convention}\label{conv:Lip}
	If $g:\R^d\to\R$ is Lipschitz continuous, then
	we can find finite constants ${\rm C}_g$ and ${\rm D}_g$
	such that
	\begin{equation}
		|g(x)|\le {\rm C}_g+ {\rm D}_g|x|
		\quad\text{for all $x\in\R^d$.}
	\end{equation}
	To be concrete, we choose ${\rm C}_g:=|g(0)|$
	and ${\rm D}_g:=\lip_g$, to be concrete.
\end{convention}

As mentioned above, the next two results describe \emph{a priori} estimates for the Walsh-integral-processes
$\mathcal{B}v$ and $\mathcal{A}v$ when $v$ is a nice predictable random field.
Together,  they imply
that the random linear operators
$\mathcal{A}$ and $\mathcal{B}$ map each and every $\bm{B}_{\beta,p}$
into itself boundedly and continuously. The respective operator norms
are both described in terms of a replica potential of the correlation function $f$.

\begin{lemma}\label{lem:Norm:B}
	For all integers $p\ge 2$, real numbers $\beta>0$,
	and predictable random fields $v$ and $w$,
	\begin{equation}
		\|\mathcal{B}v\|_{\beta,p} \le \frac{p}{\beta}
		\left( \frac{{\rm C}_b}\e + {\rm D}_b\|v\|_{\beta,p}\right),
	\end{equation}
	and
	\begin{equation}
		\| \mathcal{B}v-\mathcal{B}w\|_{\beta,p}\le
		\frac{p\lip_b}{\beta}\|v-w\|_{\beta,p}.
	\end{equation}
\end{lemma}

\begin{proof}
	On one hand, the triangle inequality implies that
	$\E(|(\mathcal{B}u)_t(x)|^p)$ is bounded
	above by the following quantity:
	\[
		\int_0^t\d s_1\int_{\R^d}\d y_1\,\cdots
		\int_0^t\d s_p\int_{\R^d}\d y_p\
		\prod_{k=1}^p p_{t-s_k}(t_k-x)\cdot
		\E\left( \prod_{j=1}^p |b(v_{s_j}(y_j))|\right).
	\]
	On the other hand, the generalized H\"older inequality tells us that
	\begin{equation}
		\E\left( \prod_{j=1}^p |b(v_{s_j}(y_j))|\right)
		\le \prod_{j=1}^p \| b(v_{s_j}(y_j))\|_p.
	\end{equation}
	Therefore, we can conclude that
	\begin{equation}
		\E\left(\left| (\mathcal{B}v)_t(x) \right|^p\right)
		\le \left({\rm C}_bt+{\rm D}_b\int_0^t\d s\int_{\R^d}\d y\
		p_{t-s}(y-x)\|v_s(y)\|_p\right)^p.
	\end{equation}
	We multiply the preceding by $\exp(-\beta t)$ and take the the $(1/p)$-th root
	to find that
	\begin{align}\label{eq:Norm:B}\nonumber
		&\left[\e^{-\beta t}\E\left(\left| 
			(\mathcal{B}v)_t(x) \right|^p\right) \right]^{1/p}\\
			\nonumber
		&\le {\rm C}_bt\e^{-\beta t/p}+{\rm D}_b\int_0^\infty\d s\int_{\R^d}\d y\
			\e^{-\beta(t-s)/p}p_{t-s}(y-x)\e^{-\beta s/p}\|v_s(y)\|_p\\
		&\le {\rm C}_bt\e^{-\beta t/p} + \frac{p{\rm D}_b}\beta \|v\|_{\beta,p}.
	\end{align}
	The first display of the
	lemma follows because $t\exp(-\beta t/p)\le p/(\e\beta)$
	for all $t>0$. In order to obtain the second display we note that
	\begin{equation}
		|(\mathcal{B}v)_t(x)-(\mathcal{B}w)_t(x)|\le\lip_b\cdot
		(\mathcal{B}_1(|v-w|))_t(x),
	\end{equation}
	where $\mathcal{B}_1$ is
	defined exactly as $\mathcal{B}$ was, but with $b(x)$ replaced
	by $b_1(x)=x$. Because we may choose
	${\rm C}_{b_1}=0$
	and ${\rm D}_{b_1}=1$, the second assertion of the lemma follows from the first.
\end{proof}

\begin{lemma}\label{lem:Norm:sigma}
	For all even integers $p\ge 2$, real numbers $\beta>0$,
	and predictable random fields $v$ and $w$,
	\begin{equation}
		\|\mathcal{A}v\|_{\beta,p} \le z_p
		\left( {\rm C}_\sigma + {\rm D}_\sigma \|v\|_{\beta,p}\right)\sqrt{%
		( \bar R_{2\beta/p} f)(0)},
	\end{equation}
	and
	\begin{equation}
		\| \mathcal{A}v-\mathcal{A}w\|_{\beta,p}\le z_p
		\lip_\sigma\|v-w\|_{\beta,p}\sqrt{( \bar R_{2\beta/p} f)(0)}.
	\end{equation}
\end{lemma}

\begin{proof}
	According to Davis's formulation \cite{Davis}
	of the Burkholder--Davis--Gundy inequality,
	$\E (  | (\mathcal{A}v)_t(x) |^p )$
	is bounded above by $z_p^p$ times the expectation of
	\begin{equation}
		\left| \int_0^t \d s\int_{\R^d}\d y
		\int_{\R^d}\d z\ V_s(y\,,z) p_{t-s}(y-x)p_{t-s}(z-x)f(z-y)
		\right|^{p/2},
	\end{equation}
	where
	\begin{equation}
		V_s(y\,,z):=\sigma(v_s(y))\sigma(v_s(z)).
	\end{equation}
	By the generalized H\"older inequality,
	\begin{equation}
		\E\left(\prod_{j=1}^{p/2} V_{s_j}(y_j\,,z_j)\right)
		\le \prod_{j=1}^{p/2}\left\| \sigma(v_{s_j}(y_j))\right\|_p
		\|\sigma(v_{s_j}(z_j)\|_p.
	\end{equation}
	Consequently, $\E\left( \left| (\mathcal{A}v)_t(x)\right|^p\right)$
	is bounded above by $z_p^p$ times
	\begin{equation}\label{eq:Norm:sigma}
		\left(\int_0^t \d s\int_{\R^d}\d y
		\int_{\R^d}\d z\ V'_s(y\,,z) p_{t-s}(y-x)p_{t-s}(z-x)f(z-y)
		\right)^{p/2},
	\end{equation}
	where
	\begin{equation}
		V_s'(y\,,z) := \left\| \sigma(v_s(y))\right\|_p
		\left\| \sigma(v_s(z))\right\|_p.
	\end{equation}
	We can note that for all $s>0$ and $y\in\R^d$,
	\begin{equation}\label{eq:Norm:sigma:1}\begin{split}
		\left\| \sigma(v_s(y))\right\|_p &\le {\rm C}_\sigma
			+{\rm D}_\sigma\|v_s(y)\|_p\\
		&\le {\rm C}_\sigma + {\rm D}_\sigma\e^{\beta s/p}\|v\|_{\beta,p}\\
		&\le \e^{\beta s/p}\left( {\rm C}_\sigma + {\rm D}_\sigma \|v\|_{\beta,p}\right).
	\end{split}\end{equation}
	Therefore, a line or two of computation yield
	\begin{equation}\label{eq:Sf:1}
		\E\left( \left| (\mathcal{A}v)_t(x)\right|^p\right)
		\le  \e^{\beta t}z_p^p\left( {\rm C}_\sigma +
		{\rm D}_\sigma \|v\|_{\beta,p}\right)^p
		\left(\int_0^\infty \e^{-2\beta s/p} \mathcal{H}_s\,\d s
		\right)^{p/2},
	\end{equation}
	where
	\begin{equation}\begin{split}
		\mathcal{H}_s &:= \int_{\R^d}\d a\int_{\R^d}\d b\
			f(a-b) p_s(a)p_s(b)\\
		&=\left( P^*_sP_s f\right)(0)\\
		&= (\bar P_sf)(0).
	\end{split}\end{equation}
	And hence,
	\begin{equation}\label{eq:Sf:1:1}
		\E\left( \left| (\mathcal{A}v)_t(x)\right|^p\right)\\
		\le  \e^{\beta t}z_p^p\left( {\rm C}_\sigma +
		{\rm D}_\sigma \|v\|_{\beta,p}\right)^p
		\left( ( \bar R_{2\beta/p} f)(0)
		\right)^{p/2},
	\end{equation}
	The first assertion of the lemma follows immediately from this.
	
	In order to deduce the second assertion we note that
	\begin{equation}
		\|\mathcal{A}v-\mathcal{A}w\|_{\beta,p}
		\le \lip_\sigma\cdot\mathcal{A}_1(|v-w|), 
	\end{equation}
	where $\mathcal{A}_1$
	is the same as $\mathcal{A}$, but with $\sigma(x)$
	replaced by $\sigma_1(x)=x$. Therefore, the first assertion
	of the lemma implies the second. This completes the proof.
\end{proof}

We are ready to begin a more-or-less standard
iterative construction that is used to
prove Theorem \ref{th:existence}.

Let
\begin{equation}
	u^0_t(x):=u_0(x),
\end{equation}
and define iteratively: For all $n\ge 0$,
\begin{equation}
	u^{n+1}_t(x) = (P_tu_0)(x) + (\mathcal{B}u^n)_t(x)
	+(\mathcal{A}u^n)_t(x).
\end{equation}

\begin{lemma}\label{lem:Lp:bdd}
	Choose and fix $\beta>0$ and an even integer $p\ge 2$.
	If
	\begin{equation}
		\frac{p{\rm D}_b}{\beta}+z_p{\rm D}_\sigma
		\sqrt{( \bar R_{2\beta/p} f)(0)}<1,
	\end{equation}
	then $\sup_{n\ge 0} \|u^n\|_{\beta,p}<\infty$.
\end{lemma}

\begin{proof}
	Because $P_tu_0$ is bounded, uniformly in modulus, by
	$\sup_{x\in\R^d}|u_0(x)|$, the triangle inequality implies that
	\begin{equation}
		\| u^{n+1}\|_{\beta,p}  \le \sup_{x\in\R^d}|u_0(x)|
		+ \| \mathcal{B}u^n\|_{\beta,p} + \|\mathcal{A}u^n\|_{\beta,p}.
	\end{equation}
	Lemmas \ref{lem:Norm:B} and \ref{lem:Norm:sigma},
	and a few lines of direct computation, together imply that
	\begin{equation}
		\| u^{n+1}\|_{\beta,p}
		\le A + B\|u^n\|_{\beta,p},
	\end{equation}
	where 
	\begin{equation}
		A := \sup_{x\in\R^d}|u_0(x)| + \frac{p{\rm C}_b}{\beta\e} + z_p
		{\rm C}_\sigma\sqrt{( \bar R_{2\beta/p} f)(0)},
	\end{equation}
	and
	\begin{equation}
		B:=\frac{p{\rm D}_b}{\beta}+z_p{\rm D}_\sigma
		\sqrt{( \bar R_{2\beta/p} f)(0)}.
	\end{equation}
	Iteration yields the bound
	\begin{equation}
		\| u^{n+1}\|_{\beta,p} \le A\left( 1 + B+ \cdots + B^{n-1} +
		B^n\sup_{x\in\R^d}|u_0(x)|\right).
	\end{equation}
	Consequently, if $B<1$ then
	\begin{equation}
		\sup_{k\ge 1} \| u^k\|_{\beta,p} \le \frac{A}{1-B}.
	\end{equation}
	Since $\|u^0\|_{\beta,p}\le\sup_{x\in\R^d}|u_0(x)|<\infty$,
	the lemma follows.
\end{proof}
We now have all the technical estimates for the proof of Theorem \ref{th:existence}.
\subsection{Proof of Theorem \ref{th:existence}}
Without loss of generality, we can find $\beta>0$
such that $Q(p\,,\beta)<1$, otherwise there is nothing to prove.
Choose and fix such a $\beta$.
	
Thanks to Lemma \ref{lem:Lp:bdd}, every $u^n$ is well defined
and $\|u^n\|_{\beta,p}$ is finite, uniformly in $n$. In particular,
$u^n\in\bm{B}_{\beta,p}$.
Next we apply Lemmas \ref{lem:Norm:B} and \ref{lem:Norm:sigma}---with
${\rm C}_g:=|g(0)|$ and ${\rm D}_g:=\lip_g$---to find that
\begin{equation}\begin{split}
	\| u^{n+1}-u^n\|_{\beta,p} &\le
		\|\mathcal{B}u^n-\mathcal{B}u^{n-1}\|_{\beta,p}
		+ \|\mathcal{A}u^n-\mathcal{A}u^{n-1} \|_{\beta,p}\\
	&\le \|u^n-u^{n-1}\|_{\beta,p}\cdot Q(p\,,\beta).
\end{split}\end{equation}
Because $Q(p\,,\beta)<1$, the preceding implies that
\begin{equation}
	\sum_{n=1}^\infty \| u^{n+1}-u^n\|_{\beta,p}\le\text{const}\cdot
	\sum_{n=1}^\infty \left\{
	Q(p\,,\beta)\right\}^n<\infty.
\end{equation}
Therefore,
we can find a predictable random field $u^\infty\in\bm{B}_{\beta,p}$ such that
\begin{equation}
	\lim_{n\to\infty} \| u^n-u^\infty\|_{\beta,p}=0.
\end{equation}
Our arguments
can be adjusted to show also that
\begin{equation}
	\lim_{n\to\infty}\| \mathcal{B}u^n-\mathcal{B}u^\infty\|_{\beta,p}=0,
\end{equation}
as well as
\begin{equation}
	\lim_{n\to\infty} \| \mathcal{A}u^n-\mathcal{A}u^\infty\|_{\beta,p}=0.
\end{equation}
This proves that $u^\infty$ is another solution to \eqref{heat:mild:form}.
As we have mentioned, $u$ is the almost-surely unique solution to
\eqref{heat:mild:form}. Therefore, $u^\infty$ is equal to $u$,  up to
evanescence.  It follows that
$u\in\bm{B}_{\beta,p}$ for all $\beta>0$ such that
$Q(p\,,\beta)<1$. This proves the theorem.
\qed\\

\section{Intermittency}

In this section we prove Theorems \ref{th:interm} and
\ref{th:interm:asymp}, 
which state that the solution to the stochastic heat equation
can be weakly intermittent in the presence of enough symmetry
and nonlinearity. The basic idea is to follow our earlier work on
space-time white noise \cite{FK} and apply the renewal-theory
methods of Choquet and Deny \cite{ChoquetDeny}. However,
this turns out to be a difficult adaptation which appears to require
a fair bit of harmonic analysis.

It might help to recall the definition \eqref{eq:Upsilon}
[page \pageref{eq:Upsilon}] of the function $\Upsilon$,
and the relation [Theorem \ref{th:Dalang:1}, page
\pageref{th:Dalang:1}] between the positive
number $\Upsilon(\beta)$
and the maximum value of the replica $\beta$-potential of the correlation
function $f$, as well as the positive number $(\bar R_\beta f)(0)$
of that replica potential at the origin.

\begin{proof}[Proof of Theorem \ref{th:interm}]
	We can assume, without loss of generality, that there exists
	$\beta>0$ such that
	\begin{equation}\label{eq:Upsilon:LB}
		\Upsilon(\beta)\ge \frac{2}{{\rm L}_\sigma^2},
	\end{equation}
	for there is nothing left to prove otherwise.
	Now consider any such $\beta$.

	Since $b(x)= 0$, $u_0\ge \eta$, and $\sigma(u)\ge{\rm L}_\sigma|u|$,
	we can apply \eqref{heat:mild:form} to deduce that for all $x,y\in\R^d$
	and $t>0$,
	\begin{align}\nonumber
		&\E\left( \left| u_t(x) u_t(y)\right|\right)\\\label{eq:uu:LB}
		&\ \ge \E\left( u_t(x) u_t(y)\right)\\\nonumber
		&\ \ge \eta^2+ \int_0^t\d s\int_{\R^d}
			\d z \int_{\R^d} \d z'\ W_s(z\,,z') p_{t-s}(z-x) p_{t-s}(z'-y) f(z-z'),
	\end{align}
	where
	\begin{equation}\label{eq:W}\begin{split}
		W_s(z\,,z') &:= \E\left( \sigma(u_s(z))\sigma(u_s(z')\right)\\
		&\ge {\rm L}_\sigma^2\E\left( \left| u_s(z) u_s(z') \right| \right).
	\end{split}\end{equation}

	Consider the following $\R_+$-valued functions on $(\R^d)^2$:
	\begin{equation}\begin{split}
		H_\beta(a\,, b) &:= \int_0^\infty \e^{-\beta t}
			\E\left(\left| u_t(a)u_t(b)\right|\right)\d t,\\
		G_\beta(a\,, b) &:=\int_0^\infty \e^{-\beta t}
		p_t(a)p_t(b)\,\d t,\\
		F(a\,, b) &:= f(a-b).
	\end{split}\end{equation}
	Also consider the linear operator $\mathcal{A}_\beta$ defined as follows:
	For all nonnegative Borel-measurable functions $h:(\R^d)^2\to\R_+$,
	\begin{equation}
		(\mathcal{A}_\beta h)(x\,, y) := \left( Fh*\tilde{G}_\beta
		\right)(x\,, y).
	\end{equation}
	A line or two of computation shows that the preceding is
	simply a quick way to write the following:
	\begin{equation}\begin{split}
		&(\mathcal{A}_\beta h)(x\,, y)\\
		&= \int_{\R^d}\,\d a\int_{\R^d}\d b\ F(a\,,b)h(a\,,b)G_\beta(x-a\,,y-b)\\
		&=\int_0^\infty \e^{-\beta t}\,\d t\int_{\R^d}\d a\int_{\R^d}\d b\
			f(a-b) h(a\,,b) p_t(x-a)p_t(y-a).
	\end{split}\end{equation}

	With the preceding definitions under way, we
	can write equation \eqref{eq:uu:LB} in short hand
	as follows:
	\begin{equation}
		H_\beta(x\,, y) \ge \frac{\eta^2}{\beta} +
		{\rm L}_\sigma^2 \left( \mathcal{A}_\beta H_\beta\right)
		(x\,, y).
	\end{equation}
	Since $F\ge 0$ on $(\R^d)^2$, we can apply the preceding to find
	the following pointwise bounds:
	\begin{equation}\begin{split}
		H_\beta &\ge\frac{\eta^2}{\beta}\1+ {\rm L}_\sigma^2 \left(\mathcal{A}_\beta
			\left\{ \frac{\eta^2}{\beta} + {\rm L}_\sigma^2(
			\mathcal{A}_\beta H_\beta)\right\}\right)\\
		&= \frac{\eta^2}{\beta}\1 + {\rm L}_\sigma^2 \frac{\eta^2}{\beta}
			\mathcal{A}_\beta\mathbf{1} +
			{\rm L}_\sigma^4 \mathcal{A}^2_\beta H_\beta,\\
		&\ge \frac{\eta^2}{\beta}\1 + {\rm L}_\sigma^2\frac{\eta^2}{\beta}
			\mathcal{A}_\beta\mathbf{1}
			+{\rm L}_\sigma^4\mathcal{A}^2_\beta\left( \frac{\eta^2}{\beta}
			+ {\rm L}_\sigma^2
			\mathcal{A}_\beta H_\beta\right)\\
		&=\frac{\eta^2}{\beta}\1 + {\rm L}_\sigma^2\frac{\eta^2}{\beta}
			\mathcal{A}_\beta\mathbf{1}
			+{\rm L}_\sigma^4\frac{\eta^2}{\beta}\mathcal{A}^2_\beta\mathbf{1} +
			{\rm L}_\sigma^6\mathcal{A}^3_\beta H_\beta,
	\end{split}\end{equation}
	where $\mathbf{1}(x\,, y):=1$ for all $x,y\in\R^d$. 
	By applying induction we may arrive at the following simplified bound:
	\begin{equation}\begin{split}
		H_\beta 
			&\ge\frac{\eta^2}{\beta}\cdot\sum_{\ell=0}^k{\rm L}_\sigma^{2\ell}
			\mathcal{A}^\ell_\beta\mathbf{1}
			+{\rm L}_\sigma^{2(k+1)}\mathcal{A}_\beta^{k+1}H_\beta\\
		&\ge \frac{\eta^2}{\beta}\cdot\sum_{\ell=0}^k{\rm L}_\sigma^{2\ell}
			\mathcal{A}^\ell_\beta\mathbf{1}.
	\end{split}\end{equation}
	Because the parameter $k\ge 0$
	is arbitrary, it follows that
	\begin{equation}\label{eq:H:LB}
		H_\beta \ge \frac{\eta^2}{\beta}\cdot
		\sum_{\ell=0}^\infty {\rm L}_\sigma^{2\ell} \mathcal{A}^\ell_\beta
		\mathbf{1}.
	\end{equation}
	In order to better understand the behavior of this infinite sum, we
	begin by inspecting only the first few terms.

	The first term in the sum is identically $1$. And the
	more interesting second term can be written as
	${\rm L}_\sigma^2$ multipled by
	\begin{equation}\begin{split}
		(\mathcal{A}_\beta\mathbf{1})(x\,, y) &=
			\left( F*\tilde{G}_\beta\right)(x\,, y)\\
		&=\int_0^\infty\e^{-\beta t}\d t
			\int_{\R^d}\d a\int_{\R^d}\d b\ f(a-b)p_t(a-x)p_t(b-y)\\
		&\ge \frac{1}{(2\pi)^d}\int_0^\infty\e^{-\beta t}\d t
			\int \d\xi\ \hat{f}(\xi)\e^{-2t\Re \Psi(\xi)}\e^{i\xi\cdot (x-y)},
	\end{split}\end{equation}
	thanks to Proposition \ref{pr:KX} [p.\ \pageref{pr:KX}]. 
	We change the order of the double integral $[\d t\,\d\xi]$ to find that
	\begin{equation}
		(\mathcal{A}_\beta\mathbf{1})(x\,, y)
		\ge \frac{1}{(2\pi)^d}\int \frac{\e^{i\xi\cdot(x-y)}
		\hat{f}(\xi)}{\beta+2\Re\Psi(\xi)}\, \d\xi.
	\end{equation}
	[Theorem \ref{th:Dalang:1} and condition \eqref{cond:1} together
	justify the use of Fubini's theorem.]
	
	In order to bound the third term in the infinite sum in \eqref{eq:H:LB},
	we need to estimate the following quantity:
	\begin{equation}\begin{split}
		(\mathcal{A}^2_\beta\mathbf{1})(x\,, y)
			&= \left( F(\mathcal{A}_\beta\mathbf{1})*\tilde{G}_\beta
			\right)(x\,, y)\\
		&\ge\frac{1}{(2\pi)^d}\int\frac{\hat{f}(\xi)\,\d\xi}{%
			\beta+2\Re\Psi(\xi)}Z_\xi(x\,,y),
	\end{split}\end{equation}
	where
	\begin{equation}
		Z_\xi(x\,,y):=\int_0^\infty\e^{-\beta t}\d t
		\int\d a\int\d b\ f(a-b)\e^{i\xi\cdot(a-b)}
		p_t(a-x)p_t(b-y).
	\end{equation}
	[The same ideas that were applied to the second term can be applied
	here, in exactly the same manner, to produce this bound.]
	The Fourier transform of the function
	\begin{equation}
		a\mapsto\exp(i\xi\cdot a)p_t(a-x)
	\end{equation}
	is
	\begin{equation}
		\zeta\mapsto\exp(i(\xi+\zeta)\cdot x-t\Psi(\xi+\zeta)).
	\end{equation}
	Therefore, Proposition \ref{pr:KX} implies the following:
	\begin{equation}\begin{split}
		Z_\xi(x\,,y) &\ge \frac{1}{(2\pi)^d}\int_0^\infty\e^{-\beta t}\d t
			\int \d\zeta\ \hat{f}(\zeta)
			\e^{i(\xi+\zeta)\cdot(x-y)-2t\Re\Psi(\xi+\zeta)}\\
		&= \frac{1}{(2\pi)^d}\int\frac{%
			\e^{i(\xi+\zeta)\cdot(x-y)}\hat{f}(\zeta)}{%
			\beta+2\Re\Psi(\xi+\zeta)}\, \d\zeta.
	\end{split}\end{equation}
	And therefore,
	\begin{equation}\begin{split}
		&(\mathcal{A}^2_\beta\mathbf{1})(x\,, y)\\
		&\ge\frac{1}{(2\pi)^{2d}}\int \frac{\hat{f}(\xi_1)\,
			\d\xi_1}{\beta+2\Re\Psi(\xi_1)}
			\int\frac{\hat{f}(\xi_2)\,\d\xi_2}{\beta+2
			\Re\Psi(\xi_1+\xi_2)}\e^{i(\xi_1+\xi_2)\cdot
			(x-y)}.
	\end{split}\end{equation}

	Now, we can apply induction to deduce that
	we have the following estimate [used to analyse the $\ell$th term
	in the infinite sum in \eqref{eq:H:LB}]
	in general: For all integers $\ell\ge 1$
	and $x,y\in\R^d$,
	\begin{equation}\label{eq:F1}\begin{split}
		&(\mathcal{A}^\ell_\beta\mathbf{1})(x\,, y)\\
		&\ge\frac{1}{(2\pi)^{\ell d}}\int \frac{\hat{f}(\xi_1)\d\xi_1}{%
			\beta+2\Re\Psi(\xi_1)}
			\int\frac{\hat{f}(\xi_2)\,\d\xi_2}{\beta+2\Re\Psi(\xi_1+\xi_2)}
			\cdots\\
		&\hskip1.4in\cdots \int\frac{\hat{f}(\xi_\ell)\,\d\xi_\ell}{
			\beta+2\Re\Psi(\xi_1+\cdots+\xi_\ell)}
			\e^{i\sum_{j=1}^\ell\xi_j\cdot (x-y)}.
	\end{split}\end{equation}

	Although the preceding multiple integral is manifestly nonnegative,
	the integrand itself is complex valued. Fortunately,
	we wish to only understand the behavior of $\mathcal{F}^\ell_\beta\mathbf{1}$
	on the diagonal of $(\R^d)^2$. In that case, the integrand is real and nonnegative.
	As such, we can estimate the integrand directly.
	In order to do so, let us set $y:=x$ in \eqref{eq:F1},
	and then plug the result in \eqref{eq:H:LB}, to find that
	\begin{align} \nonumber
		&\inf_{x\in\R^d}\int_0^\infty \e^{-\beta t}
			\E\left( |u_t(x)|^2\right)\d t\\
		&\hskip.1in\ge  \frac{\eta^2}{\beta}\cdot\sum_{\ell=0}^\infty
			\frac{{\rm L}_\sigma^{2\ell}}{(2\pi)^{\ell d}}
			\int_{\R^d}\d\xi_1\cdots\int_{\R^d}\d\xi_\ell\
			\prod_{j=1}^\ell\frac{\hat f(\xi_j)}{
			\left( \beta+2\Re\Psi(\xi_1+\cdots+\xi_j) \right)}.
	\end{align}
	A change of variables yields the following:
	\begin{equation}\label{eq:INT}\begin{split}
		&\inf_{x\in\R^d}\int_0^\infty \e^{-\beta t}
			\E\left( |u_t(x)|^2\right)\d t\\
		&\hskip.6in\ge \frac{\eta^2}{\beta}\cdot\sum_{\ell=0}^\infty
			\frac{{\rm L}_\sigma^{2\ell}}{(2\pi)^{\ell d}}
			\int_{\R^d}\d z_1\cdots\int_{\R^d}\d z_d\
			\prod_{j=1}^\ell\frac{\hat f(z_j-z_{j-1})}{
			\left( \beta+2\Re\Psi(z_j) \right)},
	\end{split}\end{equation}
	where $z_0:=0$.

	The preceding holds even without Condition \ref{conds:f} [p.\ \pageref{conds:f}].
	But now we recall that Condition \ref{conds:f} is in place, and use it
	to produce the announced lower bound on $\inf_{x\in\R^d}\overline\gamma_x(2)$.

	Define
	\begin{equation}
		\Sigma:=\left\{ x:=(x_1\,,\ldots,x_d)\in\R^d:
		\text{sign}(x_1)=\cdots=\text{sign}(x_d)\right\}.
	\end{equation}
	Equivalently, $\Sigma:=\R_+^d\cup\R_-^d$.
	Because the terms under the product sign in \eqref{eq:INT}
	are all individually nonnegative, Condition \ref{conds:f}
	assures us that the following holds:
	\begin{equation}\begin{split}
		&\inf_{x\in\R^d}\int_0^\infty \e^{-\beta t}\E\left( |u_t(x)|^2\right)\d t\\
		&\hskip.7in\ge \frac{\eta^2}{\beta}\cdot\sum_{\ell=0}^\infty
			\frac{{\rm L}_\sigma^{2\ell}}{(2\pi)^{\ell d}}
			\int_\Sigma\d z_1\cdots\int_\Sigma\d z_d\
			\prod_{j=1}^\ell\frac{\hat f(z_j-z_{j-1})}{
			\left( \beta+2\Re\Psi(z_j) \right)}.
	\end{split}\end{equation}
	If $z_1,\ldots,z_d\in\Sigma$, then the absolute value of the
	$k$th coordinate of
	$z_j-z_{j-1}$ is less than or equal to the absolute value of
	the $k$th coordinate of
	$z_j$ for all $k=1,\ldots d$; therefore
	\begin{equation}
		\hat f(z_j-z_{j-1})\ge \hat f(z_j),
	\end{equation}
	thanks to Condition \ref{conds:f}. Consequently,
	\begin{equation}\label{eq:Sigma}\begin{split}
		&\inf_{x\in\R^d}\int_0^\infty \e^{-\beta t}\E\left( |u_t(x)|^2\right)\d t\\
		&\hskip1.4in\ge\frac{\eta^2}{\beta}\cdot\sum_{\ell=0}^\infty
			\frac{{\rm L}_\sigma^{2\ell}}{(2\pi)^{\ell d}}
			\int_{\Sigma^\ell}\d \bm{z}\
			\prod_{j=1}^\ell\frac{\hat f(z_j)}{
			\left( \beta+2\Re\Psi(z_j) \right)}\\
		&\hskip1.4in=\frac{\eta^2}{\beta}\sum_{\ell=0}^\infty
			\left( \frac{{\rm L}_\sigma^2}{(2\pi)^d}
			\int_{\Sigma}\frac{\hat f(z)}{
			\left( \beta+2\Re\Psi(z) \right)}\,\d z\right)^\ell.
	\end{split}\end{equation}
	In particular, if there exists $\beta>0$ such that
	\begin{equation}\label{eq:Quadrant}
		\frac{1}{(2\pi)^d}\int_\Sigma \frac{\hat{f}(\xi)\,\d\xi}{\beta+2\Re\Psi(\xi)}\ge
		{\rm L}_\sigma^{-2},
	\end{equation}
	then
	\begin{equation}\label{eq:Q}
		\int_0^\infty \e^{-\beta t}\E\left( \left| u_t(x)\right|^2\right)\,\d t
		=\infty\quad\text{for all $x\in\R^d$}.
	\end{equation}
	Thanks to symmetry considerations,
	Condition \ref{conds:f} has the following consequence:
	\begin{equation}\begin{split}
		\int_\Sigma \frac{\hat{f}(\xi)\,\d\xi}{\beta+2\Re\Psi(\xi)}
			&= 2^{-d+1} \int_{\R^d}\frac{\hat{f}(\xi)\,\d\xi}{\beta+2\Re\Psi(\xi)}\\
		&=2^{-d+1} \Upsilon(\beta).
	\end{split}\end{equation}
	This and \eqref{eq:Quadrant} together imply that
	\eqref{eq:Q} holds whenever 
	\begin{equation}
		\Upsilon(\beta)\ge \frac{2^{d-1}}{{\rm L}_\sigma^2}.
	\end{equation}
	
	Now we can apply a real-variable argument to prove that if
	\eqref{eq:Q} holds for some $\beta>0$, then
	\begin{equation}
		\overline\gamma_x(2)\ge\beta\qquad
		\text{for all $x\in\R^d$}.
	\end{equation}
	This and \eqref{eq:Upsilon:LB} together imply the theorem. 
	
	Indeed, let us suppose to the contrary
	that \eqref{eq:Q} holds for our $\beta$,
	and yet $\inf_{x\in\R^d}\overline\gamma_x(2)<\beta$ for the
	very same $\beta$. It follows immediately that there exists $x\in\R^d$,
	 $\delta\in(0\,,\beta)$, and $C\in(0\,,\infty)$ such that 
	\begin{equation}\label{eq:realvar}
		\E\left(|u_t(x)|^2\right) \le
		C \e^{(\beta-\delta)t}
		\qquad\text{for all $t>0$}.
	\end{equation}
	And hence, \eqref{eq:Q} cannot hold in this case. This produces
	a contradiction, and shows that \eqref{eq:Q} implies the theorem.
\end{proof}

\begin{proof}[Proof of Theorem \ref{th:interm:asymp}]
	We begin as we did with \eqref{eq:uu:LB}, but can no longer
	apply the inequality in \eqref{eq:W}. To circumvent that,
	note that for all $q_0\in(0\,,q)$ there exists $A:=A(q_0)\in(0\,,\infty)$
	such that $\sigma(y)\ge q_0|y|$ as soon as $|y|>A$. We have assumed that
	$\P\{u_s(y)>0\}=1$ for all $s>0$ and $y\in\R^d$. Therefore,
	\begin{align} \nonumber
		W_s(z\,,z') &\ge q_0^2 \E\left(  
			u_s(z)u_s(z') ~;~ u_s(z)\wedge u_s(z')>A\right)\\ \nonumber
		&\ge q_0^2\E\left(u_s(z)u_s(z')\right) - q_0^2 A^2
			-q_0A\E\left( u_s(z)~;~u_s(z)>A\right)\\
		&\hskip1.7in-q_0A\E\left( u_s(z')~;~u_s(z')>A\right)\\\nonumber
		&\ge q_0^2\E\left(u_s(z)u_s(z')\right) - q_0^2 A^2
			-q_0A\left\{ \E\left( u_s(z)\right)+\E\left( u_s(z')\right)
			\right\}.\nonumber
	\end{align}
	On the other hand, \eqref{heat:mild:form} guarantees that
	\begin{equation}\begin{split}
		0 &\le \E \left( u_t(x) \right)\\
		&= \left( P_tu_0\right)(x)\\
		&\le \|u_0\|_{L^\infty(\R^d)}.
	\end{split}\end{equation}
	Consequently,
	\begin{equation}
		W_s(z\,,z') \ge q_0^2\left\{ \E\left( u_s(z)u_s(z')\right)
		- A_*\right\},
	\end{equation}
	where
	\begin{equation}
		A_* := \max\left( A^2, 2\|u_0\|_{L^\infty(\R^d)}\right).
	\end{equation}
	Now we apply the recursion argument in the proof of Theorem
	\ref{th:interm}, and find the following pointwise bounds:
	\begin{align}\nonumber
		H_\beta &\ge \frac{\eta^2}{\beta} + q_0^2\left\{
			\mathcal{A}_\beta H_\beta - A_* \mathcal{A}_\beta\mathbf{1}
			\right\}\\\nonumber
		&\ge \frac{\eta^2}{\beta} + q_0^2\left(
			\frac{\eta^2}{\beta}- A_*\right)\mathcal{A}_\beta\mathbf{1}
			+ q_0^4\mathcal{A}_\beta^2H_\beta -q_0^4A_*\mathcal{A}_\beta^2\mathbf{1}\\
		&\ \vdots\\\nonumber
		&\ge\frac{\eta^2}{\beta} + \left( \frac{\eta^2}{\beta}-A_*\right)
			\sum_{\ell=1}^N q_0^{2\ell}\mathcal{A}_\beta^\ell \mathbf{1}
			+q_0^{2(N+1)}\left( \mathcal{A}_\beta^{N+1}H_\beta - A_*
			\mathcal{A}_\beta^{N+1}\mathbf{1}\right);
	\end{align}
	valid for every integer $N\ge 1$. We apply the following
	obvious inequalities: $\eta^2/\beta\ge \eta^2/\beta - A_*$ to the first
	term on the right; and
	$H_\beta\ge \eta^2/\beta$ to the last bracketed term,
	to find that for all integers $N\ge 1$,
	\begin{equation}
		H_\beta \ge \left(\frac{\eta^2}{\beta}-A_*\right)
		\sum_{\ell=0}^{N+1} q_0^{2\ell}\mathcal{A}_\beta^\ell\mathbf{1}.
	\end{equation}
	We can now let $N\uparrow\infty$ and apply the same estimate
	that we used to derive \eqref{eq:Sigma}, and deduce the following bound:
	\begin{align}\nonumber
		&\inf_{x\in\R^d}\int_0^\infty \e^{-\beta t}\E\left( |u_t(x)|^2\right)\d t\\\nonumber
		&\hskip1.2in\ge\left(\frac{\eta^2}{\beta}-A_*\right)
			\cdot \sum_{\ell=0}^\infty
			\left( \frac{q_0^2}{(2\pi)^d}
			\int_{\Sigma}\frac{\hat f(z)}{
			\left( \beta+2\Re\Psi(z) \right)}\,\d z\right)^\ell\\
		&\hskip1.2in=\left(\frac{\eta^2}{\beta}-A_*\right)
			\cdot\sum_{\ell=0}^\infty
			\left( \frac{q_0^2}{2^{d-1}}\Upsilon(\beta)\right)^\ell.
	\end{align}
	Thanks to Theorem \ref{th:Dalang:1}, the preceding implies the following
	bound:
	\begin{equation}
		\inf_{x\in\R^d}\int_0^\infty \e^{-\beta t}\E\left( |u_t(x)|^2\right)\d t
		\ge \left(\frac{\eta^2}{\beta}-A_*\right)
		\cdot\sum_{\ell=0}^\infty
		\left( \frac{q_0^2}{2^{d-1}}(\bar R_\beta f)(0)\right)^\ell.
	\end{equation}
	
	Because $(\bar R_0 f)(0)=\infty$, we can find a $\beta_0>0$ such that
	\begin{equation}
		(\bar R_{\beta_0}f)(0)\ge \frac{2^{d-1}}{q_0^2}. 
	\end{equation}
	For that choice of $\beta_0$,
	the preceding sum diverges. Therefore,
	\begin{equation}
		\int_0^\infty\e^{-\beta_0 t}\E(|u_t(x)|^2)\,\d t=\infty
		\quad\text{provided that}\quad
		\eta>\sqrt{\beta_0 A_*}.
	\end{equation}
	This and an elementary real-variable argument; \eqref{eq:realvar} and together prove the theorem.
\end{proof}

\section{The Massive and Dissipative Operators}

David Nualart asked us about the effect of the drift
coefficient $b$ in \eqref{heat} on the intermittent behavior
of the solution to the stochastic heat equation \eqref{heat}.
At present, we have only an answer to this in a special but
physically-interesting family of cases.

Indeed, let us consider the stochastic heat equation
\begin{equation}\label{heat:massive}
	\frac{\partial}{\partial t}u_t(x) = (\sL u_t)(x) +
	\frac{\lambda}{2} u_t(x) + \sigma(u_t(x))
	\dot F_t(x),
\end{equation}
where $x\in\R^d$, $t>0$, $\lambda\in\R$, and $\dot{F}$ is
as before. Moreover, $\sigma:\R\to\R$ is Lipschitz continuous,
also as before. That is, \eqref{heat:massive} corresponds
to the drift-free stochastic heat equation for the \emph{massive
operator} $\sL^{(\lambda)}:=\sL +(\lambda/2)I$ when $\lambda>0$,
the \emph{dissipative operator} $\sL^{(\lambda)}=
\sL -|\lambda/2|I$ when
$\lambda<0$, and the \emph{free operator} $\sL^{(0)}=\sL$ 
when $\lambda=0$.\footnote{%
The operator
$\mathcal{L}^{(\lambda)}$ is also known as the ``relativistic'' form
of $\mathcal{L}$.}
Of course, Theorem \ref{th:existence}---applied
with $b(u):=\lambda u/2$---guarantees us of the existence
and uniqueness of a mild solution to \eqref{heat:massive}.

The operator $\sL^{(\lambda)}$ is the generator of the semigroup
$\{P^{(\lambda)}_t\}_{t\ge 0}$ defined by
\begin{equation}
	(P^{(\lambda)}_t \phi)(x) := \e^{\lambda t/2} (P_t\phi)(x).
\end{equation}
This can be seen immediately by a semi-formal differential
of $t\mapsto P^{(\lambda)}_t$ at $t=0$; and it is easy to
make the argument rigorous as well. The corresponding
``transition functions'' are given by
\begin{equation}
	p^{(\lambda)}_t(y-x) := \e^{\lambda t/2} p_t(y-x).
\end{equation}

Then the domain $\text{Dom}[\sL^{(\lambda)}]$
of the definition of $\sL^{(\lambda)}$ is the same as $\text{Dom}[\sL]$, and
\begin{equation}
	\sL^{(\lambda)} \phi = \sL \phi + \frac{\lambda}{2}\phi
	\qquad\text{for all $\phi\in \text{Dom}[\sL^{(\lambda)}]$}.
\end{equation}

Let $\{P^{*(\lambda)}_t\}_{t\ge 0}$ denote the adjoint 
[or dual, in probabilistic terms] semigroup. That is,
\begin{equation}
	(P^{*(\lambda)}_t \phi)(x) := \e^{\lambda t/2} (P_t^*\phi)(x).
\end{equation}
with corresponding transition functions,
\begin{equation}
	p^{*(\lambda)}_t(y-x) := \e^{\lambda t/2} p_t(x-y).
\end{equation}
And finally there is also a corresponding replica semigroup,
\begin{equation}
	(\bar{P}^{(\lambda)}_t \phi)(x) := \e^{\lambda t} (\bar{P}_t\phi)(x),
\end{equation}
whose resolvent is described by the following:
\begin{equation}
	(\bar{R}^{(\lambda)}_\alpha \phi)(x) := \int_0^\infty \e^{-(\alpha-\lambda)s}
	(\bar P_s\phi)(x)\,\d s\qquad\text{for all $\alpha\ge \lambda$}.
\end{equation}
We might note that $\bar{R}^{(\lambda)}_\alpha f=\bar R_{\alpha-\lambda}f$
is merely
a shift of the the free replica resolvent of $f$. Therefore, the proof
of Theorem \ref{th:existence} goes through unhindered, and
after accounting for the mentioned shift, produces the following:

\begin{theorem}\label{th:existence:massive}
	Suppose $u_0:\R^d\to\R$ is bounded and measurable.
	Then, under Condition \ref{cond:1}, the mild solution to
	\eqref{heat:massive} satisfies
	the following: For all integers $p\ge 2$ and $\lambda\in\R$,
	\begin{equation}\label{eq:exist:nonlinear:massive}
		\overline\gamma_*(p)\le \lambda+\frac p2
			\inf\left\{\alpha>0:\ ( \bar R_\alpha f)(0)
			<\frac{1}{z_p^2\lip_\sigma^2}\right\},
	\end{equation}
	where $ z_p$ denotes the largest 
	positive zero of the Hermite polynomial
	$\text{\sl He}_p$.
\end{theorem}

Recall the function $Q$ of Theorem \ref{th:existence}.
Since
\begin{equation}
	Q(p\,,\beta)\ge\max\left(\frac{p\lambda}{2\beta} ~,\, z_p\lip_\sigma
	\sqrt{(\bar R_{2\beta/p}f)(0)}\right)
\end{equation}
a few lines of arithmetic show that Theorem \ref{th:existence:massive}
provides us with a better upper bound than Theorem \ref{th:existence}
for the top Liapounov $L^p$-exponent of the mild solution to \eqref{heat}.  Next, we produce instances where the solution is intermittent.

First of all, note that according to \eqref{eq:exist:nonlinear:massive},
\begin{equation}
	\overline\gamma_*(2)\le \lambda+ 
	\inf\left\{\alpha>0:\ ( \bar R_\alpha f)(0)
	<\frac{1}{\lip_\sigma^2}\right\},
\end{equation}
because $z_2=2$.
We apply similar ``shifting arguments'' together with Theorem \ref{th:interm}
to deduce that the following offers a converse, under some symmetry and
regularity conditions. We note, in advance, that when $d=1$,
the preceding estimate and the following essentially match up.

\begin{theorem}\label{th:interm:massive}
	Suppose that both Conditions \ref{cond:1} and 
	\ref{conds:f} hold, $\eta:=\inf_{x\in\R^d}u_0(x)>0$, 
	and there exists ${\rm L}_\sigma\in(0\,,\infty)$ such that
	$\sigma(z)\ge{\rm L}_\sigma|z|$ for all $z\in\R$. Then,
	\begin{equation}\label{eq:cond:interm:massive}
		\inf_{x\in\R^d}\overline\gamma_x(2)\ge \lambda+ \sup\left\{\alpha>0:\
		(\bar R_\alpha f)(0)\ge \frac{2^{d-1}}{{\rm L}_\sigma^2}
		\right\},
	\end{equation}
	where $\sup\varnothing:= 0$.
\end{theorem}

We end this chapter with the example mentioned on page \pageref{ouch} of the Introduction.

\begin{example}\label{ex:massdiss}
	Consider the case that $\sL=-(-\Delta)^{q/2}$ is a power of
	the Laplacian. In that case, $\sL$ is the generator of an isotropic
	stable process of index $q$, and $q\in(0\,,2]$ necessarily.
	Consider also the case that $f(x)=\|x\|^{-d+b}$
	is a Riesz kernel, where $b\in(0\,,d)$. Then [see \eqref{eq:Riesz:FT}
	on page \pageref{eq:Riesz:FT}],
	\begin{equation}\label{rieszhat}
		\hat f(\xi) = \frac{C_{d,b}}{\|\xi\|^b},
		\quad\text{where}\quad C_{d,b}:=
		\frac{\pi^{d/2}2^b \Gamma(b/2)}{\Gamma((d-b)/2)}.
	\end{equation}
	And therefore,
	\begin{equation}\begin{split}
		(\bar R_\alpha f)(0)&=\Upsilon(\alpha)
			\hskip2.1in\text{[by Theorem \ref{th:Dalang:1}]}\\
		&=\frac{C_{d,b}}{(2\pi)^d}\cdot\int_{\R^d}
			\frac{\|\xi\|^{-b}}{\alpha + 2\|\xi\|^q}\,\d\xi\\
		&=\frac{C_{d,b}\alpha^{-1+(d-b)/q}}{(2\pi)^d\cdot 2^{(d-b)/q}}\cdot\int_{\R^d}
			\frac{\d z}{\|z\|^b+ \|z\|^{q+b}}
			\quad[z:=(2/\alpha)^{1/q}\xi].
	\end{split}\end{equation}
	In other words,
	\begin{equation}
		(\bar R_\alpha f)(0)= \frac{\mathfrak{A}_{d,q,b}}{\alpha^{1-\nu}},
	\end{equation}
	where
	\begin{equation}
		\nu:= \frac{d-b}{q}
		\quad\text{and}\quad
		\mathfrak{A}_{d,q,b}:=\frac{C_{d,b}}{(2\pi)^d \cdot 2^\nu}\cdot\int_{\R^d}
		\frac{\d z}{\| z\|^b+\|z\|^{q+b}}.
	\end{equation}
	
	Since $b\in(0\,,d)$, $\mathfrak{A}_{d,q,b}$---and hence
	$(\bar R_\alpha f)(0)$---is finite if and only if $q+b>d$; this is
	the sufficient condition for the existence of a unique mild solution
	to the resulting stochastic PDE. And
	not surprisingly,
	this ``$q+b>d$'' condition
	is the necessary and sufficient condition for the existence of a solution
	to the linear equation 
	[Theorem \ref{th:example:riesz}, p.\ \pageref{th:example:riesz}]. Moreover,
	when $q+b>d$, we can apply
	Theorems \ref{th:existence:massive} and \ref{th:interm:massive}
	to find that for all $x\in\R^d$,
	\begin{equation}
		\lambda+\left(\frac{\mathfrak{A}_{d,q,b}
		{\rm L}_\sigma^2}{2^{d-1}}\right)^{1/(1-\nu)} \le
		\overline\gamma_x(2) \le \lambda+\left(\mathfrak{A}_{d,q,b}
		\lip_\sigma^2\right)^{1/(1-\nu)}.
	\end{equation}
	In particular, consider the massive/dissipative
	``parabolic Anderson model,''
	\begin{equation}
		\frac{\partial}{\partial t}u_t(x) = -\left( (\Delta)^{q/2} u_t \right)(x)
		+\frac{\lambda}{2}\, u_t(x)+\kappa u_t(x)\dot{F}_t(x),
	\end{equation}
	where $u_0:\R\to\R$ is measurable and bounded uniformly away from zero
	and infinity, and $\kappa\neq 0$. The preceding discussion shows that
	the parabolic Anderson model has a solution if 	$q+b>d$. And when $q+b>d$, we obtain the 
	following bounds for the upper $L^2$-Liapounov exponent of the solution:
	For all $x\in\R^d$,
	\begin{equation}
		\lambda+\left(\frac{\mathfrak{A}_{d,q,b}\kappa^2}{2^{d-1}}\right)^{1/(1-\nu)} \le
		\overline\gamma_x(2) \le \lambda+\left(\mathfrak{A}_{d,q,b}
		\kappa^2\right)^{1/(1-\nu)}.
	\end{equation}
	Theorem \ref{th:existence} shows also that $\overline\gamma_*(p)<\infty$
	for all $p\ge 2$. Therefore, we have no weak intermittency if
	$\lambda\le -(\mathfrak{A}_{d,q,b}\kappa^2)^{1/(1-\nu)}$, whereas
	there is weak intermittency if 
	$\lambda > -(\mathfrak{A}_{d,q,b}\kappa^2/(2^{d-1}))^{1/(1-\nu)}.$
	Our condition is sharp when, and only when, $d=1$. In that one-dimensional
	case, we have a solution if and only if $q+b>1$, and if this inequality
	holds then
	\begin{equation}\begin{split}
		\mathfrak{A}_{1,q,b} &=\frac{C_{1,b}}{2^\nu \pi}\cdot\int_0^\infty
			\frac{{\rm d}z}{z^b+z^{b+q}}\\
		&=\frac{C_{1,b}}{2^\nu\pi q}
		\cdot {\rm B}\left(\tfrac{1-b}{q}\,,1-\tfrac{1-b}{q}\right),
	\end{split}\end{equation}
	where ${\rm B}(x\,,y):=\Gamma(x)\Gamma(y)/\Gamma(x+y)$ is the 
	beta  function. Thanks to \eqref{rieszhat}, this implies that
	\begin{equation}	
		\mathfrak{A}_{1,q,b} =\frac{2^{b-\nu}\Gamma(b/2)}{
		\pi^{1/2} q\Gamma((1-b)/2)}\,{\rm B}\left(\tfrac{1-b}{q}\,,1-\tfrac{1-b}{q}\right),
	\end{equation}
	and we find that weak intermittency holds if and only if
	\begin{equation}\label{eq:relativistic}
		\lambda > - \left( \mathfrak{A}_{1,q,b}\,\kappa^2 \right)^{1/(1-\nu)}.
	\end{equation}
	Another simple though tedious computation
	shows that this example [applied with $q:=2$] includes 
	the material that led to \eqref{ouch} of page
	\pageref{ouch}.
	
\end{example}

%% file: FoonKhosh-NL2.tex
\chapter{Temperate Solutions to Parabolic SPDEs}
\label{ch:temperate}

In this chapter we continue the study of the stochastic heat equation 
of the following form:
\begin{equation}\label{heat:temperate}
	\frac{\partial}{\partial t}u_t(x)=
	(\mathcal{L}u_t(x))+ \sigma(u_t(x))\dot{F}_t(x),
\end{equation}
where $\sigma:\R^d\to\R$ is Lipschitz continuous, as before.
But here $u_0$ can be a finite Borel measure.  

Such a problem arises naturally in various specific
cases; see, for example, Bertini and Cancrini \cite{BertiniCancrini},
Carmona and Molchanov \cite{CarmonaMolchanov:94}, and
Molchanov \cite{Molchanov}. It is also motivated by the physical literature
on statistical mechanics, where $u_0$ typically represents the initial
[probability] distribution of a particle system in a random environment
that interacts with it random environment. In fact, the role of
$u_0$ as a measure is typically taken for granted and 
little or no mention of $u_0$ is made explicitly in the physics
literature \cite{KrugSpohn,Kardar,KPZ}.

Thus, our present goal is to continue and produce solutions to \eqref{heat:temperate}, even though the initial
data is a measure. 

The fact that $u_0$ is no longer a function poses some
fundamental problems for the existing theory of SPDEs. In fact,
we argue next that in order to make sense of \eqref{heat:temperate}
in the case that $u_0$ is a finite Borel measure, we need to impose
regularity conditions on both the semigroup $\{P_t\}_{t>0}$
and the initial measure. More importantly, we need to develop a slightly less
restrictive concept of a solution than the well known, as well
as standard, notion of a mild solution.

Let us begin with a few observations.  

It follows from the properties of the Walsh stochastic integral
\cite[Chapter 2]{Walsh} that if $\bm{u}$ is a predictable random field
which is the mild solution of \eqref{heat:mild:form} [with $b\equiv 0$] and 
satisfies \eqref{eq:mild}, then
\begin{equation}\begin{split}
	&\E\left(\left| u_t(x)\right|^2\right)\\
	& \hskip.4in= \left| (P_tu_0)(x)\right|^2
		+ \E\left(\left|\int_0^t\int_{\R^d} p_{t-s}(y-x)
		\sigma(u_s(y))\, F(\d s\,\d y)\right|^2\right)\\
	& \hskip.4in\ge \left| (P_tu_0)(x)\right|^2,
\end{split}\end{equation}
where as usual,
\begin{equation}
	(P_tu_0)(x) :=\int_{\R^d} p_t(y-x)\,u_0(\d y).
\end{equation}
Thus, in particular, \eqref{eq:mild} implies that the following
is a necessary condition for the existence of a mild solution:
\begin{equation}
	\sup_{t\in (0,T)}\sup_{x\in\R^d}\left|
	(P_tu_0)(x)\right|<\infty\qquad\text{for all $T>0$}.
\end{equation}
It is clear that the latter condition rules out many 
interesting choices of $u_0$. 
For instance, consider the case that
$\sL=\Delta$ and one of the most natural
choices $u_0:=\delta_z$ for initial data. In that case,
\begin{equation}
	(P_t\delta_z)(a)=\frac{\e^{-\|a-z\|^2/(4t)}}{(4\pi t)^{d/2}}, 
\end{equation}
and hence
\begin{equation}
	\sup_{t\in(0,T)}\sup_{x\in\R^d} |(P_t\delta_z)(x)|=\infty
	\qquad \text{for all $T>0$ and $z\in\R^d$}.
\end{equation}
Thus, we can never have
a mild solution to \eqref{heat:temperate} when the initial
data is a point mass. But when $\sigma(u)=\text{const}\cdot u$
and $\sL=\Delta$, \eqref{heat:temperate} is expected to have a solution 
defined by the Feynman--Kac formula in many cases
where $f$ is ``nice''; see Bertini and Cancrini \cite{BertiniCancrini}
and Carmona and Molchanov \cite{CarmonaMolchanov:94}, for example. 
It is therefore natural to try to use a different definition of a solution to handle such problems. 

It turns out that one can turn the suggested strategy into a fruitful 
rigorous approach.
In fact, in this chapter we consider a different notion of solution which satisfies,
among other things, the following condition in place of the more restrictive
condition \eqref{eq:mild}:
\begin{equation}\label{eq:mild:1}
	\sup_{x\in\R^d}\E\left(|u_t(x)|^2\right)<\infty
	\qquad\text{for almost every $t>0$}.
\end{equation}

Now suppose we want to know when \eqref{heat:temperate}, 
with $u_0=\delta_z$, has a random-field solution which 
satisfies \eqref{eq:mild:1}. Similar 
considerations as those of the previous paragraph tell 
us that a necessary condition is that
$p_t$ is a bounded function for almost all $t>0$. Thanks to
Hawkes's theorem [Proposition \ref{pr:hawkes},
p.\ \pageref{pr:hawkes}], \eqref{eq:mild:1} implies
that
\begin{equation}\label{cond:hawkes}
	\exp(-\Re\Psi)\in L^t(\R^d)
	\qquad\text{for all $t>0$.}
\end{equation}
Thus, we will will need to assume, at the very least,
that \eqref{cond:hawkes} holds.

Recall that a mild solution $\bm{u}$ of \eqref{heat:temperate}
is a predictable random field that satisfies \eqref{heat:mild:form}
with $b\equiv 0$.
To be specific, one requires that for all nonrandom
pairs $(t\,,x)\in(0\,,\infty)\times\R^d$,
the random equation \eqref{heat:mild:form} holds almost surely.
The following is a slightly less stringent notion of a solution.

\begin{definition}
	\index{Temperate solution}%
	Let $\bm{u}:=\{u_t(x)\}_{t>0,x\in\R^d}$ be a predictable
	random field. We say that $\bm{u}$ is a \emph{temperate
	solution} to \eqref{heat:temperate} if there exists a null set
	$N_0\subset\R_+$ such that for all $t\not\in N_0$,
	the following holds for every $x\in\R^d$:
	\begin{equation}\label{def:temperate}
		u_t(x)=(P_tu_0)(x)+\int_0^t\int_{\R^d} p_{t-s}(y-x)
		\sigma(u_s(y))\, F(\d s\,\d y)
		\quad\text{a.s.}
	\end{equation}
\end{definition}

As we shall see in the next section, the stochastic integral in \eqref{def:temperate}
can be defined properly though, 
technically speaking, it is not a Walsh integral.
Regardless, we have the following, which is the main result of this chapter.

\begin{theorem}\label{th:temperate}
	Suppose $\sigma:\R^d\to\R$ is Lipschitz continuous and
	$u_0$ is a finite Borel measure on $\R^d$ such that
	\begin{equation}\label{eq:temperate}
		\int_{\R^d} \frac{|\hat u_0(\xi)|+\hat f(\xi)}{1+2\Re\Psi(\xi)}\,\d\xi<\infty.
	\end{equation}
	If, in addition, \eqref{cond:hawkes} holds, 
	then there exists a temperate solution $\bm{u}:=
	\{u_t(x)\}_{t>0,x\in\R^d}$ 
	to \eqref{heat:temperate}.
\end{theorem}

There is also a corresponding [limited] uniqueness result
for these temperate solutions; see Proposition \ref{pr:uniqueness}.

Let us remark that \eqref{eq:temperate} holds if and only if:
\begin{enumerate}
	\item[(a)] Condition \eqref{cond:1} holds; and 
	\item[(b)] 
	\begin{equation}
		\int_{\R^d}\frac{|\hat u_0(\xi)|}{1+2\Re\Psi(\xi)}\,\d\xi<\infty.
	\end{equation}
\end{enumerate}
In addition, in the case that $u_0$ is a positive-definite function,
the preceding display holds if and only if $u_0$ has bounded potentials;
more precisely, that $(\bar R_\alpha u_0)(0)=
\sup_{x\in\R^d}(\bar R_\alpha u_0)(x)<\infty$ for all $\alpha>0$;
see Theorem \ref{th:Dalang:1} for a more precise statement.

\begin{corollary}\label{cor:temperate}
	Suppose $d=1$, $\sigma:\R\to\R$ is Lipschitz continuous,
	and Condition \ref{cond:1} and the following are both valid:
	\begin{equation}\label{cond:dalang}
		\int_{-\infty}^\infty \frac{\d\xi}{1+2\Re\Psi(\xi)}<\infty.
	\end{equation}
	Then, \eqref{heat:temperate} has a temperate solution for every
	finite initial measure $u_0$.
\end{corollary}

\begin{remark}
	Condition \eqref{cond:dalang} is equivalent to the property that
	the replica process $\bar X$ has local times; see
	Hawkes \cite{Hawkes:86}. In that case,
	it can be shown, as in our earlier work with
	E. Nualart \cite{FKN}, that \eqref{heat:temperate} has
	a temperate solution even when $\dot F$ is space-time white
	noise. We have stated Corollary \ref{cor:temperate} for $d=1$
	only because \eqref{Bochner} [see p.\ \pageref{Bochner}]
	implies fairly readily that 
	\eqref{cond:dalang} can never hold when $d\ge 2$.
	\qed
\end{remark}

\begin{remark}
	Corollary \ref{cor:temperate} has the following remarkable consequence:
	When \eqref{cond:dalang} holds, the stochastic heat equation
	\eqref{heat:temperate} has a temperate solution for 
	every finite measure $u_0$; this includes every $u_0\in L^1(\R)$. This property fails
	to hold for a mild solution that satisfies the usual integrability
	condition \eqref{eq:mild}.\qed 
\end{remark}

We conclude this section with our final main result concerning temperate 
solutions to \eqref{heat:temperate}.  This result yields an upper 
bound for the growth of the solution.  
Since our solution is not in general a mild one, we will need to
also adapt our previously-used notion of \emph{Liapounov exponents}. 
Thus, we introduce the following, which seems to have a 
number of desirable mathematical properties.

\begin{definition}
	We define the \emph{integrated $L^p(\P)$-Liapounov exponent}
	$\lambda_*(p)$ of the temperate solution to \eqref{heat:temperate}---when
	it exists---as
	\begin{equation}\label{def:Gamma*:p}
			\lambda_*(p) := \inf\left\{
			\beta>0:\ \int_0^\infty \e^{-\beta t}
			\left\{ \sup_{x\in\R^d}
			\E\left(\left| u_t(x) \right|^p\right) \right\}^{2/p}
			\,\d t<\infty\right\},
	\end{equation}
	where $\inf\varnothing:=\infty$.
\end{definition}

\begin{theorem}\label{th:interm:temperate}
	Suppose \eqref{cond:hawkes} holds, 
	$\sigma:\R\to\R$ is Lipschitz continuous and nonrandom,
	and $u_0$ is a nonrandom finite Borel measure on $\R^d$
	that satisfies \eqref{eq:temperate}. Then the temperate
	solution $\bm{u}$ to \eqref{heat:temperate} satisfies
	the following:
	\begin{equation}\label{eq:Gamma*:p}
		\lambda_*(p) \le \inf\left\{ \beta>0:\
		(\bar R_\beta f)(0) <\frac{1}{2z_p^2\lip_\sigma^2}\right\}.
	\end{equation}
\end{theorem}

According to Theorem \ref{th:Dalang:1}, the integral condition \eqref{eq:temperate}
implies that $(\bar R_\alpha f)(0)<\infty$ for all $\alpha>0$, whence
$\lim_{\beta\to\infty}(\bar R_\beta f)(0)=0$, thanks to 
the dominated convergence theorem. Consequently, 
Theorem \ref{th:interm:temperate}
implies, among other things, that $\lambda_*(p)<\infty$ 
for all $p\ge 2$.

It would be interesting to find nontrivial conditions
on $f$ and $\mathcal{L}$ that ensure that the preceding
temperate solution satisfies $\lambda_*(2)>0$, and in this way
derive a version of weak intermittency for temperate solutions.
Such an undertaking would  also
require nondegeneracy conditions on $u_0$. But it seems hard to find nontrivial conditions that can be placed on
an initial measure $u_0$ that guarantee the strict positivity of
$\lambda_*(2)$ without assuming that $u_0$ is a bounded
measurable function [and not just a measure] which is also bounded
away from zero. 

We end this section by making a final observation about weak intermittency.

Suppose $(\bar R_0 f)(0)<\infty$. Then, $\lambda_*(p)=0$
provided that
\begin{equation}\label{eq:temperate:noninterm}
	\lip_\sigma<\frac{1}{\sqrt{2z_p^2(\bar R_0 f)(0)}}.
\end{equation}
We can apply \eqref{eq:temperate:noninterm} with $p=2$ to deduce that
 if $\sigma$ is sufficiently smooth in the sense that 
\begin{equation}
	\lip_\sigma<\frac{1}{\sqrt{2(\bar R_0 f)(0)}},
\end{equation}
then $\lambda_*(2)=0$, and ``weak intermittency'' [with respect to
the new Liapounov exponents $\{\lambda_*(p)\}_{p>2}$]
does not hold.

\section{Stochastic Convolutions}

Define for all predictable random fields $Z:=\{Z_t(x)\}_{t>0,x\in\R^d}$,
\begin{equation}
	(\tilde{p}*Z\dot F)_t(x) := \int_0^t\int_{\R^d}
	p_{t-s}(y-x)Z_s(y)\, F(\d s\,\d y),
\end{equation}
provided that the preceding Walsh-type stochastic integral exists.
Note, in particular, that whenever $v:=\{v_t(x)\}_{t>0,x\in\R^d}$
is a predictable random field, 
\begin{equation}
	\mathcal{A}v=\tilde{p}*(\sigma\circ v)\dot F,
\end{equation}
provided that the Walsh stochastic integrals are well defined.

According to the theory of Walsh \cite{Walsh},
the stochastic-convolution process $\tilde{p}*Z\dot F$ is well defined
provided that the quantity
\begin{equation}
	\label{eq:E(p*ZF2)}
	\int_0^t\d s\int_{\R^d}\d y\int_{\R^d}\d z\
	\E\left(|Z_s(y)Z_s(z)|\right)
	p_{t-s}(y-x)p_{t-s}(z-x) f(y-z)
\end{equation}
is finite for all choices of $(t\,,x)\in(0\,,\infty)\times\R^d$.
In that case,
\begin{align}\nonumber
	&\E\left(\left| (\tilde{p}*Z\dot F)_t(x)\right|^2\right)\\
	&=\int_0^t\d s\int_{\R^d}\d y\int_{\R^d}\d z\
		\E\left(Z_s(y)Z_s(z)\right)
		p_{t-s}(y-x)p_{t-s}(z-x) f(y-z).
	\label{eq:E(p*ZF2):1}
\end{align}

Define for all $\beta>0$ and all predictable random
fields $\bm{v}:=\{v_t(x)\}_{t>0,x\in\R^d}$,
\index{000Nbeta@$\mathcal{N}_\beta$, a family of Hilbertian norms}%
\begin{equation}\label{eq:N:beta}
	\mathcal{N}_\beta(\bm{v}) := \left(
	\int_0^\infty \e^{-\beta t}\sup_{x\in\R^d}
	\E\left(| v_t(x)|^2\right)\,\d t\right)^{1/2}.
\end{equation}
Each $\mathcal{N}_\beta$ is a norm on equivalence classes
of square-integrable predictable processes that are
modifications of one another.

\begin{definition}
	Let $Z:=\{Z_t(x)\}_{t>0,x\in\R^d}$ be a random field.
	We say that $Z$ is \emph{$(p, F)$-Walsh integrable} if 
	$Z$ is predictable and the quantity in
	\eqref{eq:E(p*ZF2)} is finite.
	We frequently abuse notation and write ``Walsh integrable'' in place of 
	``$(p,F)$-Walsh integrable.''
\end{definition}
In other words, Walsh-integrable random fields are random fields
for which the stochastic-convolution process $\tilde{p}*Z\dot F$
is defined by the stochastic integration theory of Walsh \cite{Walsh}. 

The following is a key ``embedding'' theorem for our
extension of stochastic convolutions.

\begin{lemma}\label{lem:P*ZF}
	If $Z:=\{Z_t(x)\}_{t>0,x\in\R^d}$ is Walsh integrable,
	then
	\begin{equation}
		\mathcal{N}_\beta(\tilde{p}*Z\dot F) \le
		\mathcal{N}_\beta(Z)\cdot\sqrt{(\bar R_\beta f)(0)}
		\quad\text{for all $\beta>0$}.
	\end{equation}
\end{lemma}

\begin{proof}
	Owing to \eqref{eq:E(p*ZF2):1}, the following is valid for
	all $t>0$ and $x\in\R^d$:
	\begin{align} \nonumber
		&\E\left(\left| (\tilde{p}*Z\dot F)_t(x)\right|^2\right)\\\nonumber
		&\le\int_0^t\sup_{a\in\R^d}\E\left(\left| Z_s(a)\right|^2\right)\,
			\d s\int_{\R^d}\d y\int_{\R^d}\d z\
			p_{t-s}(y-x)p_{t-s}(z-x)f(y-z)\\
		&=\int_0^t\sup_{a\in\R^d}\E\left(\left| Z_s(a)\right|^2\right)\,
			(\bar P_{t-s}f)(0)\,\d s.
	\end{align}
	Since the right-most quantity is independent of $x\in\R^d$,
	we can take the supremum over all $x\in\R^d$, multiply the preceding
	by $\exp(-\beta t)$ and then integrate $[\d t]$ to deduce the lemma.
\end{proof}

We now use Lemma \ref{lem:P*ZF} to extend the Walsh definition
of a stochastic convolution as follows.  Suppose $(\bar R_\alpha f)(0)<\infty$
for some $\alpha>0$; because of Theorem \ref{th:Dalang:1}
we know that
$(\bar R_\beta f)(0)<\infty$ for all $\beta>0$.
\index{000L2beta@$\bm{L}^2_\beta$, an $L^2$ space of predictable random fields}%
Define $\bm{L}^2_\beta$ to be the collection of all [equivalence classes of
modifications of] predictable random fields
$Z$ such that $\mathcal{N}_\beta(Z)<\infty$.\footnote{As is customary
in the study of $L^p$ spaces, we treat the elements of $\bm{L}^2_\beta$  
slightly carelessly
as if they are random fields, rather than equivalence classes of random fields. This is
done, as in the case of $L^p$ spaces, merely to simplify the otherwise-cumbersome
notation. There should be no loss in precision, as ought to be clear from
the context.} Clearly, $\bm{L}^2_\beta$
is a Banach space for every $\beta>0$.

\begin{lemma}\label{lem:density:Lbeta}
	If $Z\in\bm{L}^2_\beta$ for some $\beta>0$, then there exist $Z^1,Z^2,\dots\in
	\bm{L}^2_\beta$ such that $\lim_{n\to\infty}
	Z^n= Z$ in $\bm{L}^2_\beta$, and $Z^n$ is Walsh integrable for all $n\ge 1$.
\end{lemma}

\begin{proof}
	We may assume without loss of generality that $Z_t(x)\ge 0$
	for all $t>0$ and $x\in\R^d$; for otherwise we can consider
	the predictable random fields $(Z_t(x))^+$ and $(Z_t(x))^-$
	separately [they too are in $\bm{L}^2_\beta$]. 
	
	For all $t>0$,
	$x\in\R^d$, and $n\ge 1$, let
	$Z^n_t(x)$ denote the minimum of $Z_t(x)$ and $n$. Then
	$Z^n$ converges in $\bm{L}^2_\beta$ to $Z$ as $n\to\infty$,
	thanks to the dominated convergence theorem. And 
	\begin{align}\nonumber
		&\int_0^t\d s\int_{\R^d}\d y\int_{\R^d}\d z\
			\E\left(|Z_s^n(y)Z_s^n(z)|\right)
			p_{t-s}(y-x)p_{t-s}(z-x) f(y-z)\\\nonumber
		&\hskip1.1in\le n^2\int_0^t\d s\int_{\R^d}\d y\int_{\R^d}\d z\
			p_s(y-x)p_s(z-x) f(y-z)\\
		&\hskip1.1in= n^2\int_0^t (\bar P_sf)(0)\,\d s\\\nonumber
		&\hskip1.1in\le n^2 \e^{\beta t}(\bar R_\beta f)(0),
	\end{align}
	which is finite.
\end{proof}

Now let us choose and fix some $Z\in\bm{L}^2_\beta$.
According to Lemma \ref{lem:density:Lbeta} we can find
$Z^n\to Z$ [in $\bm{L}^2_\beta$] such that every $Z^n$
is $(p, F)$-Walsh integrable. Lemma \ref{lem:P*ZF} tells us that
$\lim_{n\to\infty} (\tilde{p}*Z^n\dot F)$ exists in $\bm{L}^2_\beta$.
Thus, we can define the stochastic-convolution process
$\tilde{p}*Z\dot F$ for every $Z\in\bm{L}^2_\beta$ as follows:
\begin{equation}
	(\tilde{p}*Z\dot F)_t(x) := \lim_{n\to\infty} 
	(\tilde{p}*Z^n\dot F)_t(x),
\end{equation}
where the limit takes place in $\bm{L}^2_\beta$. Of course,
the preceding limit does not depend on the choice of $\{Z^n\}_{n=1}^\infty$.
But it also does not depend on $\beta$ because 
\begin{equation}
	\bm{L}^2_\beta\subseteq\bm{L}^2_\alpha
	\qquad\text{if $\alpha\le\beta$.}
\end{equation}

We abuse notation slightly and write for all $t>0$ and $x\in\R^d$,
\begin{equation}
	\int_0^t\int_{\R^d} p_{t-s}(y-x)Z_s(y)\, F(\d s\,\d y):=
	(\tilde{p}*Z\dot F)_t(x)
	\quad\text{for $Z\in\bigcup_{\beta>0}\bm{L}^2_\beta$}.
\end{equation}
It is easy to see that many of the standard properties of the
Walsh stochastic convolution transfer, by limiting arguments,
to those of the present extension. Chief among them is the fact
that $Z\mapsto (\tilde{p}*Z\dot F)$ is a random linear map from
$\bm{L}^2_\beta$ to itself. In other words, if $Z,W\in\bm{L}^2_\beta$
and $a,b\in\R$, then 
\begin{equation}\label{eq:id:temp}
	\tilde{p}*(aZ+bW)\dot F = a(\tilde{p}*Z\dot F) + b(\tilde{p}*Z\dot F).
\end{equation}
We emphasize that the preceding means that, as elements of $\bm{L}^2_\beta$,
the two sides agree. In particular, because of the Fubini theorem,
we obtain the following equivalent formulation of \eqref{eq:id:temp}:
There exists a null set $N_0\subset \R_+$ such that for all $t\not\in N_0$
and all $x\in\R^d$,
\begin{equation}
	\P\left\{
	\left(\tilde{p}*\left(aZ+bW\right)\dot F\right)_t(x)= a
	\left(\tilde{p}*Z\dot F\right)_t(x) + b\left(\tilde{p}* W\dot F\right)_t(x)
	\right\}
\end{equation}
is equal to one.

We conclude this section by emphasizing that 
the stochastic convolution process $t\mapsto (\tilde{p}*Z\dot F)_t$
is hereby not defined for all $t>0$. However, it \emph{is}
defined for almost all $t>0$; and if $(\tilde{p}*Z\dot F)_t$
is defined for a given $t>0$, then $(\tilde{p}*Z\dot F)_t(x)$
is well-defined for every $x\in\R^d$, and 
$x\mapsto (\tilde{p}*Z\dot F)_t(x)$ is a random field in the usual
sense.

Our stochastic convolution is thus quite different
from the existing ones in the literature.
For example, for the most commonly-used infinite-dimensional
stochastic convolution [see, for example, Chapter 5 of Da Prato
and Zabczyk \cite[\S5.1.2]{DaPratoZabczyk}], each random process
$x\mapsto (\tilde{p}*Z\dot F)_t(x)$
is defined for every $t>0$, typically as an element of a certain
Hilbert space. But, in that case, the random variables
$(\tilde{p}*Z\dot F)_t(x)$ cannot in general be
defined for every $x$.

\section{A Priori Estimates}
We now proceed to establish \emph{a priori} estimates for
Walsh-type stochastic integrals of the form $\mathcal{A}v$,
thereby showing that the random linear operator
$\mathcal{A}$ maps every $\bm{L}^2_\beta$ to itself
continuously and boundedly. As such, the following 
result should be compared to Lemma \ref{lem:Norm:sigma} on
page \pageref{lem:Norm:sigma}.

\begin{lemma}\label{lem:Norm:sigma:N}
	For all $\beta>0$
	and predictable random fields $v$ and $w$,
	\begin{equation}
		\mathcal{N}_\beta(\mathcal{A}v) \le 
		\left( \frac{
		{\rm C}_\sigma}{\beta^{1/2}} + 
		{\rm D}_\sigma \mathcal{N}_\beta(v)\right)\sqrt{%
		2( \bar R_\beta f)(0)},
	\end{equation}
	and
	\begin{equation}
		\mathcal{N}_\beta\left( \mathcal{A}v-\mathcal{A}w\right)\le 
		\lip_\sigma\mathcal{N}_\beta(v-w)
		\sqrt{2( \bar R_\beta f)(0)}.
	\end{equation}
\end{lemma}

\begin{proof}
	According to \eqref{eq:Norm:sigma} and \eqref{eq:Norm:sigma:1}, 
	$\E( | (\mathcal{A}v)_t(x)|^2)$ 	is bounded above by
	\begin{align}\nonumber
		&\int_0^t Q_s^2\,\d s\int_{\R^d}\d y
			\int_{\R^d}\d z\  p_{t-s}(y-x)p_{t-s}(z-x)f(z-y)\\
		&\hskip1in=\int_0^t Q_s^2\,\d s\int_{\R^d}\d y
			\int_{\R^d}\d z\  p_{t-s}(y)p_{t-s}(z)f(z-y)\\
		&\hskip1in=\int_0^t Q_s^2 \cdot(\bar P_{t-s}f)(0)\,\d s,
	\end{align}
	where
	\begin{equation}
		Q_s := {\rm C}_\sigma
		+{\rm D}_\sigma\sup_{a\in\R^d}\|v_s(a)\|_2.
	\end{equation}
	Therefore,
	\begin{equation}\begin{split}
		\int_0^\infty  \e^{-\beta t}
			\E\left( | (\mathcal{A}v)_t(x)|^2\right)\,\d t
			&\le \int_0^\infty \e^{-\beta s} Q_s^2\,\d s
			\cdot\int_0^\infty  \e^{-\beta t}(\bar P_t f)(0)\,\d t\\
		& =\int_0^\infty \e^{-\beta s}
			Q_s^2\,\d s\cdot(\bar R_\beta f)(0).
	\end{split}\end{equation}
	Since
	\begin{equation}
		Q_s^2 \le 2{\rm C}_\sigma^2 + 2{\rm D}_\sigma^2\sup_{a\in\R^d}
		\E\left(|v_s(a)|^2\right),
	\end{equation}
	the first part of the lemma follows from the Minkowski-type bound
	\begin{equation}
		(|a|+|b|)^{1/2}\le|a|^{1/2}+|b|^{1/2},
	\end{equation}
	valid for all $a,b\in\R$.
	
	The first follows from the second
	as in the proof of Lemma \ref{lem:Norm:sigma}.
\end{proof}

According to Theorem \ref{th:existence}, if $u_0$ and
$v_0$ are bounded and measurable functions, then there
exist a.s.-unique mild solutions $u$ and $v$ to the stochastic
heat equation \eqref{heat:temperate}. Our next lemma shows
that if $u_0$ and $v_0$ are suitably close, then $u$ and
$v$ are as well.

\begin{lemma}\label{lem:effect0}
	Assume Condition \ref{cond:1} holds.
	Let $u$ and $v$ be two mild solutions to \eqref{heat:temperate}
	whose initial functions $u_0,v_0:\R^d\to\R$ 
	are both bounded and in $L^1(\R^d)$. Then there exists $\beta_0$ which depends 
	only on $f$ and $\lip_\sigma$, such that for all $\beta>\beta_0$,
	\begin{equation}
		\mathcal{N}_\beta(u-v) \le \frac{4}{(2\pi)^d}
		\int_{\R^d}\frac{\left|\hat u_0(\xi) -\hat v_0(\xi)\right|}{
		\beta+2\Re\Psi(\xi)}\,\d\xi.
	\end{equation}
\end{lemma}

\begin{proof}
	Throughout we suppose that $c>1$. Now let us consider
	a fixed $x\in\R^d$ and $t>0$, and define $u^n$ and $v^n$ to
	be the respective Picard approximation to $u$ and $v$ at the $n$th stage. Then,
	\begin{equation}\label{eq:temp:1}
		\left\| u_t^n(x) - v_t^n(x) \right\|_2^2
		\le 2\left| (P_tu_0)(x)-(P_tv_0)(x)\right|^2+2\mathcal{Q}_t^{n-1},
	\end{equation}
	where $\mathcal{Q}_t^{n-1}$ is defined as
	\begin{equation}
		\left\| \int_0^t\int_{\R^d} p_{t-s}(y-x)
		\left( \sigma(u_s^{n-1}(y))-\sigma(v_s^{n-1}(y))\right)
		\, F(\d s\,\d y)\right\|_2^2.
	\end{equation}
	Define for all $t>0$ and $n\ge 1$,
	\begin{equation}
		H_t^n := \sup_{a\in\R^d} \E\left(\left| u_t^n(a)-v_t^n(a)
		\right|^2\right).
	\end{equation}
	Then it follows from our construction of stochastic convolutions
	[and \eqref{eq:E(p*ZF2)}] that $\mathcal{Q}_t^{n-1}$ is bounded
	above by
	\begin{equation}\begin{split}
		&\lip_\sigma^2\int_0^t H_s^{n-1}\,\d s
			\int_{\R^d}\d y\int_{\R^d}\d z\ p_{t-s}(y-x)
			p_{t-s}(z-x) f(z-y)\\
		&\hskip2.2in=\lip_\sigma^2\cdot\int_0^t H_s^{n-1}(\bar P_{t-s}f)(0)\,\d s.
	\end{split}\end{equation}
	And hence,
	\begin{equation}
		\int_0^\infty \e^{-\beta t}\mathcal{Q}_t^{n-1}\,\d t
		\le\lip_\sigma^2\cdot\int_0^\infty \e^{-\beta t}H_t^{n-1}\,\d t\cdot
		(\bar R_\beta f)(0).
	\end{equation}
	This and \eqref{eq:temp:1} together prove that
	\begin{equation}\begin{split}
		&\int_0^\infty \e^{-\beta t} H_t^n
			\,\d t \le 2\int_0^\infty \e^{-\beta t}\sup_{x\in\R^d}\left|
			(P_t u_0)(x)-(P_tv_0)(x)\right|^2\,\d t\\
		&\hskip1.7in+2\lip_\sigma^2\cdot\int_0^\infty \e^{-\beta t}H_t^{n-1}\,\d t\cdot
			(\bar R_\beta f)(0).
	\end{split}\end{equation}
	Condition \ref{cond:1} and the monotone convergence theorem
	together imply that $\lim_{\beta\to\infty}(\bar R_\beta f)(0)=0$.
	We can choose $\beta_0$ so large that
	\begin{equation}
		4\lip_\sigma^2(\bar R_{\beta_0} f)(0)<1.
	\end{equation}
	
	Since
	\begin{equation}
		\int_0^\infty \e^{-\beta t} H_t^n\,\d t=
		\{ \mathcal{N}_\beta(u^n-v^n)\}^2,
	\end{equation}
	it follows that for all $\beta>\beta_0$,
	\begin{equation}\begin{split}
		\mathcal{N}_\beta\left( u^n-v^n\right) &\le 2\left(
			\int_0^\infty \e^{-\beta t}\sup_{x\in\R^d}
			\left| (P_tu_0)(x)-(P_tv_0)(x)\right|^2\,\d t\right)^{1/2}\\
		&\le 2\int_0^\infty \e^{-\beta t/2}\sup_{x\in\R^d}
			\left| (P_tu_0)(x)-(P_tv_0)(x)\right|\,\d t,
	\end{split}\end{equation}
	thanks to Minkowski's inequality. Recall from the proof of Theorem of
	\ref{th:existence} that, among other things,
	\begin{equation}
		\|u^n-u\|_2+\|v^n-v\|_2\to 0
		\qquad\text{as $n\to\infty$}.
	\end{equation}
	Therefore,
	Fatou's lemma implies that for all $\beta>\beta_0$,
	\begin{equation}
		\mathcal{N}_\beta\left( u-v\right) 
		\le 2\int_0^\infty \e^{-\beta t/2}\sup_{x\in\R^d}
		\left| (P_tu_0)(x)-(P_tv_0)(x)\right|\,\d t.
	\end{equation}
	
	In order to complete the lemma, we choose $\beta$ as large as needed
	in order to ensure that the preceding inequality holds, and define
	\begin{equation}
		q := u_0-v_0.
	\end{equation}
	Hawkes's theorem [Proposition \ref{pr:hawkes}, p.\ \pageref{pr:hawkes}]
	assures us that $p_t$ and $\hat p_t$ are both in $L^1(\R^d)$;
	in Fourier-analytic language, $p_t$ is in the Fourier [or Wiener] algebra
	\begin{equation}
		A(\R^d):= \left\{ h\in L^1(\R^d):\, \hat h\in L^1(\R^d)\right\},
	\end{equation}
	for all $t>0$. It is easy to see that the Fourier transform is one-to-one
	and onto on $A(\R^d)$. And continuous elements of $A(\R^d)$ 
	are in fact in $C_0(\R^d)$, by 
	the Riemann--Lebesgue lemma.
	Since $u_0$ and $v_0$ are integrable, it follows also that
	$P_tu_0, P_tv_0\in A(\R^d)$ for all $t>0$, and that $P_tu_0$
	and $P_tv_0$ are both continuous. The inversion theorem of Fourier
	analysis can be applied pointwise to the elements of $A(\R^d)$. Therefore,
	\begin{equation}\begin{split}
		\sup_{x\in\R^d}\left| (P_t q)(x)\right|
			&\le  \| \widehat{P_t q} \|_{L^1(\R^d)}\\
		&=\frac{1}{(2\pi)^d}\int_{\R^d}\left| \e^{-t\Psi(-\xi)}\hat q(\xi)
			\right|\,\d\xi\\
		&\le\frac{1}{(2\pi)^d}\int_{\R^d} \e^{-t\Re\Psi(\xi)}
			\left| \hat u_0(\xi)-\hat v_0(\xi)\right|\,\d\xi.
	\end{split}\end{equation}
	And hence,
	\begin{equation}\label{eq:N:beta:p:Ptu0}
		\frac12\int_0^\infty\e^{-\beta t/2}\sup_{x\in\R^d}\left| 
		(P_t q)(x)\right|\,\d t
		\le \frac{1}{(2\pi)^d}\int_{\R^d}\frac{\left| 
		\hat u_0(\xi)-\hat v_0(\xi)\right|}{
		\beta+2\Re\Psi(\xi)}\,\d\xi.
	\end{equation}
	This proves the lemma.
\end{proof}

\section{Proof of Existence}

Having prepared the background material, we are ready
to demonstrate Theorem \ref{th:temperate}.
Before we begin the proof, let us recall the following:
\begin{definition}
	Let $\{\psi_n\}_{n=1}^\infty$ denote a collection
	of measurable functions from $\R^d$ to $\R_+$ such that:
	(i) $n\mapsto \|\psi_n\|_{L^1(\R^d)}$ is uniformly bounded; and
	(ii) $\lim_{n\to\infty} \hat\psi_n=1$ pointwise.
	Then we say that $\{\psi_n\}_{n=1}^\infty$ is a \emph{weak
	mollifier}.
\end{definition}
Clearly, weak mollifiers are mollifiers in the usual sense. But the
converse is not in general true, because weak mollifiers need
not have very good smoothness properties.

We need this definition because our strategy of the proof is the following:
Let $\{\psi_n\}_{n=1}^\infty$ denote a weak  mollifier. Then,
according to  Theorem \ref{th:existence},
we can solve \eqref{heat:temperate} subject to the initial condition
$u_0*\psi_n$, since $u_0*\psi_n$ is a bounded and measurable function.
If $\bm{u}^{(n)}$ denotes the solution, we then proceed to show that
$\bm{u}:=\lim_{n\to\infty}\bm{u}^{(n)}$ exists in a suitable sense
and is a temperate solution to \eqref{heat:temperate}. Now we write
down the details of this argument.

\begin{proof}[Proof of Theorem \ref{th:temperate}]
	Recall that $u_0$ is a finite  Borel measure and
	define $u_0^{(n)}:=\psi_n*u_0$, where $\{\psi_n\}_{n=1}^\infty$
	is a weak mollifier. 
	According to Theorem \ref{th:existence}, there is a solution $\bm{u}^{(n)}$
	to the following SPDE:
	\begin{equation}
		\frac{\partial}{\partial t} u^{(n)}_t(x)=
		(\sL u^{(n)}_t)(x) + \sigma(u^{(n)}_t(x))\dot{F}_t(x),
	\end{equation}
	and the solution is unique up to evanescence. Theorem \ref{th:existence}
	also assures us that
	\begin{equation}
		C_\alpha:=\sup_{t>0}\sup_{x\in\R^d} \e^{-\alpha t}
		\E\left(\left| u_t^{(n)}(x) \right|^2\right)<\infty,
	\end{equation}
	provided that $\alpha>0$ is sufficiently large. Note that if
	$\theta>\alpha$, then
	\begin{equation}\begin{split}
		\mathcal{N}_\theta(u^{(n)}) &\le C_\alpha\int_0^\infty \e^{-(\theta-\alpha)t}
			\,\d t\\
		&=\frac{C_\alpha}{\theta-\alpha}<\infty.
	\end{split}\end{equation}
	That is, $\bm{u}^{(n)}\in\bm{L}^2_\theta$ for all $\theta$ sufficiently large.
	According to Lemma \ref{lem:effect0}, for all $n,m\ge 1$,
	\begin{equation}
		\mathcal{N}_\beta \left( \bm{u}^{(n)}-\bm{u}^{(m)}\right) \le
		\frac{4}{(2\pi)^d}\int_{\R^d}\frac{|\hat u_0(\xi)|}{\beta+2\Re\Psi(\xi)}
		\cdot\left| \hat \psi_n(\xi)-\hat\psi_m(\xi)\right|\,\d\xi,
	\end{equation}
	provided that $\beta>\beta_0$ for a $\beta_0$ that depends only on $f$
	and $\lip_\sigma$. Therefore, it follows that if $\beta>\beta_0$,
	then $\{\bm{u}^{(n)}\}_{n=1}^\infty$ is a Cauchy sequence in
	$\bm{L}^2_\beta$. Let $\bm{u}:=\{u_t(x)\}_{t>0,x\in\R^d}$ denote the limit.
	By definition,
	\begin{equation}
		\lim_{n\to\infty} \bm{u}^{(n)} =\bm{u}
		\qquad\text{in $\bigcap_{\beta>\beta_0}\bm{L}^2_\beta$}.
	\end{equation}
	And therefore, Lemma \ref{lem:Norm:sigma:N} and our extension
	of stochastic convolutions together imply that
	\begin{equation}\label{eq:GG}
		\lim_{n\to\infty}\left(\bm{u}^{(n)} - \mathcal{A}\bm{u}^{(n)}
		\right)=\bm{u} - \mathcal{A}\bm{u}
		\qquad\text{in $\bigcap_{\beta>\beta_0}\bm{L}^2_\beta$},
	\end{equation}
	where $\mathcal{A}\bm{u}^{(n)}$ is defined in \eqref{def:A}
	on page \pageref{def:A}; and $\mathcal{A}\bm{u}$ is defined
	in the same way, but now interpreted as a stochastic convolution in the
	sense of the present chapter.

	In particular,  for all $T>0$,
	\begin{equation}
		\int_0^T \sup_{x\in\R^d}\E\left(
		\left| \left[ u^n_t(x) - \left( \mathcal{A}u^n \right)_t(x)\right]
		-\left[ u_t(x) - \left(
		\mathcal{A}u \right)_t(x)\right]\right|^2\right)\,\d t\to 0,
	\end{equation}
	as $n\to\infty$. This follows simply because
	\begin{equation}
		\int_0^T \kappa(s)\,\d s
		\le \e^{\beta T}\int_0^\infty\e^{-\beta s}\kappa(s)\,\d s,
	\end{equation}
	for all nonnegative measurable $\kappa:\R_+\to\R_+$.

	Since $\bm{u}^{(n)}$ is a mild solution to \eqref{heat:temperate}
	with initial data $u^{(n)}_0=\psi_n*u_0$, Tonelli's theorem implies that
	\begin{equation}
		u^{(n)}_t(x)-(\mathcal{A}u^n)_t(x)= \left( (P_tu_0) * \psi_n \right)(x).
	\end{equation}
	Hawkes's theorem [Proposition \ref{pr:hawkes}, p.\
	\pageref{pr:hawkes}] implies that $P_tu_0$ is uniformly continuous
	for every $t>0$. Therefore, for every $t>0$ fixed,
	\begin{equation}
		\lim_{n\to\infty} (P_tu_0)*\psi_n= P_tu_0
		\qquad\text{uniformly}.
	\end{equation}
	It follows from the preceding and Fatou's lemma that
	for all $T>0$,
	\begin{equation}
		\int_0^T \sup_{x\in\R^d}\E\left(
		\left|  u_t(x) - (P_tu_0)(x) -(\mathcal{A}u)_t(x) \right|^2\right)\,\d t= 0.
	\end{equation}
	This proves the theorem.
\end{proof}

The following is the final result of this section.
It shows that the temperate solution to \eqref{heat:temperate}
is unique among a natural family of possible solutions. It is entirely
possible that one can establish the uniqueness of the temperate solution
among a larger family than that offered below. But we are not
aware of such further improvements.

\begin{proposition}\label{pr:uniqueness}
	The temperate solution $\bm{u}$, provided by the preceding proof,
	does not depend on the choice of the weak mollifier, up to evanescence.
\end{proposition}

\begin{proof}
	Let $\{\kappa_n\}_{n=1}^\infty$ be another weak mollifier,
	and define $\bm{v}$ to be the temperate solution that the preceding
	proof yields based on $\{\kappa_n\}_{n=1}^\infty$; 
	that is, $v^n$ is obtained from $\{\kappa_l\}_{l=1}^\infty$
	in the same way that $u^n$ was obtained from $\{\psi_l\}_{l=1}^\infty$.
	
	Let $\beta_0$ be as in the proof of Theorem \ref{th:temperate}
	[see also Lemma \ref{lem:effect0}],
	and recall that $\beta_0$ does not depend on $\{\psi_n\}_{n=1}^\infty$,
	$\{\kappa_n\}_{n=1}^\infty$, etc.
	In accord with Fatou's lemma,
	\begin{equation}\begin{split}
		\mathcal{N}_\beta(\bm{u}-\bm{v}) &\le\liminf_{n\to\infty}
			\mathcal{N}_\beta\left(u^n-v^n\right)\\
		&\le\frac{4}{(2\pi)^d}\liminf_{n\to\infty}
			\int_{\R^d}\frac{|\hat u_0(\xi)|}{\beta+2\Re\Psi(\xi)}
			\cdot\left|\hat\psi_n(\xi)-\hat\kappa_n(\xi)\right|\,\d\xi,
	\end{split}\end{equation}
	for every $\beta>\beta_0$,
	owing to Lemma \ref{lem:effect0}. And the preceding is zero by 
	the dominated convergence theorem. This proves the proposition.
\end{proof}

\section{More A Priori Estimates}

Let $u_0$ be a finite Borel measure on $\R^d$ that satisfies
\eqref{eq:temperate}.
In order to investigate the large-time behavior of
the temperate solution to \eqref{heat:temperate}, we need
to introduce a suitable family of Banach spaces.
The task of the present section is precisely to do that.
Our analysis of the temperate solution to \eqref{heat:temperate}
will resume after this section.

Define for all $\beta>0$, finite real numbers $p\ge 1$, and predictable random
fields $\bm{v}:=\{v_t(x)\}_{t>0,x\in\R^d}$,
\index{000Npbeta@$\mathcal{N}_{\beta,p}$, a family of norms}%
\begin{equation}\label{eq:N:beta,p}
	\mathcal{N}_{\beta,p}(\bm{v}) := \left(
	\int_0^\infty \e^{-\beta t}\sup_{x\in\R^d}
	\left\{\E\left(| v_t(x)|^p\right)\right\}^{2/p}
	\,\d t\right)^{1/2},
\end{equation}
so that $\mathcal{N}_{\beta,2}(v)$ is the same quantity as
$\mathcal{N}_\beta(v)$; the latter was defined earlier in
\eqref{eq:N:beta}.
 And, just as was the case with $\mathcal{N}_\beta$ when $p=2$,
every $\mathcal{N}_{\beta,p}$ is a norm on equivalence classes
of $p$-times integrable predictable processes that are
modifications of one another.

\begin{definition}
\index{000L2pbeta@$\bm{L}^p_\beta$, an $L^p$ space of predictable random fields}%
	Let $\bm{L}^p_\beta$ denote the collection of all predictable
	random fields $v:=\{v_t(x)\}_{t>0,x\in\R^d}$ such that
	$\mathcal{N}_{\beta,p}(v)<\infty$.
\end{definition}
Thus, our present definition of $\bm{L}^2_\beta$ agrees with
our older one.

The following generalizes Lemma \ref{lem:Norm:sigma:N} to the 
case that $p$ is an integer $\ge 2$.

\begin{lemma}\label{lem:Norm:sigma:N:p}
	For all $\beta>0$, even integers $p\ge 2$,
	and Walsh-integrable random fields $v$ and $w$,
	\begin{equation}
		\mathcal{N}_{\beta,p}(\mathcal{A}v) \le 
		z_p\left( \frac{
		{\rm C}_\sigma}{\beta^{1/2}} + 
		{\rm D}_\sigma \mathcal{N}_{\beta,p}(v)\right)\sqrt{%
		2( \bar R_\beta f)(0)},
	\end{equation}
	and
	\begin{equation}
		\mathcal{N}_{\beta,p}\left( \mathcal{A}v-\mathcal{A}w\right)\le 
		z_p\lip_\sigma\mathcal{N}_{\beta,p}(v-w)
		\sqrt{2( \bar R_\beta f)(0)}.
	\end{equation}
\end{lemma}

\begin{proof}
	The ensuing argument is a generalization of the proof
	of Lemma \ref{lem:Norm:sigma}. Indeed, by
	\eqref{eq:Norm:sigma}, $\E(| (\mathcal{A}v)_t(x) |^p)$ is at
	most $z_p^p$ times
	\begin{equation}
		\left(\int_0^t\d s\int_{\R^d}\d y\int_{\R^d}\d z\
		V_s'(y\,,z)p_{t-s}(y-x)p_{t-s}(z-x)f(y-z)\right)^{p/2},
	\end{equation}
	where
	\begin{equation}\begin{split}
		V_s'(y\,,z) &:= \left\| \sigma\left( v_s(y)\right) \right\|_p
			\cdot\left\| \sigma\left( v_s(z) \right) \right\|_p\\
		&\le \left( {\rm C}_\sigma + {\rm D}_\sigma \sup_{q\in\R^d}
			\|v_s(q)\|_p\right)^2\\
		&\le 2\Theta_s,
	\end{split}\end{equation}
	where
	\begin{equation}
		\Theta_s:= {\rm C}_\sigma^2+ {\rm D}_\sigma^2 \sup_{q\in\R^d}
		\|v_s(q)\|_p^2.
	\end{equation}
	It follows that $\E(| (\mathcal{A}v)_t(x) |^p)$ is bounded above by
	$2^{p/2}z_p^p$ times
	\begin{equation}\begin{split}
		&\E(| (\mathcal{A}v)_t(x) |^p)\\
		&\hskip.2in\le 2^{p/2}z_p^p
			\left(\int_0^t\Theta_s\,\d s\int_{\R^d}\d y\int_{\R^d}\d z\
			p_{t-s}(y)p_{t-s}(z) f(y-z)\right)^{p/2}\\
		&\hskip.2in=2^{p/2}z_p^p
			\left( \int_0^t\Theta_s (\bar P_{t-s}f)(0)\,\d s
			\right)^{p/2},
	\end{split}\end{equation}
	thanks a direct computation that was made already during
	the course of the proof of Lemma \ref{lem:Norm:sigma}.
	
	Let us examine the preceding display next:
	The right-most quantity is independent of the variable
	$x$. Therefore, we can maximize the extreme terms in 
	that display over all $x\in\R^d$, raise the resulting inequality
	to the power $2/p$ on all sides, and then finally
	multiply by $\exp(-\beta t)$ and integrate $[\d t]$ to
	obtain the following:
	\begin{equation}
		\mathcal{N}_{\beta,p}(v) \le z_p
		\sqrt{2\int_0^\infty \e^{-\beta s}\Theta_s\,\d s\cdot
		(\bar R_\beta f)(0)}
	\end{equation}
	The first bound of the lemma follows after a few direct
	computations. The second bound follows from the first
	in a manner that has been pointed out several times already;
	see, for example, the end of the proof of Lemma \ref{lem:Norm:sigma}
	for an outline.
\end{proof}

Note that if $v$ is a Walsh-integrable random field, then so
is $\sigma\circ v$, and
then $\mathcal{A}v$ is none other than the stochastic
convolution $\tilde{p}*(\sigma\circ v)\dot F$. Therefore,
the following is a refinement of Lemma \ref{lem:Norm:sigma:N:p}
that works in the setting of stochastic convolutions.

\begin{lemma}\label{lem:SConvolution:N:p}
	Recall that $\sigma:\R\to\R$ is a predetermined
	nonrandom and Lipschitz-continuous function.
	Choose and fix a real $\beta>0$ and an even integer
	$p\ge 2$. Then, for all predictable random fields $v$
	and $w$,
	\begin{equation}
		\mathcal{N}_{\beta,p}\left( \tilde{p}*(\sigma\circ v)\dot F\right) \le 
		z_p\left( \frac{
		{\rm C}_\sigma}{\beta^{1/2}} + 
		{\rm D}_\sigma \mathcal{N}_{\beta,p}(v)\right)\sqrt{%
		2( \bar R_\beta f)(0)},
	\end{equation}
	and
	\begin{equation}\begin{split}
		&\mathcal{N}_{\beta,p}\left( \tilde{p}*(\sigma\circ v)\dot F
			-\tilde{p}*(\sigma\circ w)\dot F\right) \\
		&\hskip2in\le  z_p\lip_\sigma\mathcal{N}_{\beta,p}(v-w)
			\sqrt{2( \bar R_\beta f)(0)}.
	\end{split}\end{equation}
\end{lemma}

\begin{proof}
	We can, and will, assume without loss of generality that
	the three quantities 
	$(\bar R_\beta f)(0)$,
	$\mathcal{N}_{\beta,p}(v)$,
	and $\mathcal{N}_{\beta,p}(v-w)$ are all finite.
	
	Since $p\ge 2$, it follows that the stochastic convolution
	$\tilde{p}*v\dot F$ is well defined. Moreover, we can find
	a sequence $v^1,v^2,\ldots$ of Walsh-integrable random fields
	such that $\tilde{p}*v^n\dot F$ converges to $\tilde{p}*v\dot F$ in
	$\bm{L}^2_\beta$. In particular, there exists a Lebesgue-null
	set $N_0\subset\R_+$ such that for all $t\not\in N_0$,
	\begin{equation}\label{eq:mild:temperate:approx}
		\lim_{n\to\infty}
		\sup_{x\in\R^d} \E\left( |v^n_t(x)-v_t(x)|^2\right)=0.
	\end{equation}
	Therefore, we may apply Fatou's lemma to the [pseudo-]
	norm $\mathcal{N}_{\beta,p}$ and deduce the following:
	\begin{equation}\begin{split}
		\mathcal{N}_{\beta,p}\left( \tilde{p}*v\dot F\right)
			&\le \liminf_{n\to\infty}\mathcal{N}_{\beta,p}
			\left( \tilde{p}*v^n\dot F\right)\\
		&\le z_p\liminf_{n\to\infty} \mathcal{N}_{\beta,p}\left(
			v^n\right)\sqrt{(\bar R_\beta f)(0)},
	\end{split}\end{equation}
	because we can apply Lemma \ref{lem:Norm:sigma:N:p} to each
	Walsh-integrable process $v^n$ [with $\sigma(u):=u$]. Among other
	things, this proves that $\tilde{p}*v\dot F\in\bm{L}^p_\beta$,
	and
	\begin{equation}
		\mathcal{N}_{\beta,p}\left( \tilde{p}*v\dot F\right)
		\le z_p \mathcal{N}_{\beta,p}(v) \sqrt{2(\bar R_\beta f)(0)}.
	\end{equation}
	The first inequality of the lemma follows from the above
	and Minkowski's inequality, since $\sigma(v_t(x))\le{\rm C}_\sigma
	+{\rm D}_\sigma |v_t(x)|$ pointwise.
	
	Because we also have that $\mathcal{N}_{\beta,p}(w)<\infty$,
	we can deduce the second inequality from
	Lemma \ref{lem:Norm:sigma:N:p} as well.
\end{proof}

We will not need the entire strength of the preceding lemma itself.
We have stated it for two reasons: First of all,
it suggests that our extension of the stochastic
convolution has good continuity properties,
viewed as elements of the Banach spaces $\bm{L}^p_\beta$;
and also, this is a simple setting in which an
approximation method introduced via
the proof of Lemma \ref{lem:SConvolution:N:p}. We will
use that method later on.

\section{An Upper Bound for Growth}

The goal of this section is to establish
Theorem \ref{th:interm:temperate}.

\begin{proof}[Proof of Theorem \ref{th:interm:temperate}]
	Let $\{\psi^n\}_{n=1}^\infty$ be a weak mollifier, and
	let $\bm{u}^{(n)}:=\{u^{(n)}_t(x)\}_{t>0,x\in\R^d}$ be the
	solution to \eqref{heat:temperate} with initial data $\psi^n*u_0$.
	By the Minkowski inequality, the following is valid for every
	integer $n\ge 1$:
	\begin{equation}
		\mathcal{N}_{\beta,p}\left(\bm{u}^{(n)}\right)
		\le \mathcal{N}_{\beta,p}\left( P_\bullet u_0\right) +
		\mathcal{N}_{\beta,p}\left( \tilde{p}*\left(
		\sigma\circ u^{(n)}\right)\dot F \right).
	\end{equation}
	The first term on the right is estimated easily, thanks to
	Minkowski's inequality, as follows:
	\begin{equation}\begin{split}
		\mathcal{N}_{\beta,p}\left( P_\bullet u_0\right)
			&=\left(\int_0^\infty \e^{-\beta t}\sup_{x\in\R^d}
			\left| (P_tu_0)(x) \right|^2\,\d t \right)^{1/2}\\
		&\le \int_0^\infty \e^{-\beta t/2}\sup_{x\in\R^d}
			\left| (P_tu_0)(x) \right|\,\d t\\
		&\le \frac{2}{(2\pi)^d}\int_{\R^d}\frac{\left| 
			\hat u_0(\xi)\right|}{
			\beta+2\Re\Psi(\xi)}\,\d\xi;
	\end{split}\end{equation}
	see \eqref{eq:N:beta:p:Ptu0}. Therefore, Lemma
	\ref{lem:SConvolution:N:p} implies that
	\begin{equation}\label{eq:prepre:prepre}\begin{split}
		\mathcal{N}_{\beta,p}\left(\bm{u}^{(n)}\right)
		&\le\frac{2}{(2\pi)^d}\int_{\R^d}\frac{\left| 
			\hat u_0(\xi)\right|}{
			\beta+2\Re\Psi(\xi)}\,\d\xi\\
		&\qquad+ z_p\left(
			\frac{{\rm C}_\sigma}{\beta^{1/2}}+
			{\rm D}_\sigma\mathcal{N}_{\beta,p}\left( \bm{u}^{(n)}\right)
			\right)\sqrt{2(\bar R_\beta f)(0)}.
	\end{split}\end{equation}
	Let us assume, for the time being,
	that we could prove that $\mathcal{N}_{\beta,p}(\bm{u}^{(n)})$ is finite---that
	is, $\bm{u}^{(n)}$ is in the Banach space $\bm{L}^p_\beta$---provided
	that\footnote{\eqref{cond:beta} holds for all $\beta$ sufficiently large
	because thanks to \eqref{cond:1}, $\lim_{\beta\to\infty}
	(\bar R_\beta f)(0)=0$.}
	\begin{equation}\label{cond:beta}
		z_p{\rm D}_\sigma\sqrt{2(\bar R_\beta f)(0)}<1.
	\end{equation}
	Then we could rearrange \eqref{eq:prepre:prepre} and find that
	\begin{equation}
		\mathcal{N}_{\beta,p}\left(\bm{u}^{(n)}\right)\le
		c \left(\frac{2}{(2\pi)^d}\int_{\R^d}\frac{|\hat u_0(\xi)|}{
		\beta+2\Re\Psi(\xi)}\,\d\xi + \frac{z_p{\rm C}_\sigma}{\beta^{1/2}}\right),
	\end{equation}
	where
	\begin{equation}
		c := \frac{1}{1-z_p{\rm D}_\sigma\sqrt{2(\bar R_\beta f)(0)}}.
	\end{equation}
	And therefore, the proof of Lemma \ref{lem:SConvolution:N:p} implies that
	\begin{equation}\begin{split}
		\mathcal{N}_{\beta,p}\left(\bm{u}\right) &\le
			\liminf_{n\to\infty}
			\mathcal{N}_{\beta,p}(\bm{u}^{(n)})\\
		&\le c \left(\frac{2}{(2\pi)^d}\int_{\R^d}\frac{|\hat u_0(\xi)|}{
			\beta+2\Re\Psi(\xi)}\,\d\xi + \frac{z_p{\rm C}_\sigma}{\beta^{1/2}}\right).
	\end{split}\end{equation}
	In other words, we have shown the following: If we could prove
	that $\bm{u}^{(n)}\in \bm{L}^p_\beta$ for every $n$, then
	in fact $\{\bm{u}^{(n)}\}_{n=1}^\infty$ is bounded in $\bm{L}^p_\beta$,
	whence $\bm{u}\in\bm{L}^p_\beta$ and Theorem \ref{th:interm:temperate}
	follows. 
	
	Now we proceed as follows:
	Let
	\begin{equation}
		u^{(n,0)}_t(x):=(\psi^n*u_0)(x)
		\qquad\text{for all $t\ge 0$ and $x\in\R^d$,}
	\end{equation}
	and iteratively define
	\begin{equation}
		u^{(n,k)}_t(x) := (P_t u_0)(x) + \int_0^t\int_{\R^d}
		p_{t-s}(y-x)\sigma\left( u^{(n,k-1)}_s(y)\right)
		\,F(\d y\,\d s).
	\end{equation}
	By the Minkowski inequality and \eqref{eq:N:beta:p:Ptu0},
	\begin{equation}\label{eq:prepre:n:zero}\begin{split}
		\mathcal{N}_{\beta,p}\left( \bm{u}^{(n,0)}\right)
			&=\left( \int_0^\infty \e^{-\beta t}\sup_{x\in\R^d}
			\left| \left(P_t (\psi^n*u_0)\right)(x)\right|^2\,\d t\right)^{1/2}\\
		&\le \frac{2}{(2\pi)^d}
			\int_{\R^d}\frac{|\hat\psi_n(\xi)|\cdot |\hat u_0(\xi)|}{
			\beta+2\Re\Psi(\xi)}\,\d\xi,
	\end{split}\end{equation}
	which is finite; it is in fact bounded uniformly in $n$ since 
	$\sup_{n\ge 1}\|\psi_n\|_{L^1(\R^d)}$ is finite by the very
	definition of weak mollifiers. Now the argument that led
	to \eqref{eq:prepre:prepre} leads also to the following:
	\begin{equation}\begin{split}
		\mathcal{N}_{\beta,p}\left(\bm{u}^{(n,k+1)}\right)
		&\le\frac{2}{(2\pi)^d}\int_{\R^d}\frac{\left| 
			\hat u_0(\xi)\right|}{
			\beta+2\Re\Psi(\xi)}\,\d\xi\\
		&\qquad+ z_p\left(
			\frac{{\rm C}_\sigma}{\beta^{1/2}}+
			{\rm D}_\sigma\mathcal{N}_{\beta,p}\left( \bm{u}^{(n,k)}\right)
			\right)\sqrt{2(\bar R_\beta f)(0)}.
	\end{split}\end{equation}
	This, \eqref{eq:prepre:n:zero}, and induction together prove that
	if \eqref{cond:beta} is in effect, then
	\begin{equation}
		\sup_{k\ge 1}\mathcal{N}_{\beta,p}\left(\bm{u}^{(n,k)}\right)<\infty.
	\end{equation}
	By the Borel--Cantelli lemma, and owing to Theorem \ref{th:existence}
	and its proof,
	\begin{equation}
		\lim_{k\to\infty} \sup_{x\in\R^d}\E\left(
		\left| u_t^{(n,k)}(x)-u^{(n)}_t(x) \right|^2\right)=0
		\qquad\text{for all $t> 0$}.
	\end{equation}
	[In fact, this is essentially
	\eqref{eq:mild:temperate:approx}, but we have
	noted further that ``almost all $t$'' can be replaced by ``all $t$''
	in the present setting, since $\bm{u}^{(n)}$ is a mild solution
	to \eqref{heat:temperate}.]
	Therefore, Fatou's lemma proves that 
	\begin{equation}
		\mathcal{N}_{\beta,p}\left(\bm{u}^{(n)}\right)
		\le \liminf_{k\to\infty}\mathcal{N}_{\beta,p}\left(
		\bm{u}^{(n,k)}\right)<\infty.
	\end{equation}
	This establishes Theorem \ref{th:interm:temperate}.
\end{proof}

%% file: FoonKhosh.bbl
\def\cprime{$'$} \def\polhk#1{\setbox0=\hbox{#1}{\ooalign{\hidewidth
  \lower1.5ex\hbox{`}\hidewidth\crcr\unhbox0}}}
  \def\polhk#1{\setbox0=\hbox{#1}{\ooalign{\hidewidth
  \lower1.5ex\hbox{`}\hidewidth\crcr\unhbox0}}} \def\cprime{$'$}
\providecommand{\bysame}{\leavevmode\hbox to3em{\hrulefill}\thinspace}
\providecommand{\MR}{\relax\ifhmode\unskip\space\fi MR }
\providecommand{\MRhref}[2]{%
  \href{http://www.ams.org/mathscinet-getitem?mr=#1}{#2}
}
\providecommand{\href}[2]{#2}
\begin{thebibliography}{vdHKM06}

\bibitem[AM95]{AssingManthey}
S.~Assing and R.~Manthey, \emph{The behavior of solutions of stochastic
  differential inequalities}, Probab. Theory Related Fields \textbf{103}
  (1995), no.~4, 493--514. \MR{MR1360202 (97c:60147)}

\bibitem[BC86]{BassCranston}
R.~F. Bass and M.~Cranston, \emph{The {M}alliavin calculus for pure jump
  processes and applications to local time}, Ann. Probab. \textbf{14} (1986),
  no.~2, 490--532. \MR{MR832021 (88b:60113)}

\bibitem[BC95]{BertiniCancrini}
Lorenzo Bertini and Nicoletta Cancrini, \emph{The stochastic heat equation:
  {F}eynman-{K}ac formula and intermittence}, J. Statist. Phys. \textbf{78}
  (1995), no.~5-6, 1377--1401. \MR{MR1316109 (95j:60093)}

\bibitem[BC98]{BertiniCancrini:98}
\bysame, \emph{The two-dimensional stochastic heat equation: renormalizing a
  multiplicative noise}, J. Phys. A \textbf{31} (1998), no.~2, 615--622.
  \MR{MR1629198 (99c:82051)}

\bibitem[BCJL94]{BertiniCancriniJona}
L.~Bertini, N.~Cancrini, and G.~Jona-Lasinio, \emph{The stochastic {B}urgers
  equation}, Comm. Math. Phys. \textbf{165} (1994), no.~2, 211--232.
  \MR{MR1301846 (96d:60092)}

\bibitem[Ber96]{Bertoin:book}
Jean Bertoin, \emph{L\'evy {P}rocesses}, Cambridge Tracts in Mathematics, vol.
  121, Cambridge University Press, Cambridge, 1996. \MR{MR1406564 (98e:60117)}

\bibitem[Ber99]{Bertoin:99}
\bysame, \emph{Subordinators: {E}xamples and {A}pplications}, Lectures on
  {P}robability {T}heory and {S}tatistics ({S}aint-{F}lour, 1997), Lecture
  Notes in Math., vol. 1717, Springer, Berlin, 1999, pp.~1--91. \MR{MR1746300
  (2002a:60001)}

\bibitem[BG61]{BG:61}
R.~M. Blumenthal and R.~K. Getoor, \emph{Sample functions of stochastic
  processes with stationary independent increments}, J. Math. Mech. \textbf{10}
  (1961), 493--516. \MR{MR0123362 (23 \#A689)}

\bibitem[BG68]{BG}
\bysame, \emph{Markov {P}rocesses and {P}otential {T}heory}, Pure and Applied
  Mathematics, Vol. 29, Academic Press, New York, 1968. \MR{MR0264757 (41
  \#9348)}

\bibitem[BG97]{BertiniGiacomin:97}
Lorenzo Bertini and Giambattista Giacomin, \emph{Stochastic {B}urgers and {KPZ}
  equations from particle systems}, Comm. Math. Phys. \textbf{183} (1997),
  no.~3, 571--607. \MR{MR1462228 (99e:60212)}

\bibitem[BG99]{BertiniGiacomin:99}
\bysame, \emph{On the long-time behavior of the stochastic heat equation},
  Probab. Theory Related Fields \textbf{114} (1999), no.~3, 279--289.
  \MR{MR1705123 (2000k:60123)}

\bibitem[Boc33]{Bochner:33}
S.~Bochner, \emph{Monotone {F}unktionen, {S}tieldjessche {I}ntegrale, und
  harmonische {A}nalyse}, Math. Ann. \textbf{108} (1933), 378--410.

\bibitem[Boc55]{Bochner}
Salomon Bochner, \emph{Harmonic {A}nalysis and the {T}heory of {P}robability},
  University of California Press, Berkeley and Los Angeles, 1955. \MR{MR0072370
  (17,273d)}

\bibitem[BR59]{BlumRosenblatt}
J.~R. Blum and Murray Rosenblatt, \emph{On the structure of infinitely
  divisible distributions}, Pacific J. Math. \textbf{9} (1959), 1--7.
  \MR{MR0105729 (21 \#4465)}

\bibitem[CD60]{ChoquetDeny}
Gustave Choquet and Jacques Deny, \emph{Sur l'\'equation de convolution {$\mu
  =\mu \ast \sigma $}}, C. R. Acad. Sci. Paris \textbf{250} (1960), 799--801.
  \MR{MR0119041 (22 \#9808)}

\bibitem[CD08]{ConusDalang}
Daniel Conus and Robert~C. Dalang, \emph{The non-linear stochastic wave
  equation in high dimensions}, Electron. J. Probab. \textbf{13} (2008), no.
  22, 629--670. \MR{MR2399293 (2009c:60170)}

\bibitem[CK91]{CarlenKree}
Eric Carlen and Paul Kr{\'e}e, \emph{{$L\sp p$} estimates on iterated
  stochastic integrals}, Ann. Probab. \textbf{19} (1991), no.~1, 354--368.
  \MR{MR1085341 (92e:60085)}

\bibitem[CKM01]{CarmonaKoralovMolchanov}
Rene Carmona, Leonid Koralov, and Stanislav Molchanov, \emph{Asymptotics for
  the almost sure {L}yapunov exponent for the solution of the parabolic
  {A}nderson problem}, Random Oper. Stochastic Equations \textbf{9} (2001),
  no.~1, 77--86. \MR{MR1910468 (2003g:60104)}

\bibitem[CM94]{CarmonaMolchanov:94}
Ren{\'e}~A. Carmona and S.~A. Molchanov, \emph{Parabolic {A}nderson {P}roblem
  and {I}ntermittency}, Mem. Amer. Math. Soc. \textbf{108} (1994), no.~518,
  viii+125. \MR{MR1185878 (94h:35080)}

\bibitem[CM95]{CarmonaMolchanov:95}
R.~A. Carmona and S.~A. Molchanov, \emph{Stationary parabolic {A}nderson model
  and intermittency}, Probab. Theory Related Fields \textbf{102} (1995), no.~4,
  433--453. \MR{MR1346261 (96m:60142)}

\bibitem[CM07a]{CranstonMolchanov:07b}
M.~Cranston and S.~Molchanov, \emph{On phase transitions and limit theorems for
  homopolymers}, Probability and {M}athematical {P}hysics, CRM Proc. Lecture
  Notes, vol.~42, Amer. Math. Soc., Providence, RI, 2007, pp.~97--112.
  \MR{MR2352264 (2009a:60119)}

\bibitem[CM07b]{CranstonMolchanov:07a}
\bysame, \emph{Quenched to annealed transition in the parabolic {A}nderson
  problem}, Probab. Theory Related Fields \textbf{138} (2007), no.~1-2,
  177--193. \MR{MR2288068 (2008h:60066)}

\bibitem[CMS02]{CranstonMountfordShiga:02}
M.~Cranston, T.~S. Mountford, and T.~Shiga, \emph{Lyapunov exponents for the
  parabolic {A}nderson model}, Acta Math. Univ. Comenian. (N.S.) \textbf{71}
  (2002), no.~2, 163--188. \MR{MR1980378 (2004d:60162)}

\bibitem[CMS05]{CranstonMountfordShiga:05}
\bysame, \emph{Lyapunov exponent for the parabolic {A}nderson model with
  {L}\'evy noise}, Probab. Theory Related Fields \textbf{132} (2005), no.~3,
  321--355. \MR{MR2197105 (2007h:60053)}

\bibitem[CSY03]{CometsShigaYoshida2}
Francis Comets, Tokuzo Shiga, and Nobuo Yoshida, \emph{Directed polymers in a
  random environment: path localization and strong disorder}, Bernoulli
  \textbf{9} (2003), no.~4, 705--723. \MR{MR1996276 (2004f:60210)}

\bibitem[CSY04]{CometsShigaYoshida1}
\bysame, \emph{Probabilistic analysis of directed polymers in a random
  environment: a review}, Stochastic analysis on large scale interacting
  systems, Adv. Stud. Pure Math., vol.~39, Math. Soc. Japan, Tokyo, 2004,
  pp.~115--142. \MR{MR2073332 (2005d:82050)}

\bibitem[CV98]{CarmonaViens}
Ren{\'e}~A. Carmona and Frederi~G. Viens, \emph{Almost-sure exponential
  behavior of a stochastic {A}nderson model with continuous space parameter},
  Stochastics Stochastics Rep. \textbf{62} (1998), no.~3-4, 251--273.
  \MR{MR1615092 (99c:60126)}

\bibitem[CY05]{CometsYoshida}
Francis Comets and Nobuo Yoshida, \emph{Brownian directed polymers in random
  environment}, Comm. Math. Phys. \textbf{254} (2005), no.~2, 257--287.
  \MR{MR2117626 (2005m:60242)}

\bibitem[Dal99]{Dalang}
Robert~C. Dalang, \emph{Extending the martingale measure stochastic integral
  with applications to spatially homogeneous s.p.d.e.'s}, Electron. J. Probab.
  \textbf{4} (1999), no.\ 6, 29 pp.\ (electronic). \MR{MR1684157 (2000b:60132)}

\bibitem[Dav76]{Davis}
Burgess Davis, \emph{On the {$L\sp{p}$} norms of stochastic integrals and other
  martingales}, Duke Math. J. \textbf{43} (1976), no.~4, 697--704.
  \MR{MR0418219 (54 \#6260)}

\bibitem[DF98]{DalangFrangos}
Robert~C. Dalang and N.~E. Frangos, \emph{The stochastic wave equation in two
  spatial dimensions}, Ann. Probab. \textbf{26} (1998), no.~1, 187--212.
  \MR{MR1617046 (99c:60127)}

\bibitem[DIP89]{DawsonIscoePerkins}
D.~A. Dawson, I.~Iscoe, and E.~A. Perkins, \emph{Super-{B}rownian motion: path
  properties and hitting probabilities}, Probab. Theory Related Fields
  \textbf{83} (1989), no.~1-2, 135--205. \MR{MR1012498 (90k:60073)}

\bibitem[DL04a]{DalangLeveque:04a}
Robert~C. Dalang and Olivier L{\'e}v{\^e}que, \emph{Second-order hyperbolic
  {S}.{P}.{D}.{E}.'s driven by boundary noises}, Seminar on {S}tochastic
  {A}nalysis, {R}andom {F}ields and {A}pplications {IV}, Progr. Probab.,
  vol.~58, Birkh\"auser, Basel, 2004, pp.~83--93. \MR{MR2096282}

\bibitem[DL04b]{DalangLeveque:04b}
\bysame, \emph{Second-order linear hyperbolic {SPDE}s driven by isotropic
  {G}aussian noise on a sphere}, Ann. Probab. \textbf{32} (2004), no.~1B,
  1068--1099. \MR{MR2044674 (2005h:60182)}

\bibitem[DL06]{DalangLeveque:06}
\bysame, \emph{Second-order hyperbolic {S}.{P}.{D}.{E}.'s driven by homogeneous
  {G}aussian noise on a hyperplane}, Trans. Amer. Math. Soc. \textbf{358}
  (2006), no.~5, 2123--2159 (electronic). \MR{MR2197451 (2006k:60112)}

\bibitem[DM03]{DalangMueller}
Robert~C. Dalang and Carl Mueller, \emph{Some non-linear {S}.{P}.{D}.{E}.'s
  that are second order in time}, Electron. J. Probab. \textbf{8} (2003), no.
  1, 21 pp. (electronic). \MR{MR1961163 (2004a:60118)}

\bibitem[DMP93]{DonatiMartinPardoux}
C.~Donati-Martin and {\'E}.~Pardoux, \emph{White noise driven {SPDE}s with
  reflection}, Probab. Theory Related Fields \textbf{95} (1993), no.~1, 1--24.
  \MR{MR1207304 (94f:60083)}

\bibitem[DMT08]{DalangMuellerTribe}
Robert~C. Dalang, Carl Mueller, and Roger Tribe, \emph{A {F}eynman-{K}ac-type
  formula for the deterministic and stochastic wave equations and other
  {P}.{D}.{E}.'s}, Trans. Amer. Math. Soc. \textbf{360} (2008), no.~9,
  4681--4703. \MR{MR2403701 (2009e:60146)}

\bibitem[DPZ92]{DaPratoZabczyk}
Giuseppe Da~Prato and Jerzy Zabczyk, \emph{Stochastic {E}quations in {I}nfinite
  {D}imensions}, Encyclopedia of Mathematics and its Applications, vol.~44,
  Cambridge University Press, Cambridge, 1992. \MR{MR1207136 (95g:60073)}

\bibitem[DSS09]{DalangSanzSole}
Robert~C. Dalang and Marta Sanz-Sol{\'e}, \emph{H\"older-{S}obolev regularity
  of the solution to the stochastic wave equation in dimension three}, Mem.
  Amer. Math. Soc. \textbf{199} (2009), no.~931, vi+70. \MR{MR2512755}

\bibitem[Dyn84]{Dynkin}
E.~B. Dynkin, \emph{Polynomials of the occupation field and related random
  fields}, J. Funct. Anal. \textbf{58} (1984), no.~1, 20--52. \MR{MR756768
  (86h:60085b)}

\bibitem[EFK09]{EFK}
Nathalie Eisenbaum, Mohammud Foondun, and Davar Khoshnevisan, \emph{Dynkin's
  isomorphism and the stochastic heat equation}, 2009, preprint.

\bibitem[FK09]{FK}
Mohammud Foondun and Davar Khoshnevisan, \emph{Intermittence and nonlinear
  parabolic stochastic partial differential equations}, Electron. J. Probab.
  \textbf{14} (2009), no. 21, 548--568. \MR{MR2480553}

\bibitem[FKN09]{FKN}
Mohammud Foondun, Davar Khoshnevisan, and Eulalia Nualart, \emph{A local time
  correspondence for stochastic partial differential equations}, Trans. Amer.
  Math. Soc. (2009+), 40 pages, to appear.

\bibitem[F{\=O}T94]{FOT}
Masatoshi Fukushima, Y{\=o}ichi {\=O}shima, and Masayoshi Takeda,
  \emph{Dirichlet {F}orms and {S}ymmetric {M}arkov {P}rocesses}, de Gruyter
  Studies in Mathematics, vol.~19, Walter de Gruyter \& Co., Berlin, 1994.
  \MR{MR1303354 (96f:60126)}

\bibitem[FV63]{FiszVaradara}
M.~Fisz and V.~S. Varadarajan, \emph{A condition for absolute continuity of
  infinitely divisible distribution functions}, Z. Wahrscheinlichkeitstheorie
  und Verw. Gebiete \textbf{1} (1962/1963), 335--339. \MR{MR0149521 (26
  \#7007)}

\bibitem[FV06]{FlorescuViens}
Ionu{\c{t}} Florescu and Frederi Viens, \emph{Sharp estimation of the
  almost-sure {L}yapunov exponent for the {A}nderson model in continuous
  space}, Probab. Theory Related Fields \textbf{135} (2006), no.~4, 603--644.
  \MR{MR2240702 (2008g:60189)}

\bibitem[GdH06]{GartnerDenHollander}
J.~G{\"a}rtner and F.~den Hollander, \emph{Intermittency in a catalytic random
  medium}, Ann. Probab. \textbf{34} (2006), no.~6, 2219--2287. \MR{MR2294981
  (2008e:60200)}

\bibitem[Gir60]{Girsanov:60}
I.~V. Girsanov, \emph{Strong {F}eller processes. {I}. {G}eneral properties},
  Teor. Verojatnost. i Primenen. \textbf{5} (1960), 7--28. \MR{MR0137151 (25
  \#607)}

\bibitem[GK05]{GartnerKonig}
J{\"u}rgen G{\"a}rtner and Wolfgang K{\"o}nig, \emph{The parabolic {A}nderson
  model}, Interacting {S}tochastic {S}ystems, Springer, Berlin, 2005,
  pp.~153--179. \MR{MR2118574 (2005k:82042)}

\bibitem[GV77]{GV}
I.~M. Gel{\cprime}fand and N.~Ya. Vilenkin, \emph{Generalized {F}unctions.
  {V}ol. 4}, Academic Press [Harcourt Brace Jovanovich Publishers], New York,
  1964 [1977], Applications of harmonic analysis, Translated from the Russian
  by Amiel Feinstein. \MR{MR0435834 (55 \#8786d)}

\bibitem[Haw79]{Hawkes:PLP}
John Hawkes, \emph{Potential theory of {L}\'evy processes}, Proc. London Math.
  Soc. (3) \textbf{38} (1979), no.~2, 335--352. \MR{MR531166 (80g:60077)}

\bibitem[Haw84]{Hawkes}
\bysame, \emph{Transition and resolvent densities for {L}\'evy processes},
  Unpublished manuscript (1984).

\bibitem[Haw86]{Hawkes:86}
J.~Hawkes, \emph{Local times as stationary processes}, From {L}ocal {T}imes to
  {G}lobal {G}eometry, {C}ontrol and {P}hysics ({C}oventry, 1984/85), Pitman
  Res. Notes Math. Ser., vol. 150, Longman Sci. Tech., Harlow, 1986,
  pp.~111--120. \MR{MR894527 (88g:60189)}

\bibitem[Her11]{Herglotz}
G.~Herglotz, \emph{{\"{U}}ber {P}otenzreihen mit positivem reellen {T}eil im
  {E}inheitskreis}, Acta Math. \textbf{63} (1911), 501--511.

\bibitem[HN09]{HuNualart}
Yaozhong Hu and David Nualart, \emph{Stochastic heat equation driven by
  fractional noise and local time}, Probab. Theory Related Fields \textbf{143}
  (2009), no.~1-2, 285--328. \MR{MR2449130}

\bibitem[HP89]{HaussmannPardoux}
U.~G. Haussmann and {\'E}.~Pardoux, \emph{Stochastic variational inequalities
  of parabolic type}, Appl. Math. Optim. \textbf{20} (1989), no.~2, 163--192.
  \MR{MR998402 (90k:60119)}

\bibitem[HS55]{HewittSavage}
Edwin Hewitt and Leonard~J. Savage, \emph{Symmetric measures on {C}artesian
  products}, Trans. Amer. Math. Soc. \textbf{80} (1955), 470--501.
  \MR{MR0076206 (17,863g)}

\bibitem[HW42]{HartmanWintner}
Philip Hartman and Aurel Wintner, \emph{On the infinitesimal generators of
  integral convolutions}, Amer. J. Math. \textbf{64} (1942), 273--298.
  \MR{MR0006635 (4,18a)}

\bibitem[Jac05]{Jacob}
N.~Jacob, \emph{Pseudo {D}ifferential {O}perators and {M}arkov {P}rocesses.
  {V}ol. {III}}, Imperial College Press, London, 2005, Markov processes and
  applications. \MR{MR2158336 (2006i:60001)}

\bibitem[Kar87]{Kardar}
Mehran Kardar, \emph{Replica {B}ethe ansatz studies of two-dimensional
  interfaces with quenched random impurities}, Nuclear Phys. B \textbf{290}
  (1987), no.~4, 582--602. \MR{MR922846 (89f:82058)}

\bibitem[KLMS09]{KonigLacoinMorters}
Wolfgang K{\"o}nig, Hubert Lacoin, Peter M{\"o}rters, and Nadia Sidorova,
  \emph{A two cities theorem for the parabolic {A}nderson model}, Ann. Probab.
  \textbf{37} (2009), no.~1, 347--392. \MR{MR2489168}

\bibitem[Kot92]{Kotelenez}
Peter Kotelenez, \emph{Comparison methods for a class of function valued
  stochastic partial differential equations}, Probab. Theory Related Fields
  \textbf{93} (1992), no.~1, 1--19. \MR{MR1172936 (93i:60116)}

\bibitem[KPZ86]{KPZ}
Mehran Kardar, Giorgio Parisi, and Yi-Cheng Zhang, \emph{Dynamical scaling of
  growing surfaces}, Phys. Rev. Lett. \textbf{56} (1986), no.~9, 889--892.

\bibitem[KS91]{KrugSpohn}
H.~Krug and H.~Spohn, \emph{Kinetic roughening of growing surfaces}, Solids Far
  From Equilibrium: Growth, Morphology, and Defects (Claude Godr\'eche, ed.),
  Cambridge University Press, Cambridge, 1991, pp.~412--525.

\bibitem[KX09]{KX:HAAL}
Davar Khoshnevisan and Yimin Xiao, \emph{Harmonic analysis of additive l\'evy
  processes}, Probab. Theory Related Fields \textbf{145} (2009), 459--515.

\bibitem[Kyp06]{Kyprianou}
Andreas~E. Kyprianou, \emph{Introductory {L}ectures on {F}luctuations of
  {L}\'evy {P}rocesses with {A}pplications}, Universitext, Springer-Verlag,
  Berlin, 2006. \MR{MR2250061 (2008a:60003)}

\bibitem[L{\'e}v37]{Levy}
Paul L{\'e}vy, \emph{Th\'eorie de l'{A}ddition des {V}ariables {A}l\'eatoires},
  Gauthier--Villars, 1937.

\bibitem[LL63]{LiebLiniger}
Elliott~H. Lieb and Werner Liniger, \emph{Exact analysis of an interacting
  {B}ose gas. {I}. {T}he general solution and the ground state}, Phys. Rev. (2)
  \textbf{130} (1963), 1605--1616. \MR{MR0156630 (27 \#6551)}

\bibitem[Lun95]{Lunardi}
Alessandra Lunardi, \emph{Analytic {S}emigroups and {O}ptimal {R}egularity in
  {P}arabolic {P}roblems}, Progress in Nonlinear Differential Equations and
  their Applications, 16, Birkh\"auser Verlag, Basel, 1995. \MR{MR1329547
  (96e:47039)}

\bibitem[Man86]{Manthey:86}
Ralf Manthey, \emph{Existence and uniqueness of a solution of a
  reaction-diffusion equation with polynomial nonlinearity and white noise
  disturbance}, Math. Nachr. \textbf{125} (1986), 121--133. \MR{MR847354
  (87j:60092)}

\bibitem[Mat95]{Mattila}
Pertti Mattila, \emph{Geometry of {S}ets and {M}easures in {E}uclidean
  {S}paces}, Cambridge Studies in Advanced Mathematics, vol.~44, Cambridge
  University Press, Cambridge, 1995, Fractals and rectifiability. \MR{MR1333890
  (96h:28006)}

\bibitem[McK05]{McKean}
Henry~P. McKean, \emph{Stochastic {I}ntegrals}, AMS Chelsea Publishing,
  Providence, RI, 2005, Reprint of the 1969 edition, with errata. \MR{MR2169626
  (2006d:60003)}

\bibitem[MHKZ89]{MHKZ}
Ernesto Medina, Terence Hwa, Mehran Kardar, and Yi-Chen Zhang, \emph{Burgers
  equation with correlated noise: Renormalization-group analysis and
  applications to directed polymers and interface growth}, Phys.\ Rev.\ A
  \textbf{39} (1989), no.~6, 3053--3075.

\bibitem[Mil73]{Millar}
P.~W. Millar, \emph{Radial processes}, Ann. Probability \textbf{1} (1973),
  613--626. \MR{MR0353464 (50 \#5947)}

\bibitem[Mol91]{Molchanov}
Stanislav~A. Molchanov, \emph{Ideas in the {T}heory of {R}andom {M}edia}, Acta
  Appl. Math. \textbf{22} (1991), no.~2-3, 139--282. \MR{MR1111743 (92m:82067)}

\bibitem[MP03]{MytnikPerkins}
Leonid Mytnik and Edwin Perkins, \emph{Regularity and irregularity of
  {$(1+\beta)$}-stable super-{B}rownian motion}, Ann. Probab. \textbf{31}
  (2003), no.~3, 1413--1440. \MR{MR1989438 (2004b:60130)}

\bibitem[MR06]{MR}
Michael~B. Marcus and Jay Rosen, \emph{Markov {P}rocesses, {G}aussian
  {P}rocesses, and {L}ocal {T}imes}, Cambridge Studies in Advanced Mathematics,
  vol. 100, Cambridge University Press, Cambridge, 2006. \MR{MR2250510
  (2008b:60001)}

\bibitem[MS91]{MantheyStiewe:90}
Ralf Manthey and Christel Stiewe, \emph{On {V}olterra's population equation
  with diffusion and noise}, Stochastic {P}rocesses and {R}elated {T}opics
  ({G}eorgenthal, 1990), Math. Res., vol.~61, Akademie-Verlag, Berlin, 1991,
  pp.~89--92. \MR{MR1127884}

\bibitem[MS92]{MantheyStiewe:92}
\bysame, \emph{Existence and uniqueness of solutions to {V}olterra's population
  equation with diffusion and noise}, Stochastics Stochastics Rep. \textbf{41}
  (1992), no.~3, 135--161. \MR{MR1275580 (95b:60072)}

\bibitem[Mue91]{Mueller}
Carl Mueller, \emph{On the support of solutions to the heat equation with
  noise}, Stochastics Stochastics Rep. \textbf{37} (1991), no.~4, 225--245.
  \MR{MR1149348 (93e:60122)}

\bibitem[NP92]{NualartPardoux}
D.~Nualart and {\'E}.~Pardoux, \emph{White noise driven quasilinear {SPDE}s
  with reflection}, Probab. Theory Related Fields \textbf{93} (1992), no.~1,
  77--89. \MR{MR1172940 (93h:60093)}

\bibitem[NS06]{NourdinSimon}
Ivan Nourdin and Thomas Simon, \emph{On the absolute continuity of {L}\'evy
  processes with drift}, Ann. Probab. \textbf{34} (2006), no.~3, 1035--1051.
  \MR{MR2243878 (2007j:60073)}

\bibitem[PZ32]{PaleyZygmund}
R.~E. A.~C. Paley and A.~Zygmund, \emph{A note on analytic functions in the
  unit circle}, Proc.\ Camb.\ Phil.\ Soc. \textbf{28} (1932), no.~[Issue] 03,
  266--272.

\bibitem[PZ00]{PeszatZabczyk}
Szymon Peszat and Jerzy Zabczyk, \emph{Nonlinear stochastic wave and heat
  equations}, Probab. Theory Related Fields \textbf{116} (2000), no.~3,
  421--443. \MR{MR1749283 (2001f:60071)}

\bibitem[RY91]{RevuzYor}
Daniel Revuz and Marc Yor, \emph{Continuous {M}artingales and {B}rownian
  {M}otion}, Grundlehren der Mathematischen Wissenschaften [Fundamental
  Principles of Mathematical Sciences], vol. 293, Springer-Verlag, Berlin,
  1991. \MR{MR1083357 (92d:60053)}

\bibitem[Sat99]{Sato}
Ken-iti Sato, \emph{L\'evy {P}rocesses and {I}nfinitely {D}ivisible
  {D}istributions}, Cambridge Studies in Advanced Mathematics, vol.~68,
  Cambridge University Press, Cambridge, 1999, Translated from the 1990
  Japanese original, Revised by the author. \MR{MR1739520 (2003b:60064)}

\bibitem[Sch38a]{Schoenberg:a}
I.~J. Schoenberg, \emph{Metric spaces and completely monotone functions}, Ann.
  of Math. (2) \textbf{39} (1938), no.~4, 811--841. \MR{MR1503439}

\bibitem[Sch38b]{Schoenberg:b}
\bysame, \emph{Metric spaces and positive definite functions}, Trans. Amer.
  Math. Soc. \textbf{44} (1938), no.~3, 522--536. \MR{MR1501980}

\bibitem[Shi94]{Shiga:contrasting}
Tokuzo Shiga, \emph{Two contrasting properties of solutions for one-dimensional
  stochastic partial differential equations}, Canad. J. Math. \textbf{46}
  (1994), no.~2, 415--437. \MR{MR1271224 (95h:60099)}

\bibitem[Shi97]{Shiga}
\bysame, \emph{Exponential decay rate of survival probability in a disastrous
  random environment}, Probab. Theory Related Fields \textbf{108} (1997),
  no.~3, 417--439. \MR{MR1465166 (98f:60212)}

\bibitem[Tuc62]{Tucker:62}
Howard~G. Tucker, \emph{Absolute continuity of infinitely divisible
  distributions}, Pacific J. Math. \textbf{12} (1962), 1125--1129.
  \MR{MR0146868 (26 \#4387)}

\bibitem[Tuc64]{Tucker:64}
\bysame, \emph{On continuous singular infinitely divisible distribution
  functions}, Ann. Math. Statist. \textbf{35} (1964), 330--335. \MR{MR0161362
  (28 \#4569)}

\bibitem[Tuc65]{Tucker:65}
\bysame, \emph{On a necessary and sufficient condition that an infinitely
  divisible distribution be absolutely continuous}, Trans. Amer. Math. Soc.
  \textbf{118} (1965), 316--330. \MR{MR0182061 (31 \#6285)}

\bibitem[vdHKM06]{HofstedKonigMorters}
Remco van~der Hofstad, Wolfgang K{\"o}nig, and Peter M{\"o}rters, \emph{The
  universality classes in the parabolic {A}nderson model}, Comm. Math. Phys.
  \textbf{267} (2006), no.~2, 307--353. \MR{MR2249772 (2007g:82029)}

\bibitem[Wal86]{Walsh}
John~B. Walsh, \emph{An {I}ntroduction to {S}tochastic {P}artial {D}ifferential
  {E}quations}, \'{E}cole d'\'et\'e de {P}robabilit\'es de {S}aint-{F}lour,
  {XIV}---1984, Lecture Notes in Math., vol. 1180, Springer, Berlin, 1986,
  pp.~265--439. \MR{MR876085 (88a:60114)}

\bibitem[Woy98]{Woyczynski}
Wojbor~A. Woyczy{\'n}ski, \emph{Burgers-{KPZ} {T}urbulence}, Lecture Notes in
  Mathematics, vol. 1700, Springer-Verlag, Berlin, 1998, G{\"o}ttingen
  lectures. \MR{MR1732301 (2000j:60077)}

\bibitem[Zab70]{Zab}
J.~Zabczyk, \emph{Sur la th\'eorie semi-classique du potentiel pour les
  processus \`a accroissements ind\'ependants}, Studia Math. \textbf{35}
  (1970), 227--247. \MR{MR0267643 (42 \#2545)}

\bibitem[ZRS90]{AlmightyChance}
Ya.~B. Zeldovich, A.~A. Ruzmaikin, and D.~D. Sokoloff, \emph{The {A}lmighty
  {C}hance}, World Scientific Lecture Notes in Physics, vol.~20, World
  Scientific, Singapore, 1990.

\end{thebibliography}
